\documentclass[11pt]{amsart}
\usepackage{amsmath}
\usepackage{amssymb}
\usepackage{amsthm}
\usepackage{latexsym}
\usepackage{graphicx}
\usepackage{hyperref}
\usepackage{enumerate}
\usepackage[all]{xy}
\usepackage{chngcntr}

\newtheorem{thm}{Theorem}[section]
\newtheorem{cor}[thm]{Corollary}
\newtheorem{lem}[thm]{Lemma}
\newtheorem{prop}[thm]{Proposition}

\newtheorem{thmintro}{Theorem}

\theoremstyle{definition}
\newtheorem{defn}[thm]{Definition}

\newtheorem{ex}[thm]{Example}

\counterwithin*{equation}{section}

\newcommand{\N}{\mathbb N}
\newcommand{\Z}{\mathbb Z}
\newcommand{\Q}{\mathbb Q}
\newcommand{\R}{\mathbb R}
\newcommand{\C}{\mathbb C}
\newcommand{\mf}{\mathfrak}
\newcommand{\mc}{\mathcal}
\newcommand{\mb}{\mathbf}
\newcommand{\mh}{\mathbb}
\def\Irr{{\rm Irr}}
\newcommand{\mr}{\mathrm}
\newcommand{\ind}{\mathrm{ind}}
\newcommand{\enuma}[1]{\begin{enumerate}[\textup{(}a\textup{)}] {#1} \end{enumerate}}

\newcommand{\Res}{\mathrm{Res}}
\newcommand{\af}{\mathrm{aff}}

\newcommand{\der}{\mathrm{der}}

\newcommand{\matje}[4]{\left(\begin{smallmatrix} #1 & #2 \\ 
#3 & #4 \end{smallmatrix}\right)}
\newcommand{\inp}[2]{\langle #1 , #2 \rangle}
\newcommand{\Mod}{\mathrm{Mod}}
\newcommand{\Hom}{\mathrm{Hom}}
\newcommand{\End}{\mathrm{End}}

\newcommand{\un}{\mathrm{un}}
\newcommand{\isom}{\xrightarrow{\sim}}
\newcommand{\triv}{\mathrm{triv}}
\newcommand{\ds}{\displaystyle}

\begin{document}

\title{Affine Hecke algebras and their representations}
\date{\today}
\thanks{Keywords: Hecke algebras, Weyl groups, representation theory, discrete series\\
The author is supported by a NWO Vidi grant "A Hecke algebra approach to the 
local Langlands correspondence" (nr. 639.032.528).}
\subjclass[2010]{20C08}
\maketitle

\begin{center}
{\Large Maarten Solleveld} \\[1mm]
IMAPP, Radboud Universiteit Nijmegen\\
Heyendaalseweg 135, 6525AJ Nijmegen, the Netherlands \\
email: m.solleveld@science.ru.nl 
\end{center}
\vspace{1cm}

\begin{abstract}
This is a survey paper about affine Hecke algebras. We start from scratch and discuss some
algebraic aspects of their representation theory, referring to the literature for proofs.
We aim in particular at the classification of irreducible representations.

Only at the end we establish a new result: a natural bijection between the set of
irreducible representations of an affine Hecke algebra with parameters in $\R_{\geq 1}$,
and the set of irreducible representations of the affine Weyl group underlying the
algebra. This can be regarded as a generalized Springer correspondence with
affine Hecke algebras. 
\end{abstract}
\vspace{1cm}

\tableofcontents

\section*{Introduction}

Affine Hecke algebras typically arise in two ways. Firstly, they are deformations of the group
algebra of a Coxeter system $(W,S)$ of affine type. Namely, keep the braid relations 
of $W$, but replace every quadratic relation $s^2 = 1 \; (s \in S)$ by
\begin{equation}\label{eq:0.1}
(s - q_s)(s + 1) = 0 ,
\end{equation}
where $q_s$ is a parameter in some field (usually we take the field $\C$).
That gives rise to an associative algebra $\mc H (W,q)$.

Secondly, affine Hecke algebras occur in the representation theory of reductive groups $G$
over $p$-adic fields. They can be isomorphic to the algebra of $G$-endo\-mor\-phisms of a
suitable $G$-representation. The classical example is the convolution algebra of compactly 
supported functions on $G$ that are bi-invariant with respect to an Iwahori subgroup. In
this way the representation theory of affine Hecke algebras is related to that of reductive
$p$-adic groups. The interpretation of (affine) Hecke algebras as deformations of group 
algebras links them with quantum groups, knot theory and noncommutative geometry. 
Further, their relation with reductive groups makes them highly relevant in the 
representation theory of such groups and in the local Langlands program.

Both views of affine Hecke algebras build upon simpler objects: the Hecke algebras of finite
Weyl groups. On the one hand such a finite dimensional Hecke algebra is a deformation of 
a group algebra, again with relations \eqref{eq:0.1}. On the other hand it appears naturally
in the representation theory of reductive groups over finite fields. However, there
is a crucial difference between the finite and affine cases: most aspects of finite 
dimensional Hecke algebras are easily described in terms of a Coxeter group, but that is far 
from true for affine Hecke algebras.\\

Of course there exists an extensicve body of literature on affine Hecke algebras and their 
representations, see the references to this paper for a part of it. The theory is in a good
state, and (in our opinion) most of the important questions that one can ask about affine
Hecke algebras have been answered.

Unfortunately, the accessibility of this literature is rather limited. A first difficulty is
that several slightly different algebras are involved, with several presentations, and
sometimes their connection is not clear. Further, a plethora of techniques has been applied
to affine Hecke algebras: algebraic, analytic, geometric or combinatoric to various degrees.
Finally, to the best of our knowledge no textbook treats affine Hecke algebras any further
than their presentations.

With this survey paper we intend to fill a part of that gap. Our aim is an introduction to
affine Hecke algebras and their representations, of which the larger part is readable for 
mathematicians without previous experience with the subject. To ease the presentation,
we will hardly prove anything, and we will provide many examples. As a consequence, our
treatment is almost entirely algebraic - the deep analytic or geometric arguments behind 
some important results are largely suppressed.\\

Let us discuss the contents in a nutshell. In the first section we work out the various 
presentations of (affine) Hecke algebras over $\C$, and we compare them. Section 2 consists
of explicit examples: we look at the most frequent affine Hecke algebras and provide an
overview of their irreducible representations.

The third section is the core of the paper, here we build up the abstract representation
theory of an affine Hecke algebra $\mc H$. This is done in the spirit of Harish-Chandra's
analytic approach to representations of reductive groups: we put the emphasis on parabolic
induction, the discrete series and the large commutative subalgebra of $\mc H$. To make
this work well, we need to assume that the parameters $q_s$ of $\mc H$ lie in $\R_{>0}$.
With these techniques one can divide the set of irreducible representations $\Irr (\mc H)$
into L-packets, like in the local Langlands program. 
To achieve more, it pays off to reduce from affine Hecke algebras to a simpler kind of
algebras called graded Hecke algebras. This is like reducing questions about a Lie group
to its Lie algebra. 

More advanced techniques to classify $\Irr (\mc H)$ are discussed in Section 4. In principle
this achieves a complete classification, but in practice some computations remain to
be done in examples. 

In Section 5 we report on algebro-geometric approaches to affine Hecke algebras and
graded Hecke algebras, largely due to Lusztig. In many cases, these methods yield beautiful
constructions and parametrizations of all irreducible representations. Although we treat
the involved (co)homology theories as a black box, we do provide a comparison between
these constructions and the setup inspired by Harish-Chandra. 

The final section of the paper is quite different from the rest, here we do actually prove
some new results. The topic is the relation between an affine Hecke algebra $\mc H$ with
parameters $q_s \in \R_{\geq 1}$ and its version with parameters $q_s = 1$. The latter 
algebra is of the form $\C [X \rtimes W]$, where $X$ is a lattice and $W$ is the (finite)
Weyl group of a root system $R$ in $X$. We note that $X \rtimes W$ contains the affine
Weyl/Coxeter group $\Z R \rtimes W$. There is a natural way to regard any finite dimensional
$\mc H$-representation $\pi$ as a representation of $\C [W]$, we call that the $W$-type
of $\pi$. 

\begin{thmintro} \label{thm:A} (see Theorem \ref{thm:6.12}) \\
Let $\mc H$ be an affine Hecke algebra with parameters in $\R_{\geq 1}$
(and a mild condition when $R$ has components of type $F_4$).
There exists a natural bijection
\[
\zeta'_{\mc H} : \Irr (\mc H) \longrightarrow \Irr (X \rtimes W)
\]
such that the restriction of $\zeta'_{\mc H}(\pi)$ to $W$ is always a constituent
of the $W$-type of $\pi$.
\end{thmintro}

In the case where all the parameters $q_s$ are equal, Theorem \ref{thm:A} is closely related
to the Kazhdan--Lusztig parametrization of $\Irr (\mc H)$ \cite{KaLu}, and shown already
in \cite{LusCells}.

It is interesting to restrict Theorem \ref{thm:A} to irreducible $X \rtimes W$-representations on 
which $X$ acts trivially (those can clearly be identified with irreducible $W$-represen\-tations).
The inverse image of these under $\zeta'_{\mc H}$ is the set $\Irr_0 (\mc H)$ of irreducible
tempered $\mc H$-representations ``with real central character", see Paragraph \ref{par:Wtypes}.
When $\mc H$ is of geometric origin, $\Irr_0 (\mc H)$ is naturally parametrized by data as in
Lusztig's generalization of the Springer correspondence with intersection homology, 
see Paragraph \ref{par:homology}. In that case the bijection
\[
\zeta'_{\mc H} : \Irr_0 (\mc H) \to \Irr (W)
\]
becomes an instance of the generalized Springer correspondence. Moreover, for such geometric
affine Hecke algebras the whole of $\Irr (\mc H)$ admits a natural parametrization in terms of 
data that are variations on Kazhdan--Lusztig parameters, see Paragraph \ref{par:cusp}. Via this
parametrization, $\zeta'_{\mc H}$ becomes a generalized Springer correspondence for the (extended) 
affine Weyl group $X \rtimes W$.

With that in mind we can regard Theorem \ref{thm:A}, for any eligible $\mc H$, as a
``generalization of the Springer correspondence with affine Hecke algebras". This applies both
to the finite Weyl group $W$ and the (extended) affine Weyl group $X \rtimes W$.

In the final paragraph we use (the proof of) Theorem \ref{thm:A} to derive some properties of
affine Hecke algebras $\mc H$ of type $B_n / C_n$ with three independent positive $q$-parameters. 
When these parameters are generic, we provide an explicit, effective classification of 
$\Irr (\mc H)$.\\

Of course the selection of topics in any survey is to a considerable extent the taste of the author.
To preserve a reasonable size, we felt forced to omit many interesting aspects of affine Hecke
algebras: the Kazhdan--Lusztig basis \cite{KaLu1}, asymptotic Hecke algebras \cite{LusCells,Lus7}, 
unitary representations \cite{BaMo1,Ciu}, the Schwartz and $C^*$-completions of affine Hecke 
algebras \cite{Opd-Sp,DeOp1}, homological algebra \cite{OpSo1,SolHomAHA}, formal degrees of 
representations \cite{OpSo,CKK}, spectral transfer morphisms \cite{Opd3,Opd4} and so on. 
We apologize for these and other omissions and refer the reader to the literature.\\

\textbf{Acknowledgement.}\\
We thank George Lusztig for some useful clarifications of his work and the referee for his
helpful suggestions.

\section{Definitions and first properties}
\label{sec:def}
\subsection{Finite dimensional Hecke algebras} \

Let $(W,S)$ be a finite Coxeter system -- so $W$ is a finite group
generated by a set $S$ of elements of order 2. Moreover, $W$ has a presentation
\[
W = \langle S \mid (s s' )^{m(s,s')} = e \; \forall s,s' \in S \rangle ,
\]
where $m(s,s') \in \Z_{\geq 1}$ is the order of $s s'$ in $W$. The equalities
$s^2 = e \; (s \in S)$ are called the quadratic relations, while $(s s')^{m(s,s')} = e$,
or equivalently
\[
\underbrace{s \, s' \, s \, s' \cdots }_{m(s,s') \text{ terms}} = 
\underbrace{s' \, s \, s' \, s \cdots }_{m(s,s') \text{ terms}}
\]
is known as a braid relation. Examples to keep in mind are
\begin{itemize}
\item $W = S_n, S = \{ (12), (23), \ldots, (n\!-\!1 \, n) \}$ -- type $A_{n-1}$;
\item $W = S_n \ltimes \{ \pm 1 \}^n, S = \{ (12), (23), \ldots, (n\!-\!1 \, n), 
(\mr{id}, (1,\ldots,1,-1)) \}$ -- type $B_n$ or $C_n$. 
\end{itemize}
In the group algebra $\C [W]$ the quadratic relations are equivalent with 
\begin{equation}\label{eq:1.20}
(s+1) (s-1) = 0 \qquad s \in S .
\end{equation}
Now we choose, for every $s \in S$, a complex number $q_s$, such that
\begin{equation}\label{eq:1.1}
q_s = q_{s'} \text{ if } s \text{ and } s' \text{ are conjugate in } W.
\end{equation}
Let $q : S \to \C$ be the function $s \mapsto q_s$. We define a new $\C$-algebra
$\mc H (W,q)$ which has a vector space basis $\{ T_w : w \in W\}$. Here
$T_e$ is the unit element and there are quadratic relations
\[
(T_s + 1)(T_s - q_s) = 0 \qquad s \in S
\]
and braid relations
\begin{equation}\label{eq:1.2}
\underbrace{T_s \, T_{s'} \, T_{s} \cdots }_{m(s,s') \text{ terms}} = 
\underbrace{T_{s'} \, T_{s} \, T_{s'} \cdots }_{m(s,s') \text{ terms}} .
\end{equation}
Equivalent versions of these quadratic relations are
\begin{equation}\label{eq:1.3}
T_s^2 = (q_s - 1) T_s + q_s T_e \quad \text{and} \quad
T_s^{-1} = q_s^{-1} T_s + (q_s^{-1} - 1) T_e 
\end{equation}
(the latter only when $q_s \neq 0$). For $q=1$, the relations \eqref{eq:1.2}
become the defining relations of the Coxeter system $(W,S)$, so $\mc H (W,1) = \C[W]$.
We say that $\mc H (W,q)$ has equal parameters if $q_s = q_{s'}$ for all
$s,s' \in S$.

The condition \eqref{eq:1.1} is necessary and sufficient for the existence of
an associative unital algebra $\mc H(W,q)$ with these properties
\cite[\S 7.1--7.3]{Hum1}. It is known as a generic algebra or a Hecke algebra.
Such algebras appear for instance in the representation theory of reductive 
groups over finite fields \cite{Iwa1,HoLe1,HoLe2}, in knot theory \cite{Jon}
and in combinatorics \cite{HiTh,HST}.

Further $\mc H (W,q)$ has the structure of a symmetric algebra: it carries a
trace $\tau (T_w) = \delta_{w,e}$, an involution $T_w^* = T_{w^{-1}}$ and
a bilinear form $(x,y) = \tau (x^* y)$. These also give rise to interesting
properties \cite{GePf}, which however fall outside the scope of this survey.\\

Without the trace $\tau$, the representation theory of finite dimensional
Hecke algebras is quite easy. To explain this, we consider a more general situation.

Let $G$ be any finite group. By Maschke's theorem the 
group algebra $\C [G]$ is semisimple. Let 
$\{T_g : g \in G\}$ be its canonical basis, and 
$\mb k = \mh C [x_1 ,\ldots, x_r ]$ a polynomial ring
over $\mh C$. Let $A$ be a $\mb k$-algebra whose 
underlying $\mb k$-module is $\mb k [G]$ and whose 
multiplication is defined by 
\begin{equation}\label{eq:1.4}
T_g \cdot T_h = \sum\nolimits_{w \in G} a_{g,h,w} T_w
\end{equation}
for certain $a_{g,h,w} \in \mb k$. For any point 
$q \in \C^r$ we can endow the vector space $\C [G]$
with the structure of an associative algebra by
\begin{equation}\label{eq:1.5}
T_g \cdot_q T_h = \sum\nolimits_{w \in G} a_{g,h,w}(q) T_w
\end{equation}
We denote the resulting algebra by $\mc H (G,q)$. It is 
isomorphic to the tensor product $A \otimes_{\mb k} \C$
where $\C$ has the $\mb k$-module structure obtained 
from evaluating at $q$. Assume moreover that there exists a
$q^0 \in \C^r$ such that 
\[
\mc H (G ,q^0 ) = \mh C [G]
\]
We express the rigidity of finite dimensional semisimple
algebras by the following special case of Tits' deformation theorem, 
see \cite[p. 357 - 359]{Car1} or \cite[Appendix]{Iwa2}.

\begin{thm}\label{thm:1.1}
There exists a polynomial $P \in \mb k$ such that the 
following are equivalent :
\begin{itemize}
\item $P(q) \neq 0$,
\item $\mc H (G,q)$ is semisimple,
\item $\mc H (G,q) \cong \mh C [G]$.
\end{itemize}
\end{thm}

In other words: when $\mc H (G,q)$ is semisimple it is isomorphic to
$\C [G]$, and otherwise the algebra $\mc H (W,q)$ has nilpotent ideals
and looks very different from $\mh C [G]$. The simplest case where the latter
occurs is already $G = \{e,s\}$. Namely, for $q = -1$ we get the Hecke algebra
\begin{equation}\label{eq:1.25}
\mc H (G,-1) = \C [T_s + 1] / (T_s + 1)^2.
\end{equation}
Let us discuss how it works out for an arbitrary finite Coxeter system $(W,S)$
and a parameter function $q$ as before. For an element $w \in W$ with
a reduced expression $s_1 s_2 \cdots s_r$ we put
\begin{equation}\label{eq:1.10}
q(w) = q_{s_1} q_{s_2} \cdots q_{s_r}.
\end{equation}
This is well-defined by the presentation of $W$ and condition \eqref{eq:1.1}.
We want to know under which conditions the algebra $\mc H (W,q)$ is semisimple.
For such groups the polynomials $P(q)$ of Theorem \ref{thm:1.1} have been 
determined explicitly. If we are in the equal label case 
$q(s) = q \; \forall s \in S$ then we may take
\begin{equation}\label{eq:1.24}
P(q) = q \sum\nolimits_{w \in W} q^{\ell (w)} 
\end{equation}
except that we must omit the factor $q$ if $W$ is of type 
$(A_1 )^n$, see \cite{GyUn}. More generally, suppose
that if $(W,S)$ is irreducible and $S$ consists 
of two conjugacy classes, with parameters $q_1$ and $q_2$.
Gyoja \cite[p. 569]{Gyo} showed that in most of these cases we may take
\begin{equation}\label{eq:1.6}
\begin{array}{ccl}
P(q_1 ,q_2 ) & = & q_1^{|W|} q_2 W(q_1 ,q_2 ) W(q_1^{-1},q_2 )\\
W(q_1 ,q_2 ) & = & \sum_{w \in W} q(w)
\end{array}
\end{equation}
So generically there is an isomorphism 
\begin{equation}\label{eq:1.7}
\mc H (W,q) \cong \mh C [W] .
\end{equation}
In particular, for almost all $q$ the representation theory of
$\mc H(W,q)$ is just that of $W$ (over $\C$). 

When $q_1 q_2 \neq 0$ and $W(q_1,q_2) = 0$, the subgroup of $\C^\times$
generated by $q_1$ and $q_2$ contains a root of unity different from 1.
Hence the non-semisimple algebras $\mc H(W,q)$ are those with (at least)
one $q_s$ equal to 0 and some for which $q$ involves nontrivial roots
of unity. Although these algebras can have an interesting combinatorial
structure, their behaviour is quite different from that of affine Hecke algebras.

\subsection{Iwahori--Hecke algebras} \
\label{par:defIHA}

Let $(W,S)$ be any Coxeter system. As in the previous paragraph, we can assign
complex numbers $q_s$ to the elements of $S$. When \eqref{eq:1.1} is fulfilled,
we can construct an algebra $\mc H (W,q)$ exactly as before. For $q=1$ this is 
a group algebra of $W$, and hence, for $q \neq 1$, $\mc H (W,q)$ can be regarded
as some deformation of $\C [W]$. However, the situation is much more complicated
than when $W$ is finite. Tits' deformation theorem does not work here, and in general
many different $q$'s can lead to mutually non-isomorphic algebras \cite{Yan}.

To get some grip on the situation, we restrict our scope from general Coxeter groups
to affine Weyl groups. By that we mean Coxeter systems $(W_\af,S_\af)$ such that every
irreducible component $S_i$ of $S_\af$ generates a Coxeter group $W_i$ of affine type. 
The affineness condition is equivalent to: the Cartan matrix of $(W_i,S_i)$ is positive
semidefinite but not positive definite \cite[\S 2.5,\S 4.7, \S 6.5]{Hum1}. Thus
irreducible affine Weyl groups are classified by the Dynkin diagrams
$\tilde A_n, \tilde B_n, \tilde C_n, \tilde D_n, \tilde E_6, \tilde E_7,
\tilde E_8, \tilde F_4, \tilde G_2$.

In the simplest case $\tilde A_1$, $W_\af$ is an infinite dihedral group,
freely generated by two elements of order 2.

\begin{defn}\label{def:1.2}
An Iwahori--Hecke algebra is an algebra of the form $\mc H (W,q)$, where $(W,S)$ is
any Coxeter system. We say that $\mc H (W,q)$ is of affine type if $(W,S)$ is an
affine Weyl group.
\end{defn}

Iwahori--Hecke algebras of affine type where discovered first by Matsumoto and 
Iwahori \cite{Mat1,IwMa,Iwa2}, in the context of reductive $p$-adic groups. For
instance, let $G$ be a split, simply connected, semisimple group over $\Q_p$ and let
$I$ be an Iwahori subgroup of $G$. Then the convolution algebra $C_c (I \backslash G / I)$
is isomorphic to $\mc H(W_\af,q)$, where $(W_\af,S_\af)$ is derived from $G$ and $q_s = p$ 
for all $s \in S$. This is called the Iwahori-spherical Hecke algebra of $G$, because
its modules classify $G$-representations with $I$-fixed vectors.

The basic structure of an affine Weyl group $W_\af$ is described in \cite[\S 4.2]{Hum1}
and \cite[\S VI.2]{Bou}. The set of elements whose conjugacy class is finite forms a finite
index normal subgroup of $W_\af$, isomorphic to a lattice. That lattice is
spanned by an integral root system $R$, and $W_\af$ is the semidirect product of $\Z R$ 
and the Weyl group of $R$. In particular $\mc H (W_\af,q)$ contains $\mc H (W(R),q)$ as a 
subalgebra. However, the embedding of $W(R)$ in $W_\af$ is in general not unique.

\subsection{Affine Hecke algebras} \
\label{par:defAHA}

Affine Hecke algebras generalize Iwahori--Hecke algebras of affine type. Instead of 
affine Weyl groups, we allow more general groups which are semidirect products
of lattices and finite Weyl groups. The best way to do the bookkeeping is with
root data, for which our standard references are \cite{Bou,Hum1}. 

Consider a quadruple $\mc R = (X,R,Y,R^\vee)$, where
\begin{itemize}
\item $X$ and $Y$ are lattices of finite rank, with a perfect pairing 
$\langle \cdot, \cdot \rangle : X \times Y \to \Z$,
\item $R$ is a root system in $X$,
\item $R^\vee \subset Y$ is the dual root system, and a bijection $R \to R^\vee,
\alpha \mapsto \alpha^\vee$ with $\langle \alpha, \alpha^\vee \rangle = 2$ is given,
\item for every $\alpha \in R$, the reflection
\[
s_\alpha : X \to X ,\; s_\alpha (x) = x - \inp{x}{\alpha^\vee} \alpha 
\]
stabilizes $R$,
\item for every $\alpha^\vee \in R^\vee$, the reflection
\[
s^\vee_\alpha : Y \to Y ,\; s^\vee_\alpha (y) = y - \inp{\alpha}{y} \alpha^\vee 
\]
stabilizes $R^\vee$.
\end{itemize}
If all these conditions are met, we call $\mc R$ are root datum. It comes with a
finite Weyl group $W = W(R)$ and an infinite group $W(\mc R) = X \rtimes W(R)$,
the extended affine Weyl group of $\mc R$. Often we will add a base of $R$ to
$\mc R$, and speak of a based root datum. 

\begin{ex}
Take $X = Y = \Z, R = \{\pm 1\}$ and $R^\vee = \{ \pm 2\}$. Then
\begin{equation}\label{eq:1.8}
W(\mc R) = \Z \rtimes S_2, \text{ an infinite dihedral group.}
\end{equation}
We stress that we do not require $R$ to span the vector space $X \otimes_\Z \R$.
We say that $\mc R$ is semisimple if $R$ does span $X \otimes_\Z \R$.
For non-semisimple root data $R$ and $R^\vee$ may even be empty. For instance, root 
datum $\mc R = (\Z^n,\emptyset, \Z^n,\emptyset)$ has infinite group $W(\mc R) = \Z^n$,
but no reflections.
\end{ex}

More examples of root data (actually all) come from reductive groups, see \cite{Spr}.
Suppose that $\mc G$ is a reductive algebraic group, $\mc T$ is a maximal torus
in $\mc G$ and $R(\mc G,\mc T)$ is the associated root system. Denote the character 
lattice of $\mc T$ by $X^* (\mc T)$ and its cocharacter lattice by $X_* (\mc T)$. Then
\begin{equation}\label{eq:1.9}
\mc R (\mc G,\mc T) := 
\big( X^* (\mc T), R(\mc G,\mc T), X_* (\mc T), R(\mc G,\mc T)^\vee \big).
\end{equation}
is a root datum. We note that $\mc R (\mc G, \mc T)$ is semisimple if and only if
$\mc G$ is semisimple. \\

The group $W(\mc R) = X \rtimes W$ acts naturally on the vector space $X \otimes_\Z \R$,
$X$ by translations and $W$ by linear extension of its action on $X$. The collection of
hyperplanes 
\[
H_{\alpha, n} = \{ x \in X \otimes_\Z \R : \alpha^\vee (x) = n \} 
\qquad  \alpha \in R, n \in \Z 
\]
is $W(\mc R)$-stable and divides $X \otimes_\Z \R$ in open subsets called alcoves. 
Let $W_{\mr{aff}}$ be the subgroup of $W(\mc R)$ generated by the (affine) reflections in
the hyperplanes $H_{\alpha, n}$. This is an affine Weyl group, and $X \otimes_\Z \R$
with this hyperplane arrangement is its Coxeter complex.

To construct Hecke algebras from root data, we need to specify a set of Coxeter
generators of $W_\af$. In this setting they will be affine reflections.
Let $\Delta$ be a base of $R$. As is well-known, it yields a set of simple reflections 
$S = \{ s_\alpha : \alpha \in \Delta \}$, and $(W,S)$ is a finite Coxeter system.
(For a root datum of the form $\mc R (\mc G,\mc T)$ as in \eqref{eq:1.8}, the choice of $\Delta$
is equivalent to the choice of a Borel subgroup $\mc B \subset \mc G$ containing $\mc T$.)

The base $\Delta$ determines a ``fundamental alcove" $A_0$ in $X \otimes_\Z \R$, namely
the unique alcove contained in the positive Weyl chamber (with respect to $\Delta$),
such that $0 \in \overline{A_0}$. The reflections in those walls of $A_0$ that contain 0
constitute precisely $S$. The set $S_\af$ of (affine) reflections with respect to all walls 
of $A_0$ forms the required collection of Coxeter generators of $W_\af$.

\begin{ex}
A part of the hyperplane arrangement for an affine Weyl group of type $\tilde{A_2}$, 
with $A = A_0$, $\Delta = \{\alpha,\beta\}$ and $S_\af = \{s_\alpha,s_\beta,s_\gamma\}$.

\includegraphics[width=9cm]{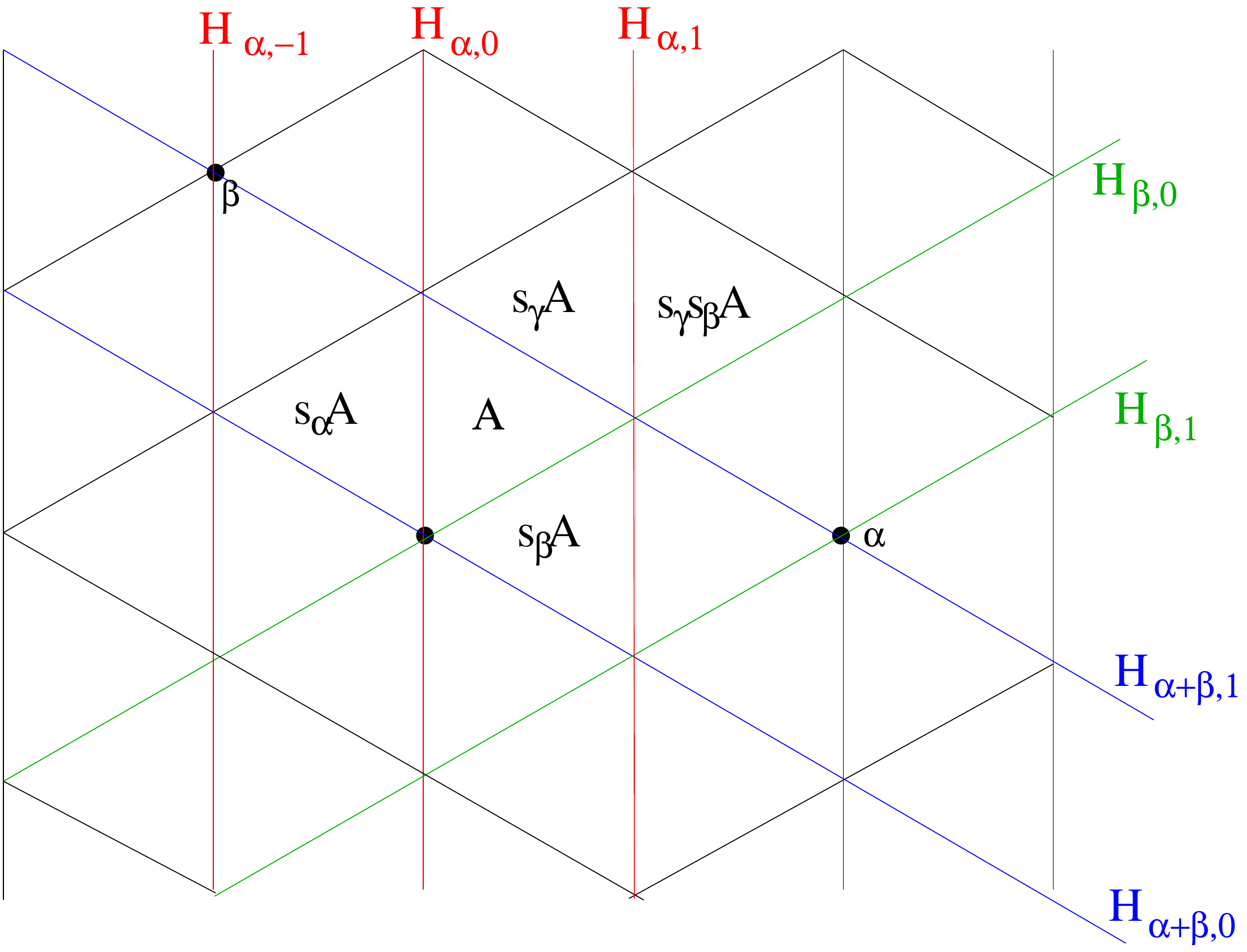}
\end{ex}

We can make $S_\af$ more explicit. Let $R^\vee_{\mr{max}}$ be the set of maximal elements of
$R^\vee$, with respect to the base $\Delta^\vee$. It contains one element for every irreducible
component of $R^\vee$. For $\alpha^\vee \in R^\vee_{\mr{max}}$, define
\[
s'_\alpha : X \to X ,\; s'_\alpha (x) = x + \alpha - \inp{x}{\alpha^\vee} \alpha .
\]
This is the reflection of $X \otimes_\Z \R$ in the hyperplane $H_{\alpha,1}$, a wall
of $A_0$. Then
\[
S_\af = S \cup \{ s'_\alpha : \alpha^\vee \in R^\vee_{\mr{max}} \}.
\]
We denote the based root datum $(X,R,Y,R^\vee,\Delta)$ also by $\mc R$. Thus the sets
$S, S_\af$ and the subgroup $W_\af \subset W(\mc R)$ are determined by $\mc R$.
\begin{ex}
A couple of important instances, coming from the reductive groups $PGL_2$ and $GL_n$:
\begin{itemize}
\item $X = Y = \Z, R = \{ \pm 2\}, R^\vee = \{ \pm 1\}, \Delta = \{ \alpha = 2\}$. Here 
$S_\af = \{ s_\alpha, s'_\alpha : x \mapsto 2 - x \}$ and $W_\af = 2 \Z \rtimes S_2$,
a proper subgroup of $W(\mc R) = \Z \rtimes S_2$. This is the same $W(\mc R)$ as in
\eqref{eq:1.8}, but with a different set of simple affine reflections.
\item $X = Y = \Z^n, R = R^\vee = A_{n-1} = \{ e_i - e_j : 1 \leq i ,j \leq n, i \neq j \}$,
$\Delta = \{ e_i - e_{i+1} : i = 1,\ldots,n-1\}$. In this case
\begin{align} \nonumber
S_\af = \{ s_i = s_{e_i - e_{i+1}} & : i = 1,\ldots,n-1\} \cup
\{ s_0 : x \mapsto x + (1 - \inp{x}{e_1 - e_n}) (e_1 - e_n) \} \\
& \label{eq:1.12} W_\af = \{ (x_1,\ldots,x_n) \in \Z^n : x_1 + \cdots + x_n = 0\} \rtimes S_n
\end{align}
\end{itemize}
\end{ex}

Like Iwahori--Hecke algebras, affine Hecke algebras involve $q$-parameters. Fix
$\mb{q} \in \R_{>1}$ and let $\lambda, \lambda^* : R \to \C$ be functions such that
\begin{itemize}
\item if $\alpha, \beta \in R$ are $W$-associate, then $\lambda (\alpha) = \lambda (\beta)$
and $\lambda^* (\alpha) = \lambda^* (\beta)$,
\item if $\alpha^\vee \notin 2 Y$, then $\lambda^* (\alpha) = \lambda (\alpha)$.
\end{itemize}
We note that $\alpha^\vee \in 2Y$ is only posssible for short roots $\alpha$ in 
a type $B$ component of $R$. For $\alpha \in R$ we write
\begin{equation}\label{eq:1.28}
q_{s_\alpha} = \mb{q}^{\lambda (\alpha)} \quad \text{and (if } 
\alpha^\vee \in R^\vee_{\mr{max}} ) \quad q_{s'_\alpha} = \mb{q}^{\lambda^* (\alpha)} .
\end{equation}
Recall that $\mc H (W,q)$ is the Iwahori--Hecke algebra of $W = W(R)$ and let
$\{ \theta_x : x \in X\}$ be the standard basis of $\C [X]$.

\begin{defn}\label{def:1.3}
The affine Hecke algebra $\mc H (\mc R,\lambda,\lambda^*,\mb{q})$ is the vector space\\
$\C [X] \otimes_\C \mc H (W,q)$ with the multiplication rules:
\begin{itemize}
\item $\C[X]$ and $\mc H (W,q)$ are embedded as subalgebras,
\item for $\alpha \in \Delta$ and $x \in X$:
\[
\theta_x T_{s_\alpha} - T_{s_\alpha} \theta_{s_\alpha (x)} =
\Big( (\mb{q}^{\lambda (\alpha)} - 1) + \theta_{-\alpha} \big( \mb{q}^{(\lambda (\alpha) +
\lambda^* (\alpha))/2} - \mb{q}^{(\lambda (\alpha) - \lambda^* (\alpha))/2} \big) \Big)
\frac{\theta_x - \theta_{s_\alpha (x)}}{\theta_0 - \theta_{-2 \alpha}} .
\]
\end{itemize}
\end{defn}
When $\alpha^\vee \notin 2 Y$, the cross relation simplifies to
\[
\theta_x T_{s_\alpha} - T_{s_\alpha} \theta_{s_\alpha (x)} =
(\mb{q}^{\lambda (\alpha)} - 1)(\theta_x - \theta_{s_\alpha (x)})(\theta_0 - \theta_{-\alpha})^{-1} .
\]
Notice that here the right hand side lies in $\C[X]$ because 
\[
\theta_x - \theta_{s_\alpha (x)} = \theta_x - \theta_{x - \inp{x}{\alpha^\vee} \alpha} =
\theta_x (\theta_0 - \theta_{- \inp{x}{\alpha^\vee} \alpha})
\]
is divisible by $\theta_0 - \theta_{-\alpha}$.
It follows from \cite[\S 3]{Lus-Gr} that $\mc H (\mc R,\lambda,\lambda^*,\mb{q})$ is
really an associative algebra with unit element $\theta_0 \otimes T_e$, and that the 
multiplication map
\begin{equation}\label{eq:1.19}
\begin{array}{ccc}
\mc H (W,q) \otimes_\C \C [X] & \to &  \mc H (\mc R, \lambda,\lambda^*,\mb{q}) \\
h \otimes f & \mapsto & h \cdot f
\end{array}
\end{equation}
is bijective. We say that $\mc H (\mc R,\lambda,\lambda^*,\mb{q})$ has equal parameters if 
\[
\lambda (\alpha) = \lambda (\beta) = \lambda^* (\alpha) = \lambda^* (\beta)
\quad \text{for all } \alpha, \beta \in R.
\]
When $\lambda = \lambda^* = 1$, we omit them from the notation and write simply $\mc H (\mc R, \mb{q})$. 
In that setting we will often allow $\mb q$ to be any element of $\C^\times$. We note that for 
$\lambda = \lambda^* = 0$ or $\mb q = 1$ we recover the group algebra $\C[X \rtimes W]$.
In particular, the only affine Hecke algebra associated to $(X,\emptyset,Y,\emptyset)$ is $\C[X]$.

Affine Hecke algebras appear foremostly in the representation theory of reductive $p$-adic
groups, see \cite{BuKu1,Mor1,Roc,ABPSaroundLLC,SolEnd}. They also have strong ties to
orthogonal polynomials \cite{Kir,McD}, which often run via double affine Hecke algebras 
\cite{Che2}. This has lead to a whole family of Hecke algebras, with adjectives like
degenerate, cyclotomic, rational, graded and (double) affine. We will only discuss one further 
member of this family, in paragraph \ref{par:defGHA}.\\

As already explained, Tits' deformation theorem does not apply to affine Hecke algebras. But 
there is a substitute for the semisimplicity part of Theorem \ref{thm:1.1}. Recall that a
finite dimensional algebra $A$ is semisimple (i.e. a direct sum of simple algebras) if and only if
its Jacobson radical Jac$(A)$ (the intersection of the kernels of all simple modules) is zero. 
In that sense Theorem \ref{thm:1.1} admits a partial generalization, see
\cite[Lemma 3.4]{SolThe} and \cite[(3.4.5)]{Mat2}:

\begin{lem}\label{lem:1.4}
The Jacobson radical of an affine Hecke algebra $\mc H (\mc R,\lambda,\lambda^*,\mb{q})$ is zero.
\end{lem}

\subsection{Presentations of affine Hecke algebras} \

We will make the relations between the algebras in paragraphs \ref{par:defIHA} and 
\ref{par:defAHA} explicit. We start with the Bernstein presentation of an 
Iwahori--Hecke algebra of affine type.

Let $W_\af$ be an affine Weyl group with Coxeter generators $S_\af$, and write it as
$W_\af = \Z R \rtimes W(R)$. Put 
\[
\mc R = \big( \Z R, R, \Hom_\Z (\Z R,\Z),R^\vee,\Delta \big) , 
\]
where $\Delta = \{ \alpha \in R : s_\alpha \in S_\af \}$. Consider an Iwahori--Hecke
algebra $\mc H (W_\af,q)$ with parameters $q_s \in \C$.

\begin{thm}\label{thm:1.5}(Bernstein, see \cite[\S 3]{Lus-Gr}) \\
Suppose that $q_s \neq 0$ for all $s \in S_\af$. Pick $\lambda (\alpha), \lambda^* (\alpha)$
such that $q_{s_\alpha} = \mb{q}^{\lambda (\alpha)}$ for all $\alpha \in R$ and 
$q_{s'_\alpha} = \mb{q}^{\lambda^* (\alpha)}$ when $\alpha^\vee \in R^\vee_{\mr{max}}$.
Then there exists a unique algebra isomorphism $\mc H (W_\af,q) \to \mc H (\mc R,\lambda,
\lambda^*, \mb{q})$ such that:
\begin{itemize}
\item it is the identity on $\mc H (W,q)$,
\item for $x \in \Z R$ with $\inp{x}{\alpha^\vee} \geq 0$ for all $\alpha \in \Delta$,
it sends $q(x)^{-1/2} T_x$ to $\theta_x$.
\end{itemize}
Here $q(x)^{1/2}$ is defined via \eqref{eq:1.10} and $q_{s_\alpha}^{1/2} = 
\mb{q}^{\lambda (\alpha)/2}, q_{s'_\alpha}^{1/2} = \mb{q}^{\lambda^* (\alpha)/2}$. 
\end{thm}

Not every affine Hecke algebra is isomorphic to an Iwahori--Hecke algebra, for instance
$\C [X]$ is not. To be precise, an isomorphism as in Theorem \ref{thm:1.5} exists if and 
only if the root datum $\mc R$ is that of an adjoint semisimple group.
To compensate for this difference in scope, we take another look at the structure of 
$W(\mc R)$ for an arbitrary based root datum $\mc R$.

We know from \cite[\S 4.4]{Hum1} that the length function $\ell$ of $(W_\af,S_\af)$ satisfies
\begin{equation}\label{eq:1.11}
\ell (w) = \text{ number of hyperplanes } H_{\alpha,n} \text{ that separate }
w(A_0) \text{ from } A_0 .
\end{equation}
We extend $\ell$ to a function $W(\mc R) \to \Z_{\geq 0}$, by decreeing that \eqref{eq:1.11}
is valid for all $w \in W(\mc R)$. Then
\[
\Omega := \{ w \in W(\mc R) : \ell (w) = 0 \} = \text{ stabilizer of } A_0 \text{ in } W(\mc R)
\]
is a subgroup of $W(\mc R)$. The group $\Omega$ acts by conjugation on $W_\af$ and that
action stabilizes $S_\af$ (the set of reflections with respect to the walls of $A_0$). 
Moreover, since $W_\af$ acts simply transitively on the set of alcoves in $X \otimes_\Z \R$:
\[
W(\mc R) = W_\af \rtimes \Omega.
\]
\begin{ex}
Let $\mc R$ be of type $GL_n$, as in \eqref{eq:1.12}. Then $\Omega$ is
isomorphic to $\Z$, generated by $\omega = e_1 (1\,2\cdots n)$. The action of $\omega$ on
$S_\af = \{s_0,s_1,\ldots,s_{n-1}\}$ is $\omega s_i \omega^{-1} = s_{i+1}$ (where $s_n$
means $s_0$).
\end{ex}

Let $q : S_\af \to \C^\times$ be the parameter function determined by $\mb{q},\lambda,
\lambda^*$. The conditions on $\lambda$ and $\lambda^*$ (before Definition \ref{def:1.3})
ensure that $q$ is $\Omega$-invariant. Hence the formula
\[
\omega (T_w) = T_{\omega w \omega^{-1}} 
\]
defines an algebra automorphism of $\mc H (W_\af,q)$. This gives a group action of
$\Omega$ on $\mc H (W_\af,q)$. Recall that the crossed product
algebra $\mc H (W_\af,q) \rtimes \Omega$ is the vector space $\mc H (W_\af,q) 
\otimes_\C \C [\Omega]$ with multiplication defined by
\[
\omega \cdot T_w \cdot \omega^{-1} = \omega (T_w) .
\]
For $w \in W_\af, \omega \in \Omega$ we write $T_{w \omega} = T_w \omega \in \mc H (W_\af,q) 
\rtimes \Omega$. The multiplication relations in $\mc H (W_\af,q) \rtimes \Omega$ become
\begin{equation}\label{eq:1.13}
\begin{array}{ll}
T_v T_w = T_{vw} & \text{if } \ell (vw) = \ell (v) + \ell (w) , \\
(T_s + 1)(T_s - q_s) = 0 & \text{for } s \in S_\af.
\end{array}
\end{equation}
Now we can formulate the counterpart of Theorem \ref{thm:1.5}.

\begin{thm}\label{thm:1.6} (Bernstein, see \cite[\S 3]{Lus-Gr}) \\
There is a unique algebra isomorphism $\mc H (\mc R,\lambda,\lambda^*,\mb{q}) \to
\mc H (W_\af,q) \rtimes \Omega$ such that
\begin{itemize}
\item it is the identity on $\mc H (W,q)$,
\item for $x \in \Z R$ with $\inp{x}{\alpha^\vee} \geq 0$ for all $\alpha \in \Delta$,
it sends $\theta_x$ to $q(x)^{-1/2} T_x$.
\end{itemize}
\end{thm}

The algebra $\mc H (W_\af,q) \rtimes \Omega$, with the multiplication rules \eqref{eq:1.13},
is called the Iwahori--Matsumoto presentation of $\mc H (\mc R,\lambda,\lambda^*,\mb{q})$.

We conclude this paragraph with yet another, more geometric, presentation of affine Hecke algebras.
Let $T$ be the complex algebraic torus $\Hom_\Z (X,\C^\times)$. By duality $\Hom (T,\C^\times)
= X$, and the ring of regular functions $\mc O (T)$ is the group algebra $\C [X]$. The group
$W$ acts naturally on $X$, and that induces actions on $\C [X]$ and on $T$.

\begin{defn}\label{def:1.7}
The algebra $\mc H (T,\lambda,\lambda^*,\mb{q})$ is the vector space $\mc O (T) \otimes_\C
\mc H (W,q)$ with the multiplication rules
\begin{itemize}
\item $\mc O (T)$ and $\mc H (W,q)$ are embedded as subalgebras,
\item for $\alpha \in \Delta$ and $f \in \mc O (T)$:
\[
f T_{s_\alpha} - T_{s_\alpha} (s_\alpha \cdot f) = 
\Big( (\mb{q}^{\lambda (\alpha)} - 1) + \theta_{-\alpha} \big( \mb{q}^{(\lambda (\alpha) +
\lambda^* (\alpha))/2} - \mb{q}^{(\lambda (\alpha) - \lambda^* (\alpha))/2} \big) \Big)
\frac{f - s_\alpha (f)}{\theta_0 - \theta_{-2 \alpha}} .
\]
\end{itemize}
\end{defn}
Clearly the identification $\mc O (T) \cong \C [X]$ induces an algebra isomorphism
\[
\mc H (T,\lambda,\lambda^*,\mb{q}) \cong \mc H (\mc R,\lambda,\lambda^*,\mb{q}) .
\]
From the above presentation it is easy to find the centre of these algebras \cite{Lus-Gr}:
\begin{equation}\label{eq:1.14}
Z \big( \mc H (\mc R,\lambda,\lambda^*,\mb{q}) \big) \cong 
Z \big( \mc H (T,\lambda,\lambda^*,\mb{q}) \big) = \mc O (T)^W = \mc O (T/W) \cong \C[X]^W.
\end{equation}
An advantage of Definition \ref{def:1.7} is that this presentation can also be used if $T$
is just known as an algebraic variety, without a group structure. In that situation
$\mc H (T,\lambda,\lambda^*,\mb{q})$ can be studied without fixing a basepoint of $T$, it
suffices to have the $W$-action and the elements $\theta_{-\alpha} \in \mc O (T)^\times$ 
for $\alpha \in \Delta$.
This is particularly handy for affine Hecke algebras arising from Bernstein components and
types for reductive $p$-adic groups.

Even more flexibly, the above presentation applies when $\mc O (T)$ is replaced by a 
reasonable algebra of differentiable functions on $T$, like rational functions, analytic
functions or smooth functions. \\

We already observed that Tits' deformation theorem fails for affine Hecke algebras.
Nevertheless, apart from Lemma \ref{lem:1.4} there is another analogue of Theorem 
\ref{thm:1.1}, obtained by replacing the centre of an affine Hecke algebra by its
quotient field.

Let $\C (X)$ be the quotient field of $\C [X]$, that is, the field of rational functions
on the complex algebraic variety $T$. The action of $W$ on $T$ gives rise to the crossed
product algebra $\C (X) \rtimes W$. The quotient field of 
$Z \big( \mc H (\mc R,\lambda,\lambda^*,\mb{q}) \big) \cong \C [X]^W$ is $\C (X)^W$,
which is also the centre of $\C (X) \rtimes W$. We construct the algebra
\begin{equation}\label{eq:1.21}
\C (X)^W \otimes_{\C [X]^W} \mc H (T,\lambda,\lambda^*,\mb{q}) \cong
\C (X) \otimes_{\C [X]} \mc H (T,\lambda,\lambda^*,\mb{q}) \cong
\C (X) \otimes_\C \mc H (W,q) .
\end{equation}
Here the multiplication comes from the description on the left, and it is the same algebra
as obtained from Definition \ref{def:1.7} by substituting $\C (X)$ for $\mc O (T)$.

For $\alpha \in \Delta$ we define an element $\imath^\circ_{s_\alpha}$ of \eqref{eq:1.21} by
\[
\imath^\circ_{s_\alpha} + 1 =
\frac{\mb{q}^{-\lambda (\alpha)} (\theta_\alpha - 1)(\theta_\alpha + 1)}{
(\theta_\alpha - \mb{q}^{(\lambda (\alpha) + \lambda^* (\alpha))/ 2} )
(\theta_\alpha + \mb{q}^{(\lambda (\alpha) - \lambda^* (\alpha)) / 2})} (1 + T_{s_\alpha}) .
\]
\begin{prop}\label{prop:1.9}
\textup{\cite[\S 5]{Lus-Gr}}
\enuma{
\item The map $s_\alpha \mapsto \imath^\circ_{s_\alpha}$ extends to a group homomorphism
\[
W \to \big( \C (X)^W \otimes_{\C [X]^W} \mc H (T,\lambda,\lambda^*,\mb{q}) \big)^\times :
w \mapsto \imath^\circ_w .
\]
\item Using the description \eqref{eq:1.21} we define a map
\[
\C (X) \rtimes W \to \C (X)^W \otimes_{\C [X]^W} \mc H (T,\lambda,\lambda^*,\mb{q}) :
f \otimes w \mapsto f \imath^\circ_w .
\]
This is an algebra isomorphism, and in particular 
\[
\imath^\circ_w f \imath^\circ_{w^{-1}} = w(f) \qquad f \in \C (X), w \in W .
\]
}
\end{prop}

\subsection{Graded Hecke algebras} \
\label{par:defGHA}

Graded (affine) Hecke algebras were discovered in \cite{Dri,Lus-Gr}. They are simplified
versions of affine Hecke algebras, more or less in the same way that a Lie algebra is
a simplification of a Lie group.

Let $\mf a$ be a finite dimensional Euclidean space and let $W$ be a finite Coxeter group
action isometrically on $\mf a$ (and hence also on $\mf a^*$). Let $R \subset \mf a^*$
be a reduced root system, stable under the action of $W$, such that the reflections 
$s_\alpha$ with $\alpha \in R$ generate $W$. These conditions imply that $W$ acts trivially
on the orthogonal complement of $\R R$ in $\mf a^*$.

In contrast with the previous paragraphs, $R$ does not have to be integral and $W$ does not
have to be crystallographic. Thus we are dealing with root systems as in \cite[\S 1.2]{Hum1}:
just reduced and $W$-stable, no further condition. In particular the upcoming construction 
applies equally well to Coxeter groups of type $H_3, H_4$ or $I_2^{(m)}$.

Write $\mf t = \mf a \otimes_\R \C$ and let $S(\mf t^*) = \mc O (\mf t)$ be the algebra
of polynomial functions on $\mf t$. We choose a $W$-invariant parameter function
$k : R \to \C$ and we let $\mb{r}$ be a formal variable. We also fix a base $\Delta$ of $R$.

\begin{defn}\label{def:1.8}
The graded Hecke algebra $\mh H (\mf t,W,k,\mb r)$ is the vector space\\
$\C [W] \otimes_\C S(\mf t^*) \otimes_\C \C [\mb{r}]$ with the multiplication rules
\begin{itemize}
\item $\C [W]$ and $S(\mf t^*) \otimes_\C \C [\mb{r}] \cong 
\mc O (\mf t \oplus \C)$ are embedded as subalgebras,
\item $\C [\mb{r}]$ is central,
\item the cross relation for $\alpha \in \Delta$ and $\xi \in S(\mf t^*)$:
\[
\xi \cdot s_\alpha - s_\alpha \cdot s_\alpha (\xi) = 
k(\alpha) \mb{r} \frac{\xi - s_\alpha (\xi)}{\alpha} .
\]
\end{itemize}
The grading is given by $\deg (w) = 0$ for $w \in W$ and $\deg (x) = \deg (\mb{r}) = 2$ 
for $x \in \mf t^* \setminus \{0\}$.
\end{defn}

The grading is motivated by a construction of graded Hecke algebras with equivariant (co)homology,
which we will discuss in Paragraph \ref{par:homology}.
Notice that for $k = 0$ Definition \ref{def:1.8} yields the crossed product algebra
\[
\mh H (\mf t, W, 0,\mb{r}) = \C[\mb{r}] \otimes_\C S(\mf t^*) \rtimes W.
\]
Let $R^\vee \subset \mf a$ be the coroot system of $R$, so with $\inp{\alpha}{\alpha^\vee} = 2$
for all $\alpha \in R$. For $x \in \mf t^*$ the following relation holds in 
$\mh H (\mf t,W,k,\mb{r})$:
\begin{equation}\label{eq:1.15}
s_\alpha \cdot x - s_\alpha (x) \cdot s_\alpha = k(\alpha) \mb{r} \inp{x}{\alpha^\vee} .
\end{equation}
In fact \eqref{eq:1.15} can be substituted for the cross relation in Definition \ref{def:1.8},
that suffices to determine the algebra structure uniquely.

With such a simple presentation, it is no surprise that the graded Hecke algebras
$\mh H (\mf t,W,k)$ have more diverse applications than affine Hecke algebras. They
appear in the representation theory of reductive groups over local fields, both in the
$p$-adic case \cite{BaMo1,BaMo2,LusUni1} and in the real case \cite{CiTr1,CiTr2}.
Further, these graded Hecke algebras can be realized with Dunkl operators \cite{Che1,Opd1},
which enables them to act on many interesting function spaces.\\

The above construction can be modified a little, and still produce the same algebra. Namely,
fix $\alpha \in R$ and $\epsilon \in \R_{>0}$. For all $w \in W$ we replace $k(w \alpha)$ by
$\epsilon k(w \alpha)$ and $w \alpha$ by $\epsilon w \alpha$--that again gives a root system
in the sense of \cite[\S 1.2]{Hum1}. This operation preserves the cross relation in
Definition \ref{def:1.8}, so does not change the algebra. 
Further, we can allow $R$ to be non-reduced, as long as we impose in addition that 
$k(\epsilon \alpha) = \epsilon k (\alpha)$ whenever $\epsilon > 0$ and 
$\alpha, \epsilon \alpha \in R$ -- that still gives the same graded Hecke algebras.

Similarly, we can scale all parameters $k(\alpha)$ simultaneously. Namely, scalar 
multiplication with $z \in \C^\times$ defines a bijection $m_z : \mf t^* \to \mf t^*$,
which clearly extends to an algebra automorphism of $S(\mf t^*)$. From Definition 
\ref{def:1.8} we see that it extends even further, to an algebra isomorphism
\begin{equation}\label{eq:1.16}
m_z : \mh H (\mf t, W, zk, \mb{r}) \to \mh H (\mf t,W,k,\mb{r}) 
\end{equation}
which is the identity on $\C [W] \otimes_\C \C [\mb{r}]$. Notice that for $z=0$ the
map $m_z$ is well-defined, but no longer bijective. It is the canonical surjection
\[
\mh H (\mf t,W,0,\mb{r}) \to \C [W] \otimes_\C \C [\mb{r}] .
\]
Algebras like $\mh H (\mf t,W,k,\mb{r})$ are degenerations of affine Hecke 
algebras (the version where $\mb{q}$ is a formal variable) and arise in the study of
cuspidal local systems on unipotent orbits in complex reductive Lie algebras 
\cite{LusCusp1,LusCusp2,LusCusp3,AMS2}.

More often one encounters versions of $\mh H (\mf t,W,k,\mb{r})$ with $\mb{r}$
specialized to a nonzero complex number. In view of \eqref{eq:1.16} it hardly
matters which specialization, so it suffices to look at $\mb r \mapsto 1$. The
resulting algebra $\mh H (\mf t,W,k)$ has underlying vector space
$\C [W] \otimes_\C S(\mf t^*)$ and cross relations
\begin{equation}\label{eq:1.17} 
\xi \cdot s_\alpha - s_\alpha \cdot s_\alpha (\xi) = 
k(\alpha) (\xi - s_\alpha (\xi)) / \alpha \qquad \alpha \in \Delta, \xi \in S(\mf t^*) .
\end{equation}
Like for affine Hecke algebras, we see from \eqref{eq:1.17} that the centre of
$\mh H (\mf t,W,k)$ is
\begin{equation}\label{eq:1.18}
Z(\mh H (\mf t,W,k)) = S(\mf t^*)^W = \mc O (\mf t / W) .
\end{equation}
As a vector space, $\mh H(\mf t,W,k)$ is still graded by $\deg (w) = 0$ for $w \in W$
and $\deg (x) = 2$ for $x \in \mf t^* \setminus \{0\}$. However, it is not a graded
algebra any more, because \eqref{eq:1.17} is not homogeneous in the case $\xi = \alpha$.
Instead, the above grading merely makes $\mh H (\mf t,W,k)$ into a filtered algebra.

The graded algebra obtained from this filtration is obtained by setting the right hand
side of \eqref{eq:1.17} equal to 0. In other words, the associated graded of
$\mh H (\mf t,W,k)$ is the crossed product algebra
\[
\mh H (\mf t,W,0) = S(\mf t^*) \rtimes W.
\]
The algebras $\mh H (\mf t,W,0)$ and $\mh H (\mf t,W,k)$ with $k \neq 0$ are usually
not isomorphic. But there is an analogue of Tits' deformation theorem, similar to
Proposition \ref{prop:1.9}. Let $Q(S(\mf t^*))$ be the quotient field of $S(\mf t^*)$,
that is, the field of rational functions on $\mf t$. It admits a natural $W$-action,
and the centre of the crossed product $Q(S(\mf t^*)) \rtimes W$ is $Q(S(\mf t^*))^W$.
Using \eqref{eq:1.18} we construct the algebra $Q(S(\mf t^*))^W \otimes_{S(\mf t^*)^W} 
\mh H (\mf t,W,k)$, which as vector space equals
\begin{equation}\label{eq:1.29}
Q(S(\mf t^*)) \otimes_{S(\mf t^*)} \mh H (\mf t,W,k) = Q(S(\mf t^*)) \otimes_\C \C [W] .
\end{equation}
In there we have elements
\[
\tilde{\imath}_{s_\alpha} = \frac{\alpha}{\alpha + k (\alpha)} (1 + s_\alpha) - 1
= \frac{\alpha}{\alpha + k (\alpha)} s_\alpha - \frac{k(\alpha)}{\alpha + k (\alpha)}
\qquad \alpha \in \Delta .
\]
\begin{prop}\label{prop:1.10}
\textup{\cite[\S 5]{Lus-Gr}}
\enuma{
\item The map $s_\alpha \mapsto \tilde{\imath}_{s_\alpha}$ extends to a group homomorphism
\[
W \to \big( Q(S(\mf t^*))^W \otimes_{S(\mf t^*)^W} \mh H (\mf t,W,k) \big)^\times :
w \mapsto \tilde{\imath}_w .
\]
\item With the description \eqref{eq:1.29} we define a map
\[
Q(S(\mf t^*)) \rtimes W \to Q(S(\mf t^*))^W \otimes_{S(\mf t^*)^W} \mh H (\mf t,W,k) :
f \otimes w \mapsto f \tilde{\imath}_w .
\]
This is an algebra isomorphism, and in particular 
\[
\tilde{\imath}_w f \tilde{\imath}_{w^{-1}} = w(f) \qquad f \in Q(S(\mf t^*)), w \in W .
\]
}
\end{prop}

Graded Hecke algebras can be decomposed like root systems and reductive Lie algebras.
Let $R_1, \ldots, R_d$ be the irreducible components of $R$. Write
$\mf a_i^* = \mr{span}(R_i) \subset \mf a^*$, $\mf t_i = \Hom_\R (\mf a_i^*,\C)$
and $\mf z = R^\perp \subset \mf t$. Then
\[
\mf t = \mf t_1 \oplus \cdots \oplus \mf t_d \oplus \mf z .
\]
The inclusions $W(R_i) \to W(R), \mf t_i^* \to \mf t^*$ and $\mf z^* \to \mf t^*$
induce an algebra isomorphism
\begin{equation}\label{eq:1.22}
\mh H (\mf t_1, W(R_1), k) \otimes_\C \cdots \otimes_\C \mh H (\mf t_d, W(R_d),k) 
\otimes_\C \mc O (\mf z) \; \longrightarrow \; \mh H (\mf t, W, k) .
\end{equation}
Hence the representation theory of $\mh H (\mf t, W, k)$ is more or less the product
of the representation theories of the tensor factors in \eqref{eq:1.22}. The commutative
algebra $\mc O (\mf z) \cong S(\mf z^*)$ is of course very simple, so the study of
graded Hecke algebra can be reduced to the case where the root system $R$ is irreducible.

\section{Irreducible representations in special cases}
\label{sec:special}

The most elementary instance of an affine Hecke algebra is with empty root system $R$.
The algebra associated to the root datum $(X,\emptyset,Y,\emptyset)$ is just the group
algebra $\C [X]$ of the lattice $X$. Its space of irreducible complex representations is
\begin{equation}\label{eq:2.1}
\Irr (\C [X]) = \Irr (X) = \Hom_\Z (X,\C^\times) = T .
\end{equation}
Since $\C [X]$ is a subalgebra of $\mc H ((X,R,Y,R^\vee),\lambda,\lambda^*,\mb{q})$
for any additional data $R,\lambda,\lambda^*$ and $\mb{q}$, \eqref{eq:2.1} is a good 
starting point for the representation theory of any affine Hecke algebra. 

We will discuss the affine Hecke algebras that appear most often, and we construct and 
classify all their irreducible representations.

\subsection{Affine Hecke algebras with $q = 1$} \
\label{par:q=1}

Let $\mc R = (X,R,Y,R^\vee,\Delta)$ be an arbitrary based root datum and take the
parameters $\lambda = \lambda^* = 0$. Then $q_s = 1$ for all $s \in S_\af$, and 
\[
\mc H (\mc R,0,0,\mb{q}) = \mc H (\mc R,1) = \C [X] \rtimes W .
\]
We denote the onedimensional representation of $\C[X]$ associated to $t \in T$ by $\C_t$.
Then $\ind_{\C[X]}^{\C [X] \rtimes W} (\C_t)$ is a $|W|$-dimensional representation of
$\C [X] \rtimes W$. Its restriction to $\C [X]$ is
\[
\Res^{\C[X] \rtimes W}_{\C [X]} \ind_{\C[X]}^{\C [X] \rtimes W} (\C_t) =
\bigoplus_{w \in W} w \C_t \cong \bigoplus_{w \in W} \C_{w(t)} .
\]
More generally, consider any $\C[X] \rtimes W$-representation $(\pi,V)$ that is generated
by the subspace
\[
V_t := \{ v \in V : \pi (\theta_x) v = x(t) v \; \forall x \in X \} .
\]
Then $V = \sum_{w \in W} \pi (w) V_t$ and $\pi (w) V_t = V_{w(t)}$. 
As $V_t \cap V_{t'} = \{0\}$ for $t \neq t'$,
\[
V = \ind_{\C [X] \rtimes W_t}^{\C [X] \rtimes W} (V_t), \quad \text{where }
W_t = \{ w \in W : w(t) = t\}.
\]
By Frobenius reciprocity
\begin{multline*}
\End_{\C[X] \rtimes W}(V) \cong \Hom_{\C[X] \rtimes W_t} (V_t ,V) = \\
\Hom_{\C[X] \rtimes W_t} \big( V_t ,\sum\nolimits_{w \in W / W_t} V_{w(t)} \big) =
\End_{\C[X] \rtimes W_t} (V_t) .
\end{multline*}

\begin{cor}\label{cor:2.1}
The functor $\ind_{\C [X] \rtimes W_t}^{\C [X] \rtimes W}$ induces an equivalence
between the following categories:
\begin{itemize}
\item $\C [X] \rtimes W_t$-representations on which $\C[X]$ acts via the character $t$,
\item $\C [X] \rtimes W$-representations $V$ that are generated by $V_t$.
\end{itemize}
\end{cor}
The first category in Corollary \ref{cor:2.1} is naturally equivalent with the category of
$W_t$-representations. We conclude that, for every irreducible $W_t$-representation 
$(\rho,V_\rho)$, the $\C[X] \rtimes W$-representation
\[
\pi (t,\rho) := \ind_{\C[X] \rtimes W_t}^{\C [X] \rtimes W} (\C_t \otimes V_\rho)
\]
is irreducible. For comparison with later results we point out that $\pi (t,\rho)$ is a
direct summand of the induced representation
\[
\ind_{\C[X]}^{\C[X] \rtimes W} (\C_t) = 
\ind_{\C[X] \rtimes W_t}^{\C [X] \rtimes W} (\C_t \otimes \C [W_t]) 
\]
and that 
\[
\Res_{\C[X]}^{\C[X] \rtimes W} \pi (t,\rho) = 
\bigoplus\nolimits_{w \in W / W_t} \C_{w(t)}^{\dim (V_\rho)} .
\]
The next result goes back to Frobenius and Clifford, see \cite[Appendix]{RaRa} for a
modern account.

\begin{thm}\label{thm:2.2}
Every irreducible $\C[X] \rtimes W$-representation is of the form $\pi (t,\rho)$
for a $t \in T$ and a $\rho \in \Irr (W_t)$. Two such representations $\pi (t,\rho)$
and $\pi (t',\rho')$ are equivalent if and only if there exists a $w \in W$ with
$t' = w(t)$ and $\rho' = w \cdot \rho$.
\end{thm}

Here $w \cdot \rho = \rho \circ \mr{Ad}(w)^{-1} : w W_t w^{-1} \to \mr{Aut}_\C (V_\rho)$.
Theorem \ref{thm:2.2} involves a group action of $W$ on the set 
\[
\tilde T = \{ (t,\rho) : t \in T, \rho \in \Irr (W_t) \} .
\]
We call 
\[
T /\!/ W := \tilde T / W
\]
the extended quotient of $T$ by $W$. Thus Theorem \ref{thm:2.2} gives a canonical bijection
\begin{equation}\label{eq:2.2}
T /\!/ W \longleftrightarrow \Irr (\C [X] \rtimes W) = 
\Irr \big( \mc H (\mc R,1) \big) .
\end{equation}
On $\Irr (\C [X] \rtimes W)$ we have the Jacobson topology, whose closed sets are
\[
\{ \pi \in \Irr (\C [X] \rtimes W) : S \subset \ker \pi \}
\quad \text{for } S \subset \C[X] \rtimes W .
\]
Via \eqref{eq:2.2} we transfer this to a topology on $T /\!/ W$. Then the natural maps
\[
\begin{array}{ccc@{\qquad \qquad}ccc}
T / W & \to & T /\!/ W & T /\!/ W & \to & T / W \\
Wt & \mapsto & [t,\triv] & [t,\rho] & \mapsto & Wt
\end{array}
\]
are continuous. The composition of \eqref{eq:2.2} with $T /\!/ W \to T/W$ is just the 
restriction of an irreducible $\C [X] \rtimes W$-representation 
to $\C[X]^W \cong \mc O (T/W)$, in other words, it is the central character map.\\

In the same terms we can analyse graded Hecke algebras with $k = 0$.
As in Paragraph \ref{par:defGHA}, we let $W$ be a finite Coxeter group acting isometrically
on a finite dimensional Euclidean space $\mf a$. When $k = 0$, we do not need a root system
$R \subset \mf a^*$ to construct the algebra
\[
\mh H (\mf t,W,0) = \mc O (\mf a \otimes_\R \C) \rtimes W = S(\mf t^*) \rtimes W .
\]
Considerations with Clifford theory, exactly as for $\mc O (T) \rtimes W$, lead to:

\begin{thm}\label{thm:2.3}
Every irreducible representation of $\mc O (\mf t) \rtimes W$ is of the form
\[
\pi (\nu ,\rho) = \ind_{\mc O (\mf t) \rtimes W_\nu}^{\mc O (\mf t) \rtimes W}
(\C_\nu \otimes V_\rho) \text{ for some } \nu \in \mf t, (\rho,V_\rho) \in \Irr (W_\rho) .
\]
Two such representations $\pi (\nu,\rho)$ and $\pi (\nu',\rho')$ are equivalent if and
only if there exists a $w \in W$ with $\nu' = w (\nu)$ and $\rho' = w \cdot \rho$.
\end{thm}

Like for affine Hecke algebras with $q=1$, $\pi (\nu,\rho)$ is a direct summand of
\[
\ind_{\mc O (\mf t)}^{\mc O (\mf t) \rtimes W} (\C_\nu) = 
\ind_{\mc O (\mf t) \rtimes W_\nu}^{\mc O (\mf t) \rtimes W} (\C_t \otimes \C [W]) .
\]
Also, there is a canonical bijection
\[
\mf t /\!/ W \longleftrightarrow \Irr (\mc O (\mf t) \rtimes W) = 
\Irr \big( \mh H (\mf t,W,0) \big) .
\]

\subsection{Iwahori--Hecke algebras of type $\widetilde{A_1}$} \
\label{par:A1t}

Consider the based root datum
\[
\mc R = \big( X = \Z, R = \{\pm 1\}, Y = \Z, 
R^\vee = \{\pm 2\}, \Delta = \{ \alpha \} = \{ 1 \} \big) .
\]
It has an affine Weyl group $W_\af = W (\mc R)$ of type $\widetilde{A_1}$, with 
Coxeter generators $S_\af = \{ s_\alpha, s'_\alpha : x \mapsto 1 - x \}$.
Affine Hecke algebras of the $\mc H (\mc R,\lambda,\lambda^*,\mb{q})$, for various
parameters $\lambda (\alpha)$ and $\lambda^* (\alpha)$, appear often in the 
representation theory of classical $p$-adic groups \cite{GoRo,MiSt}. We will work
out the irreducible representations of $\mc H (\mc R,\lambda,\lambda^*,\mb{q})$
in detail. To avoid singular cases, we assume throughout this paragraph that
\begin{equation}\label{eq:2.3}
\mb{q}^{\lambda (\alpha)} \neq -1, \mb{q}^{\lambda^* (\alpha)} \neq -1,
\mb{q}^{\lambda (\alpha) + \lambda^* (\alpha)} \neq 1 .
\end{equation}
Recall that the case $\lambda (\alpha) = \lambda^* (\alpha) = 0$ was already
discussed in Paragraph \ref{par:q=1}. We abbreviate
\[
q_1^{1/2} = \mb{q}^{\lambda (\alpha) / 2} ,\; q_0^{1/2} = \mb{q}^{\lambda^* (\alpha) / 2}
\quad \text{and} \quad \mc H = \mc H (\mc R,\lambda,\lambda^*,\mb{q}) .
\]
It is not difficult to see \cite{Mat1} that every irreducible $\mc H$-representation
is a quotient of $\ind_{\C [X]}^{\mc H} (\C_t)$ for some $t \in T$. By \eqref{eq:1.19}
this induced representation has dimension two.
Using the Iwahori--Matsumoto presentation $\mc H (W_\af,q)$, two onedimensional
representations can be written down immediately. Firstly the trivial representation,
given by
\[
\triv(T_{s_\alpha}) = \mb{q}^{\lambda (\alpha)} = q_1, \;
\triv(T_{s'_\alpha}) = \mb{q}^{\lambda^* (\alpha)} = q_0 ,
\]
and secondly the Steinberg representation, defined by
\[
\mr{St}(T_{s_\alpha}) = -1 ,\; \mr{St}(T_{s'_\alpha}) = -1 .
\]
When $\mc H$ is the Iwahori-spherical Hecke algebra of $SL_2$ over a $p$-adic field
$F$, these two representations correspond to the trivial and the Steinberg
representations of $SL_2 (F)$ -- as their names already suggested.

Via evaluation at $1 \in \Z$, we identify $T = \Hom_\Z (X, \C^\times)$ with $\C^\times$.
By Theorem \ref{thm:1.5} $\theta_1 = q_0^{-1/2} q_1^{-1/2} T_{s'_\alpha} T_{s_\alpha}$.
For the trivial and Steinberg representations that means
\begin{align*}
& \triv(\theta_1) = q_0^{-1/2} q_1^{-1/2} \triv(T_{s'_\alpha} T_{s_\alpha}) =
q_0^{1/2} q_1^{1/2} , \\
& \mr{St}(\theta_1) = q_0^{-1/2} q_1^{-1/2} \mr{St}(T_{s'_\alpha} T_{s_\alpha}) = 
q_0^{-1/2} q_1^{-1/2} . 
\end{align*}
Therefore, as $\C[X]$-representations:
\begin{equation}\label{eq:2.4}
\triv |_{\C [X]} = \C_{q_0^{1/2} q_1^{1/2}} \text{ and }
\mr{St} |_{\C [X]} = \C_{q_0^{-1/2} q_1^{-1/2}} . 
\end{equation}

\begin{thm}\label{thm:2.4}
\enuma{
\item The $\mc H$-representation $\ind_{\C[X]}^{\mc H} (\C_t)$ is irreducible for all 
\[
t \in \C^\times \setminus \big\{ q_0^{1/2} q_1^{1/2}, q_0^{-1/2} q_1^{-1/2},
-q_0^{1/2} q_1^{-1/2}, -q_0^{-1/2} q_1^{1/2} \big\} .
\]
\item For $t$ as in part (a), $\ind_{\C[X]}^{\mc H}(\C_t)$ is isomorphic with 
$\ind_{\C[X]}^{\mc H}(\C_{t^{-1}})$. There are no further relations between the 
irreducible representations $\ind_{\C[X]}^{\mc H}(\C_t)$.
\item The algebra $\mc H$ has precisely four other irreducible representations: 
$\triv,\: \mr{St}$ and two that we call $\pi (-1,\triv),\pi (-1,\mr{St})$.
They have dimension one and fit in short exact sequences 
\[
\begin{array}{c*{4}{@{\;\to\;}c}}
0 & \mr{St} & \ind_{\C[X]}^{\mc H}(\C_{q_0^{1/2} q_1^{1/2}}) & \triv & 0 ,\\
0 & \triv & \ind_{\C[X]}^{\mc H}(\C_{q_0^{-1/2} q_1^{-1/2}}) & \mr{St} & 0 ,\\
0 & \pi(-1, \mr{St}) & \ind_{\C[X]}^{\mc H}(\C_{-q_0^{-1/2} q_1^{1/2}}) &
\pi (-1,\triv) & 0 ,\\
0 & \pi(-1, \triv) & \ind_{\C[X]}^{\mc H}(\C_{-q_0^{1/2} q_1^{-1/2}}) &
\pi (-1,\mr{St}) & 0 .
\end{array}
\] }
\end{thm}
\textbf{Remark.} By the conditions \eqref{eq:2.3}, the four special values of $t$
are all different, except that the last two coincide if $q_1 = q_0$.
\begin{proof}
(a) By \eqref{eq:1.19} $\ind_{\C[X]}^{\mc H}(\C_t) = \mc H (W,q)$ as
$\mc H (W,q)$-module. Since $q_1 \neq -1$, the algebra $\mc H (W,q)$ is semisimple,
and isomorphic with $\C[W] \cong \C \oplus \C$. The quadratic relation \eqref{eq:1.20}
points us to the minimal central idempotents in $\mc H (W,q)$:
\[
p_+ := (T_{s_\alpha} + T_e) (1 + q_1)^{-1} \quad \text{and} \quad
p_- := (T_{s_\alpha} - q_1 T_e) (1 + q_1)^{-1} .
\]
Then $\C p_+$ and $\C p_-$ are the only nontrivial $\mc H (W,q)$-submodules of
$\ind_{\C[X]}^{\mc H}(\C_t)$. It follows that, whenever $\ind_{\C[X]}^{\mc H}(\C_t)$
is irreducible as $\mc H$-module, $\C p_+$ or $\C p_-$ is an $\mc H$-submodule.
We test for which $t \in T$ this happens. The cross relation in $\mc H$ gives
\begin{align*}
& \theta_1 (T_{s_\alpha} + T_e) = \\
& \theta_1 + T_{s_\alpha} \theta_{-1} + 
\Big( (\mb{q}^{\lambda (\alpha)} - 1) + \theta_{-\alpha} \big( \mb{q}^{(\lambda (\alpha) +
\lambda^* (\alpha))/2} - \mb{q}^{(\lambda (\alpha) - \lambda^* (\alpha))/2} \big) \Big)
\frac{\theta_1 - \theta_{-1}}{\theta_0 - \theta_{-2}} = \\
& T_{s_\alpha} \theta_{-1} + \theta_1 \big( (q_1 - 1) + \theta_{-1} (q_1^{1/2} q_0^{1/2} -
q_1^{1/2} q_0^{-1/2}) \big) \theta_1 = \\
& T_{s_\alpha} \theta_{-1} + q_1 \theta_1 + q_1^{1/2} ( q_0^{1/2} - q_0^{-1/2}) .
\end{align*}
In $\ind_{\C[X]}^{\mc H}(\C_t)$ we get
\begin{equation}\label{eq:2.6}
\begin{aligned}
\theta_1 (1 + q_1) p_+ & = \theta_1 (T_{s_\alpha} + T_e) =
T_{s_\alpha} \cdot t^{-1} + q_1 t +  q_1^{1/2} ( q_0^{1/2} - q_0^{-1/2}) \\
& = t^{-1} (T_{s_\alpha} + T_e) + \big( t q_1 - t^{-1} + 
q_1^{1/2} ( q_0^{1/2} - q_0^{-1/2}) \big) T_e \\
& = t^{-1} (T_{s_\alpha} + T_e) + q_1^{1/2} (t q_1^{1/2} - t^{-1}q_1^{-1/2} + q_0^{1/2} - 
q_0^{-1/2}) T_e .
\end{aligned}
\end{equation}
This can only be a scalar multiple of $p_+$ if $t q_1^{1/2} - t^{-1}q_1^{-1/2} + q_0^{1/2} - 
q_0^{-1/2} = 0$, and that happens only if $q_0^{1/2} = -t q_1^{1/2}$ or
$q_0^{1/2} = t^{-1} q_1^{-1/2}$.

Similarly we compute in $\mc H$:
\[
\theta_1 (T_{s_\alpha} - q_1 T_e) = T_{s_\alpha} \theta_{-1} - \theta_1 +
q_1^{1/2} ( q_0^{1/2} - q_0^{-1/2}) .
\]
In $\ind_{\C[X]}^{\mc H}(\C_t)$ that leads to
\begin{equation}\label{eq:2.7}
\begin{aligned}
\theta_1 (1+ q_1) p_- & = \theta_1 (T_{s_\alpha} - q_1 T_e) =
T_{s_\alpha} \cdot t^{-1} - t + q_1^{1/2} ( q_0^{1/2} - q_0^{-1/2}) \\
& = t^{-1} (T_{s_\alpha} - q_1 T_e) + \big( q_1 t^{-1} - t + 
q_1^{1/2} ( q_0^{1/2} - q_0^{-1/2}) \big) T_e \\
& = t^{-1} (T_{s_\alpha} - q_1 T_e) + q_1^{1/2} ( t^{-1} q_1^{1/2} - t q_1^{-1/2} + 
q_0^{1/2} - q_0^{-1/2} ) T_e .
\end{aligned}
\end{equation}
If this is a scalar multiple of $p_-$, then $t^{-1} q_1^{1/2} - t q_1^{-1/2} + 
q_0^{1/2} - q_0^{-1/2} = 0$, which means that $q_0^{1/2} = -t^{-1} q_1^{1/2}$ or
$q_0^{1/2} = t q_1^{-1/2}$.

We conclude that for 
\[
t \in \C^\times \setminus \big\{ q_0^{1/2} q_1^{1/2}, q_0^{-1/2} q_1^{-1/2},
-q_0^{1/2} q_1^{-1/2}, -q_0^{-1/2} q_1^{1/2} \big\} 
\]
neither $\C p_+$ nor $\C p_-$ is an $\mc H$-submodule of $\ind_{\C[X]}^{\mc H}(\C_t)$,
so that $\ind_{\C[X]}^{\mc H}(\C_t)$ is irreducible. On the other hand, when $t$
equals one of these special values, the above calculations in combination with the
fact that $\theta_1$ generates $\C [X]$ imply that $\ind_{\C[X]}^{\mc H}(\C_t)$ does
have a onedimensional $\mc H$-submodule.\\
(b) Consider the element
\[
f_\alpha := \frac{\theta_1 (q_1 - 1) + q_1^{1/2} (q_0^{1/2} - q_0^{-1/2})}{
\theta_1 - \theta_{-1}} 
\]
of the quotient field $\C (X)$ of $\C [X] = \mc O (T)$. It lies in the version of
$\mc H = \mc H (T,\lambda,\lambda^*,\mb{q})$ obtained from Definition \ref{def:1.8}
by replacing $\mc O (T)$ with $\C (X)$. By direct calculation in that algebra:
\[
\theta_x (T_{s_\alpha} - f_\alpha) = (T_{s_\alpha} - f_\alpha) \theta_x
\quad \text{for all } x \in X = \Z .
\]
Although $f_\alpha \notin \C [X]$, it has a well-defined action on 
$\ind_{\C[X]}^{\mc H}(\C_t)$, provided that $(\theta_1 - \theta_{-1})(t) \neq 0$,
or equivalently $t \notin \{1,-1\}$. For $v \in \C_t \setminus \{0\}$, the element
\[
(T_{s_\alpha} - f_\alpha) v \in \ind_{\C[X]}^{\mc H}(\C_t) \setminus \{0\}
\]
satisfies
\[
\theta_x (T_{s_\alpha} - f_\alpha) v = (T_{s_\alpha} - f_\alpha) \theta_{-x} v =
(T_{s_\alpha} - f_\alpha) \theta_{-x}(t) v = t^{-1}(x) (T_{s_\alpha} - f_\alpha) v .
\]
Hence as $\C[X]$-modules
\begin{equation}\label{eq:2.5}
\ind_{\C[X]}^{\mc H}(\C_t) = \C_t \oplus \C_{t^{-1}}
\quad \text{for all } t \in \C^\times \setminus \{1,-1\} .
\end{equation}
By Frobenius reciprocity, for such $t$:
\begin{align*}
\Hom_{\mc H} \big( \ind_{\C[X]}^{\mc H}(\C_{t^{-1}}),\ind_{\C[X]}^{\mc H}(\C_t) \big)
& \cong \Hom_{\C [X]} \big( \C_{t^{-1}},\ind_{\C[X]}^{\mc H}(\C_t) \big) \\
& = \Hom_{\C [X]} \big( \C_{t^{-1}},\C_t \oplus \C_{t^{-1}} \big) \; \cong \; \C .
\end{align*}
In particular this shows that 
\[
\ind_{\C[X]}^{\mc H}(\C_{t^{-1}}) \cong \ind_{\C[X]}^{\mc H}(\C_t)
\]
whenever these representations are irreducible (for $t = \pm 1$ this is isomorphism
is tautological).\\
(c) From \eqref{eq:2.6} we know that $\ind_{\C[X]}^{\mc H}(\C_{q_0^{1/2} q_1^{1/2}})$
has a subrepresentation $\C p_-$ with $\theta_1 p_- = q_0^{-1/2} q_1^{-1/2} p_-$.
Further, by \eqref{eq:1.20}
\begin{equation}\label{eq:2.8}
(T_{s_\alpha} + T_e) p_- = 0 ,\quad \text{so } T_{s_\alpha} p_- = -p_- .
\end{equation}
As $T_{s_\alpha}$ and $\theta_1$ generate $\mc H$, it follows that here $\C p_-$ is the
Steinberg representation. By \eqref{eq:2.5} 
\[
\ind_{\C[X]}^{\mc H}(\C_{q_0^{1/2} q_1^{1/2}}) / \C p_- \cong
\C_{q_0^{1/2} q_1^{1/2}}
\]
as $\C[X]$-representation. Also
\[
\ind_{\C[X]}^{\mc H}(\C_{q_0^{1/2} q_1^{1/2}}) / \C p_- \cong
\mc H (W,q) / \C p_- \cong \C p_+
\]
as $\mc H (W,q)$-representation. From $(T_{s_\alpha} - q_1 T_e) p_+ = 0$ we see that
$T_{s_\alpha} p_+ = q_1 p_+$. Again, since $\theta_1$ and $T_{s_\alpha}$ generate $\mc H$,
we can conclude that $\ind_{\C[X]}^{\mc H}(\C_{q_0^{1/2} q_1^{1/2}}) / \C p_-$ is the
trivial representation of $\mc H$.

The calculations around \eqref{eq:2.7} show that 
$\ind_{\C[X]}^{\mc H}(\C_{-q_0^{-1/2} q_1^{1/2}})$ contains a subrepresentation
$\pi (-1,\mr{St}) = \C p_-$ with 
\begin{equation}\label{eq:2.11}
\theta_1 p_- = -q_0^{1/2} q_1^{-1/2} p_-.
\end{equation} 
By \eqref{eq:2.8} it has the same restriction to $\mc H (W,q)$ as the Steinberg 
representation, which explains our notation $\pi (-1,\mr{St})$ for this $\C p_-$. The quotient
\begin{equation}\label{eq:2.12}
\ind_{\C[X]}^{\mc H}(\C_{-q_0^{-1/2} q_1^{1/2}}) / \C p_- \text{ equals }
\C_{-q_0^{-1/2} q_1^{1/2}} \text{ as }\C [X]\text{-representation.}
\end{equation} 
As $\mc H (W,q)$-representation it is isomorphic to $\C p_+ \cong \triv$, 
and therefore we write
\[
\pi(-1,\triv) = \ind_{\C[X]}^{\mc H}(\C_{-q_0^{1/2} q_1^{-1/2}}) / \C p_- .
\]
Analogous considerations apply to $ \ind_{\C[X]}^{\mc H}(\C_{q_0^{-1/2} q_1^{-1/2}})$
and $\ind_{\C[X]}^{\mc H}(\C_{-q_0^{1/2} q_1^{-1/2}})$.
\end{proof}

An important special case is $q_1 = q_0 = \mb{q}$. When $F$ is a non-archimedean local
field with residue field of order $\mb{q}$, $\mc H (W_\af,\mb{q}) = \mc H 
(\mc R,\mb{q})$ arises as the Iwahori-spherical Hecke algebra of $SL_2 (F)$. 

The algebra $\mc H (\mc R,\mb{q})$ is the simplest example of a true affine Hecke algebra, 
not isomorphic to some more elementary kind of algebra. For this algebra Theorem 
\ref{thm:2.4} says:
\begin{itemize}
\item The $\mc H (\mc R,\mb{q})$-representation $\ind_{\C[X]}^{\mc H (\mc R,\mb{q})}(\C_t)$ 
is irreducible for all \\ $t \in \C^\times \setminus \{\mb{q},\mb{q}^{-1},-1\}$.
\item $\ind_{\C[X]}^{\mc H (\mc R,\mb{q})}(\C_t) \cong \ind_{\C[X]}^{\mc H (\mc R,\mb{q})}
(\C_{t^{-1}})$ for all $t \in \C^\times \setminus \{\mb{q},\mb{q}^{-1}\}$.
\item $\ind_{\C[X]}^{\mc H (\mc R,\mb{q})}(\C_{-1}) =
\pi (-1,\triv) \oplus \pi (-1,\mr{St})$, with $\pi (-1,\triv)$ and 
$\pi (-1,\mr{St})$ irreducible and inequivalent.
\item There are only two other irreducible $\mc H (\mc R,\mb{q})$-representations,
triv and St, which both occur as subquotients of $\ind_{\C[X]}^{\mc H (\mc R,\mb{q})}
(\C_{\mb{q}})$ and of $\ind_{\C[X]}^{\mc H (\mc R,\mb{q})}(\C_{\mb{q}^{-1}})$.  
\end{itemize}
This classification works for almost all $\mb q \in \C^\times$, only 1 and $-1$ are 
exceptional. (In view of \eqref{eq:1.25}, that is hardly surprising.) For $\mc H (\mc R,-1)$
the trivial representation coincides with $\pi (-1,\triv)$ and the Steinberg representation
coincides with $\pi (-1,\mr{St})$. Consequently $\mc H (\mc R,-1)$ has only two onedimensional 
representations. Apart from that, the above statements are valid when $\mb q = -1$.

Although the definition and formulas for $\mc H (\mc R,\mb{q})$ look considerably
simpler than those for $\mc H (\mc R,\lambda,\lambda^*,\mb{q})$, in the end the latter
is hardly more difficult. In terms of the induced representations 
$\ind_{\C[X]}^{\mc H}(\C_t)$, the only differences occur when 
$t \in \big\{ -\!1, -q_0^{1/2} q_1^{-1/2}, -q_0^{-1/2} q_1^{1/2} \big\}$.

\subsection{Affine Hecke algebras of type $GL_n$} \
\label{par:GLn}

The root datum of type $GL_n$ is 
\[
\mc R_n = (\Z^n, A_{n-1}, \Z^n, A_{n-1}, \Delta_{n-1})
\]
with $A_{n-1} = \{ e_i - e_j : 1 \leq i,j \leq n, i \neq j\}$ and
$\Delta_{n-1} = \{ e_i - e_{i+1} : i = 1,2,\ldots n-1\}$.
Via $t \mapsto (t(e_1),\ldots,t(e_n))$ we identify $T$ with $(\C^\times)^n$.
We saw in Section \ref{sec:def} that
\begin{align*}
& W_\af = \{ x \in \Z^n : x_1 + \cdots + x_n = 0 \} \rtimes S_n ,\\
& S_\af = \{ s_i = s_{e_i - e_{i+1}} : i = 1,\ldots,n-1 \} \cup
\{ s_0 \} ,\\
& \Omega = \langle \omega \rangle \cong \Z, \quad \omega(x) = e_1 + (1\,2\cdots n) x ,
\end{align*}
where $s_0 (x) = x + (1 - \inp{x}{e_1 - e_n}) (e_1 - e_n)$.
All the simple affine reflections from $S_\af$ are $\Omega$-conjugate,
so $q_s = q_{s'}$ for $s,s' \in S_\af$. Call this parameter $\mb{q}$ and
consider the affine Hecke algebra 
\[
\mc H_n (\mb{q}) := \mc H (\mc R_n,\mb{q})  
\]
of type $GL_n$. The primary importance of such algebras is that they describe
every Bernstein block in the representation theory of $GL_n (F)$ for a 
non-archimedean local field $F$ \cite{BuKu1}. In all those cases $\mb{q}$ is a
power of a prime number, but just as affine Hecke algebra that is not necessary.
We do not even have to require that $\mb{q} \in \R_{>1}$, the representation theory
of $\mc H_n (\mb{q})$ looks the same for every $\mb{q} \in \C^\times$ which is not
a root of unity. 

The irreducible representations of $GL_n (F)$ were classified (in terms of 
supercuspidal representations) by Zelevinsky \cite{Zel,BeZe}. That classification
follows a combinatorial pattern involving certain "segments, and the irreducible
representations of $\mc H_n (\mb{q})$ exhibit the same pattern. We formulate 
their classification in terms intrinsic to Hecke algebras. The algebra 
$\mc H_1 (\mb{q}) = \C [\Z]$ has already been discussed, so we may assume that
$n \geq 2$ (which ensures that $S_\af$ is nonempty).

Like in Paragraph \ref{par:A1t} we start with the trivial and the Steinberg
representations. By definition 
\begin{equation}\label{eq:2.10}
\triv(T_s) = \mb{q} ,\; \mr{St}(T_s) = -1 \quad \text{for all } s \in S_\af.
\end{equation}
This does not yet determine the values of these representations on $T_\omega$,
to fix that one requires in addition
\[
\begin{array}{ccc@{ ,\qquad}ccc}
\triv|_{\C [X]} & = & \C_{t_+} & t_+ & = & \big( \mb{q}^{(n-1)/2},\mb{q}^{(n-3)/2},
\ldots , \mb{q}^{(1-n)/2} \big) , \\
\mr{St}|_{\C [X]} & = & \C_{t_-} & t_- & = & \big( \mb{q}^{(1-n)/2},\mb{q}^{(3-n)/2},
\ldots , \mb{q}^{(n-1)/2} \big) . 
\end{array}
\]
The terminology is motivated by Iwahori-spherical representation of $GL_n$ over a
$p$-adic field $F$: triv and St correspond to the epynomous representations of $GL_n (F)$.
 
In contrast with the Iwahori--Hecke algebra of type $\widetilde{A_1}$, the trivial
and Steinberg representations of $\mc H_n (\mb{q})$ come in a families of representations,
parametrized by $\C^\times$. This is implemented as follows. For $z \in \C^\times$ we
put $t_z = (z,z,\ldots,z) \in T^{S_n}$. We define the $\mc H_n (\mb{q})$-representation
$\triv \otimes t_z$ by \eqref{eq:2.10} and
\begin{equation}\label{eq:2.13}
(\triv \otimes t_z) |_{\C [X]} = \C_{t_+ t_z} .
\end{equation}
This is possible because $\alpha (t_z) = 1$ for all $\alpha \in R$, so that
$t_z \in \Hom (X,C^\times)$ is trivial on $X \cap W_\af$. Similarly we define the
onedimensional $\mc H_n (\mb{q})$-representation $\mr{St} \otimes t_z$, by 
requiring \eqref{eq:2.10} and 
\begin{equation}\label{eq:2.14}
(\mr{St} \otimes t_z) |_{\C [X]} = \C_{t_- t_z} . 
\end{equation}
For $\mc H_2 (\mb{q})$ the above onedimensional representations, in combination with the
induced representations $\ind_{\C[X]}^{\mc H_2 (\mb{q})} (\C_t)$, already exhaust 
$\Irr ( \mc H_2 (\mb{q}))$. With arguments very similar to those in Paragraph \ref{par:A1t}
one can show:

\begin{thm}\label{thm:2.5}
The $\mc H_2 (\mb{q})$-representation $\ind_{\C[X]}^{\mc H_2 (\mb{q})} (\C_t)$ is 
irreducible for all $t \in T$ that are not of the form $t_+ t_z = (\mb{q}^{1/2} z,
\mb{q}^{-1/2}z)$ or $t_- t_z = (\mb{q}^{-1/2} z, \mb{q}^{1/2} z)$.

This representation is isomorphic to $\ind_{\C[X]}^{\mc H_2 (\mb{q})} (\C_{t^{-1}})$,
but apart from that there are no relations between the irreducible representations of the
form $\ind_{\C[X]}^{\mc H_2 (\mb{q})} (\C_t)$.

The only other irreducible $\mc H_2 (\mb{q})$-representations are $\triv \otimes t_z$
and $\mr{St} \otimes t_z$ with $z \in \C^\times$. They are mutually inequivalent and fit
in short exact sequences
\[
\begin{array}{c*{4}{@{\;\to\;}c}}
0 & \mr{St} \otimes t_z & \ind_{\C[X]}^{\mc H_2 (\mb{q})} (\C_{(\mb{q}^{1/2} z,
\mb{q}^{-1/2}z)}) & \triv \otimes t_z & 0 , \\
0 & \triv \otimes t_z & \ind_{\C[X]}^{\mc H_2 (\mb{q})} (\C_{(\mb{q}^{-1/2} z,
\mb{q}^{1/2}z)}) & \mr{St} \otimes t_z & 0 .
\end{array}
\]
\end{thm}

The irreducible representations from Theorem \ref{thm:2.5} come in three kinds, but there is
a natural way to gather them in two families:
\begin{itemize}
\item The twists $\mr{St} \otimes t_z$ of the Steinberg representation, parametrized by
$t_z \in T^{S_2}$.
\item A family parametrized by $T / S_2$, which for almost all $Wt \in T / W$ has the
member $\ind_{\C[X]}^{\mc H_2 (\mb{q})} (\C_t)$. When that representation happens to be
reducible, we take the unique representative $t$ of $Wt$ with $|\alpha (t)| \geq 1$ and we 
assign to $Wt$ the unique irreducible quotient of $\ind_{\C[X]}^{\mc H_2 (\mb{q})} (\C_t)$.
\end{itemize}
The irreducible representations of $\mc H_n (\mb{q})$ with $n \geq 2$ can be parametrized
with a longer list of similar families. We sketch the construction and classification,
as translated from \cite{Zel}, stepwise.
\begin{itemize}
\item Choose a partition $\vec{n} = (n_1,n_2,\ldots,n_d)$ of $n$, where $n_i \geq 1$
(but the sequence $(n_i)$ need not be monotone).
\item The algebra 
\[
\bigotimes_{i=1}^d \mc H_{n_i}(\mb{q}) = \bigotimes_{i=1}^d 
\mc H \big( (\Z^{n_i},A_{n_i -1} ,\Z^{n_i}, A_{n_i -1}, \Delta_{n_i-1}), \mb{q} \big)
\]
is naturally a subalgebra of $\mc H_n (\mb{q})$, with the same commutative subalgebra
$\C [X] = \C [\Z^n] = \otimes_{i=1}^n \C [\Z^{n_i}]$.
\item For every $i$ we pick 
\[
z_i \in \Irr (\C [\Z^{n_i}])^{S_{n_i}} \cong \big( (\C^\times)^{n_i} \big)^{S_{n_i}} \cong \C^\times ,
\]
and we construct the irreducible $\mc H_{n_i}(\mb{q})$-representation
$\mr{St} \otimes z_i$. (It corresponds to a segment in Zelevinsky's setup.)
\item Then $\boxtimes_{i=1}^d (\mr{St} \otimes z_i)$ is an irreducible representation of
$\bigotimes_{i=1}^d \mc H_{n_i}(\mb{q})$ and
\[
\pi (\vec{n},\vec{z}) := \ind_{\bigotimes_{i=1}^d \mc H_{n_i}(\mb{q})}^{\mc H_n (\mb{q})}
\Big( \boxtimes_{i=1}^d ( \mr{St} \otimes z_i ) \big)
\]
is an $\mc H_n (\mb{q})$-representation. 
\item For almost all $\vec{z} = (z_i )_{i=1}^d$, $\pi (\vec{n},\vec{z})$ is irreducible and
depends only on the orbit of $\vec{z}$ under $N_{S_n} \big( \prod_{i=1}^d S_{n_i} \big)$.
Here $\prod_{i=1}^d S_{n_i}$ fixes $(z_i )_{i=1}^d$ and
\[
N_{S_n} \Big( \prod_{i=1}^d S_{n_i} \Big) \Big/ \prod_{i=1}^d S_{n_i} \cong
\prod_{m \geq 1} S ( \{ i : n_i = m \} ) .
\]
\item When $\pi (\vec{n},\vec{z})$ is reducible, we pick a representative $\vec{z}$ for
$N_{S_n}(\prod_{i=1}^d S_{n_i}) \vec{z}$ such that $|z_i| \geq |z_j|$ whenever
$n_i = n_j$ and $i \leq j$.
\item For such $\vec{z}$ it follows from the Langlands classification that  
$\pi (\vec{n},\vec{z})$ has a unique irreducible quotient. That is the irreducible
$\mc H_n (\mb{q})$-representation associated with $(\vec{n},\vec{z})$.
\item The irreducible $\mc H_n (\mb{q})$-representations assigned to 
$(\vec{n},\vec{z})$ and $(\vec{n}',\vec{z}')$ are equivalent if and only if
$(\vec{n},\vec{z})$ and $(\vec{n}',\vec{z}')$ are $S_n$-associate.
\end{itemize}
In the above parametrization $z_i \in \big( (\C^\times)^{n_i} \big)^{S_{n_i}}$, so
$\vec{z}$ can be regarded as a diagonal matrix in $GL_n (\C)$. The partition $\vec{n}$
determines a Levi subgroup
\[
M = GL_{n_1}(\C) \times \cdots \times GL_{n_d}(\C) \quad \text{of} \quad GL_n (\C) ,
\]
and $\vec{z} \in Z(M)$. Let $u_m$ be a regular unipotent element of $GL_m (\C)$
(it is unique up to conjugation). Then 
\[
\vec{u} := (u_{n_1},u_{n_2}, \ldots, u_{n_d})
\]
is a regular unipotent element of $M$, and $\vec{z} \vec{u}$ is the Jordan decomposition
of an element of $GL_n (\C)$. Up to conjugacy, every element of $GL_n (\C)$ has this shape,
for some $\vec{z}$ (unique up to $S_n$-association) and $\vec{u}$. This setup leads to:

\begin{thm}\label{thm:2.6}
There exists a canonical bijection between the following sets:
\begin{itemize}
\item conjugacy classes in $GL_n (\C)$,
\item $S_n$-association classes of data $(\vec{n},\vec{z})$, where
$\vec{n} = (n_i)$ is a partition of $n$ and $z_i \in \big( (\C^\times)^{n_i}\big)^{S_{n_i}}$,
\item $\Irr (\mc H_n (\mb{q}))$.
\end{itemize}
\end{thm}

\section{Representation theory}
\label{sec:rep}

\subsection{Parabolic induction} \
\label{par:induction}

In the representation theory of reductive groups, a pivotal role is played by
parabolic induction. An analogous operation exists for affine Hecke algebras, and
it will be crucial in large parts of this paper. 

Given a based root datum $\mc R = (X,R,Y,R^\vee,\Delta)$ and a subset 
$P \subset \Delta$, we can form the based root datum
\[
\mc R^P = (X,R_P, Y,R_P^\vee,P).
\]
Here $R_P = \Q P \cap R$ is a standard parabolic root subsystem of $R$, with dual
root system $R_P^\vee = \Q P^\vee \cap R^\vee$. We record the special cases
\[
\mc R^\Delta = \mc R \quad \text{and} \quad
\mc R^\emptyset = (X,\emptyset,Y,\emptyset,\emptyset) .
\]
Let $W_P$ be the Weyl group of $R_P$. Any parameter functions $\lambda,\lambda^*$
for $\mc R$ restrict to parameter functions for $\mc R^P$, and 
\[
\mc H^P = \mc H (\mc R^P,\lambda,\lambda^*,\mb{q})
\]
is a subalgebra of $\mc H = \mc H (\mc R,\lambda,\lambda^*,\mb{q})$.
This corresponds to the notion of a parabolic subgroup $\mc P$ of a reductive
group $\mc G$, and simultaneously to the notion of a Levi factor of $\mc P$--
for affine Hecke algebras the unipotent radical of $\mc P$ is more or less
automatically divided out.

Notice that $\mc H$ and $\mc H^P$ share the same commutative subalgebra
$\C [X] = \mc O (T)$. By \eqref{eq:1.19}, as vector spaces
\begin{equation}\label{eq:3.2}
\mc H = \mc H (W,q) \underset{\mc H (W_P,q)}{\otimes} \mc H(W_P,q) \otimes_\C \C[X]
= \mc H (W,q) \underset{\mc H (W_P,q)}{\otimes} \mc H^P .
\end{equation}
Parabolic induction for representations of affine Hecke algebras is the functor
\[
\ind_{\mc H^P}^{\mc H} : \Mod (\mc H^P) \to \Mod (\mc H) ,
\]
for any parabolic subalgebra $\mc H^P$ of $\mc H$. We also have a "parabolic 
restriction" functor, that is just restriction of $\mc H$-modules to $\mc H^P$.

The link with parabolic induction for reductive $p$-adic groups is made precise 
in \cite[\S 4.1]{SolComp}. In that setting parabolic restriction for Hecke
algebras corresponds to the Jacquet restriction functor (but with respect to
a parabolic subgroup opposite to the one used for induction).

In practice we often precompose parabolic induction with inflations of 
representations from a quotient algebra of $\mc H^P$. To define it, we write
\[
\begin{array}{l@{\qquad}l}
X_P = X \big/ \big( X \cap (P^\vee)^\perp \big) & Y^P = Y \cap P^\perp \\
X^P = X / ( X \cap \Q P) & Y_P = Y \cap \Q P^\vee \\
\mc R_P = (X_P,R_P,Y_P,R_P^\vee,P) .
\end{array}
\]
In terms of reductive groups, the semisimple root datum $\mc R_P$ corresponds to
the maximal semisimple quotient of a parabolic subgroup $\mc P$ of $\mc G$.
The parameter functions $\lambda$ and $\lambda^*$ remain well-defined for $\mc R_P$,
so there is an affine Hecke algebra
\[
\mc H_P = \mc H (\mc R_P,\lambda,\lambda^*,\mb{q}) .
\]
In particular we have
\[
\mc H_\emptyset = \mc H (0,\emptyset,0,\emptyset,\emptyset, \lambda,\lambda^*,\mb{q}) 
= \C \quad \text{and} \quad 
\mc H_\Delta = \mc H (X_\Delta,R,Y_\Delta,R^\vee,\Delta, \lambda,\lambda^*,\mb{q}) .
\]
From Definition \ref{def:1.3} we see that the quotient map
\[
X \to X_P : x \mapsto x_P
\]
induces a surjective algebra isomorphism
\begin{equation}\label{eq:3.1}
\mc H^P \to \mc H_P : \theta_x T_w \mapsto \theta_{x_P} T_w .
\end{equation}
Via this quotient map we will often (implicitly) inflate $\mc H_P$-representations
to $\mc H^P$-representations. To incorporate representations of $\mc H^P$ that are
nontrivial on $\{ \theta_x : x \in X \cap (P^\vee)^\perp \}$, we need more 
flexibility. Write
\[
T_P = \Hom_\Z (X_P,\C^\times), \qquad T^P = \Hom_\Z (X^P ,\C^\times) ,
\]
so that $T_P T^P = T$ and $T_P \cap T^P$ is finite. Every $t \in T^P$ gives rise
to an algebra automorphism
\[
\begin{array}{cccc}
\psi_t : & \mc H^P & \to & \mc H^P \\
& \theta_x T_w & \mapsto & x(t) \theta_x T_w 
\end{array} .
\]
This operation corresponds to twisting representations of a reductive group by 
an unramified character. For any $\mc H_P$-representation $(\pi,V)$ and any
$t \in T^P$, we can form the $\mc H_P$-representation
\[
\psi_t^* \mr{inf}^P (\pi) = \pi \circ \psi_t : 
\theta_x T_w \mapsto \pi (x(t) \theta_{x_P} T_w) .
\]
\begin{defn}\label{def:3.1}
The parabolically induced representation associated to $P \subset \Delta$,
$(\sigma,V_\sigma) \in \Mod (\mc H_P)$ and $t \in T^P$ is
\[
\pi (P,\sigma,t) = \ind_{\mc H^P}^{\mc H}(\sigma \circ \psi_t) .
\]
\end{defn}
By \eqref{eq:3.2} the vector space underlying $\pi (P,\sigma,t)$ is
\[
\mc H (W,q) \underset{\mc H (W_P,q)}{\otimes} V_\sigma ,
\]
and it has dimension $[W : W_P] \dim (V_\sigma)$. Let us discuss how parabolic
induction works out for the subalgebra $\C [X]$ of $\mc H$.

\begin{defn}\label{def:3.2}
Let $(\pi,V)$ be an $\mc H$-representation. For $t \in T$ we write
\[
V_t = \big\{ v \in V \mid \exists N \in \N : (\pi (\theta_x) - x(t))^N v = 0 \big\} . 
\]
When $V_t \neq 0$, we call $t$ a $\C[X]$-weight of $\pi$, and $V_t$ its 
generalized weight space. We denote the set of $\C[X]$-weights of $(\pi,V)$ by
Wt$(\pi)$ or Wt$(V)$.
\end{defn}

For all $t \in \mr{Wt}(\pi)$ the space $V_t$ contains a vector $v' \neq 0$ with
\[
\pi (\theta_x) v'  = x(t) v' \qquad \forall x \in X,
\]
which explains why we call $V_t$ a weight space and not a generalized weight space.
If $V$ has finite dimension, then we can triangularize the commuting operators
$\pi (\theta_x)$ simultaneously, and we find
\begin{equation}\label{eq:3.32}
V = \bigoplus\nolimits_{t \in T} V_t .
\end{equation}
An infinite dimensional $\mc H$-representation does not necessarily have any
$\C[X]$-weight.
 
Recall that the centre of $\mc H$ is $Z(\mc H) = \C[X]^W = \mc O(T/W)$. Hence,
whenever $\pi$ admits a central character cc$(\pi)$, we have
\[
\mr{cc}(\pi) = W \mr{Wt}(\pi) \in T / W.
\]
The set of $\C[X]$-weights of a representation behaves well under parabolic 
induction. To describe the effect, let $W^P$ be the set of shortest length
representatives for $W / W_P$.

\begin{lem}\label{lem:3.3}
\enuma{
\item Let $\pi$ be a finite dimensional $\mc H^P$-representation. Then the 
$\C[X]$-weights of $\ind_{\mc H^P}^{\mc H}(\pi)$ are the elements 
$w(t)$ with $t \in \mr{Wt}(\pi)$ and $w \in W^P$.
\item Let $\sigma$ be a finite dimensional $\mc H_P$-representation and let
$s \in T^P$. Then 
\[
\mr{Wt}(\pi (P,\sigma,s)) = \{ w (st) : t \in \mr{Wt}(\sigma), w \in W^P \} .
\]
}
\end{lem}
\begin{proof}
(a) is a consequence of the proof of \cite[Proposition 4.20]{Opd-Sp}.\\
(b) Cleary Wt$(\sigma \circ \psi_s) = s \mr{Wt}(\sigma)$. Combine with part (a).
\end{proof}

Since $\mc H$ has finite rank as a module over its centre $\C[X]^W$, every
irreducible $\mc H$-representation has finite dimension. Hence $\pi$
admits at least one $\C[X]$-weight, say $t$. By Frobenius reciprocity
\[
\Hom_{\mc H}(\ind_{\C[X]}^{\mc H}(\C_t),\pi) \cong \Hom_{\C[X]} (\C_t,\pi) \neq 0. 
\]
We conclude that:
\begin{cor}\label{cor:3.4}
Every irreducible $\mc H$-representation is a quotient of 
$\ind_{\C[X]}^{\mc H}(\C_t)$ for some $t \in T$.
\end{cor}

Most of the time $\ind_{\C[X]}^{\mc H}(\C_t)$ is itself irreducible. To make
that precise, consider the following rational functions on $T$:
\begin{equation}\label{eq:3.13}
c_\alpha = \frac{\big( \theta_\alpha - \mb{q}^{(- \lambda^* (\alpha) - \lambda (\alpha))/2}\big)
\big( \theta_\alpha + \mb{q}^{(\lambda^* (\alpha) - \lambda (\alpha))/2}\big)}{
(\theta_\alpha - 1) (\theta_\alpha + 1)} \qquad \alpha \in R.
\end{equation}
There are a few ways in which $c_\alpha$ can simplify:
\begin{itemize}
\item if $\lambda (\alpha) = \lambda^* (\alpha) = 0$, then $c_\alpha = 1$,
\item if $\lambda (\alpha) = \lambda^* (\alpha) \neq 0$, then 
$c_\alpha = (\theta_\alpha - \mb{q}^{-\lambda (\alpha)}) (\theta_\alpha - 1)^{-1}$.
\end{itemize}
\begin{thm}\label{thm:3.29}
\textup{\cite[Theorem 2.2]{Kat}} \\
Let $t \in T$. The $\mc H$-representation $\ind_{\C[X]}^{\mc H}(\C_t)$ is 
irreducible if and only if
\begin{itemize}
\item $c_\alpha (t) \neq 0$ for all $\alpha \in R$ and
\item $W_t$ is generated by $\{ s_\alpha : \alpha \in R, s_\alpha (t) = t, c_\alpha^{-1}(t) = 0\}$.
\end{itemize}
\end{thm}

The parabolically induced representations
\begin{equation}\label{eq:3.3}
\ind_{\C[X]}^{\mc H}(\C_t) = \pi (\emptyset ,\triv,t) \qquad t \in T
\end{equation}
can all be realized on the same vector space $\mc H(W,q)$. In fact, they are
isomorphic to $\mc H (W,q)$ as $\mc H(W,q)$-modules. In principle, the entire
representation theory of $\mc H$ can be uncovered by analysing the family of
representations \eqref{eq:3.3}. We already did that successfully for $W_\af$
of type $\widetilde{A_1}$ in paragraph \ref{par:A1t}. However, this direct
approach is very difficult in general. Indeed, while the irreducible representations
of $\mc H$ have been classified in several ways, the finer structure of 
$\ind_{\C[X]}^{\mc H}(\C_t)$ (e.g. a Jordan--H\"older sequence or the multiplicity
with which irreducible representations appear) is not always known.

\subsection{Tempered representations} \
\label{par:temp}

An admissible representation of a reductive group $G$ over a local field is tempered
if all its matrix coefficents have moderate growth on $G$, see \cite[\S III.2]{Wal}
and \cite[\S VII.11]{Kna}. This notion has several uses:
\begin{itemize}
\item the irreducible tempered $G$-representations form precisely the support of
the Plancherel measure of $G$,
\item the Langlands classification of irreducible admissible $G$-representations in
terms of irreducible tempered representations of Levi subgroups of $G$,
\item for general harmonic analysis on $G$, e.g. the Plancherel isomorphism.
\end{itemize}
Analogous of all these well-known results have been established for affine Hecke 
algebras, see \cite{Opd-Sp,DeOp1}. In this paragraph we will discuss the second
of the above three items.

Recall that every (finite dimensional) $\mc H$-module $(\pi,V)$ has a set of 
$\C[X]$-weights Wt$(\pi) \subset T = \Hom_\Z (X,\C^\times)$. To formulate the 
condition for temperedness in terms of weights, it will be convenient to abbreviate
\[
\mf a = Y \otimes_\Z \R ,\; \mf t = Y \otimes_\Z \C = \mr{Lie}(T) ,\; 
\mf a^* = X \otimes_\Z \R ,\; \mf t^* = X \otimes_\Z \C .
\]
The complex torus $T$ admits a polar decomposition 
\begin{equation}\label{eq:3.4}
T = \Hom_\Z (X,S^1) \times \Hom_\Z (X,\R_{>0}) = T_\un \times \exp (\mf a) .
\end{equation}
Here the unitary part $T_\un$ is the maximal compact subgroup of $T$ and the positive
part $\exp (\mf a)$ is the identity component of the maximal real split subtorus of $T$.
Notice that
\[
\mr{Lie}(T_\un) = i \mf a = Y \otimes_\Z i \R \subset \mf t = \mr{Lie}(T) . 
\]
For any $t \in T$ we write $|t|$ for the homomorphism $x \mapsto |t(x)|$. Then
$t = t \, |t|^{-1} \; |t|$ is the polar decomposition of $t$.

The acute positive cones in $\mf a$ are
\[
\begin{array}{lll}
\mf a^+ & = & \{ \nu \in \mf a : \inp{\alpha}{\nu} \geq 0 \; \forall \alpha \in \Delta \} ,\\
\mf a^{++} & = & \{ \nu \in \mf a : \inp{\alpha}{\nu} > 0 \; \forall \alpha \in \Delta \} .
\end{array}
\]
We define $\mf a^{*,+}$ and $\mf a^{*,++}$ similarly, so in particular $X^+ = X \cap \mf a^{*,+}$.
Next we we have the obtuse negative cones in $\mf a$:
\begin{equation}\label{eq:3.6}
\begin{array}{lll}
\mf a^- & = & \{ \nu \in \mf a : \inp{\delta}{\nu} \leq 0 \; \forall \delta \in \mf a^{*,+} \} 
= \{ \sum_{\alpha \in \Delta} x_\alpha \alpha^\vee : x_\alpha \leq 0 \} , \\
\mf a^{--} & = & \{ \nu \in \mf a : \inp{\delta}{\nu} < 0 \; \forall \delta \in \mf a^{*,+} \setminus \{0\} \} .
\end{array}
\end{equation}
Via the exponential map $\exp : \mf t \to T$ we get
\[
T^+ = \exp (\mf a^+) ,\; T^{++} = \exp (\mf a^{++}) ,\; T^- = \exp (\mf a^-) ,\; T^{--} = \exp (\mf a^{--}) . 
\]
\begin{defn}\label{def:3.5}
A finite dimensional $\mc H$-representation $(\pi,V)$ is tempered if the following equivalent conditions
are satisfied:
\begin{itemize}
\item $|t(x)| \leq 1$ for all $t \in \mr{Wt}(\pi), x \in X^+$,
\item $\mr{Wt}(V) \subset T_\un T^-$,
\item $|\mr{Wt}(V)| \subset T^-$.
\end{itemize}
\end{defn}
\begin{ex} \label{ex:3.7} \begin{itemize}
\item Consider the root datum $\mc R_n$ of type $GL_n$. Then
\begin{multline*}
\mf a = \mf a^* = \R^n ,\qquad 
\mf a^+ = \mf a^{*,+} = \{ \nu \in \R^n : \nu_1 \geq \nu_2 \geq \cdots \geq \nu_n \} ,\\
\mf a^- = \{ \nu \in \R^n : \nu_1 \leq 0, \nu_1 + \nu_2 \leq 0 ,\ldots, 
\nu_1 + \nu_2 + \cdots + \nu_{n-1} \leq 0, \nu_1 + \nu_2 + \cdots + \nu_n = 0 \} .
\end{multline*}
The Steinberg representation has $\big( \mb{q}^{(1-n)/2}, \mb{q}^{(3-n)/2}, \ldots, \mb{q}^{(n-1)/2} \big)$ 
as its only $\C[X]$-weight, so it is tempered when $\mb q \geq 1$. On the other hand, the trivial 
$\mc H_n (\mb q)$-representation satisfies Wt(triv) $= \big\{ ( \mb q^{(n-1)/2}, \mb q^{(n-3)/2}, \ldots, 
\mb{q}^{(1-n)/2} ) \big\}$, so it is tempered when $\mb{q} \leq 1$.
\item For $q = 1$, the $\mc O (T)$-weights of any irreducible representation of
$\mc H (\mc R,1) = \C [X \rtimes W]$ form a full $W$-orbit in $T$, see Paragraph \ref{par:q=1}.
The $W$-orbit of $\log |t| \in \mf a$ can only be contained in $\mf a^{-}$ if $\log |t| = 0$, that
is, $|t| = 1$. Hence the irreducible tempered representations of $\C [X \rtimes W]$ are precisely
the irreducible constituents of the modules $\ind_{\C[X]}^{\C [X \rtimes W]} \C_t$ with $t \in T_\un$.
\end{itemize} \end{ex}

\noindent
With Lemma \ref{lem:3.3} one can show (see \cite[Lemma 3.1.1]{SolAHA} and \cite[Lemma 2.4.c]{AMS3}):

\begin{prop}\label{prop:3.6}
Let $(\pi,V)$ be a finite dimensional representation of $\mc H^P$, for some $P \subset \Delta$. Then
\[
\pi \text{ is tempered } \Longleftrightarrow \ind_{\mc H^P}^{\mc H}(\pi) \text{ is tempered.}  
\]
\end{prop}

For $P = \emptyset$ we have $R = \emptyset$ and $T^+ = T^- = \{1\}$. Then Proposition \ref{prop:3.6}
says that $\ind_{\C [X]}^{\mc H}(\C_t)$ is tempered if and only if $t \in T_\un$.

In the Langlands classification we need a somewhat more general kind of $\mc H$-representation, for
which we merely require that it becomes tempered upon restriction to $\mc H (W_\af,q)$. This can also
be formulated with a more relaxed condition on the weights. This involves the Lie subgroup $T^\Delta$
of $T$, whose Lie algebra is identified with $\mf t^\Delta := R^\perp \subset \mf t$.

\begin{defn}\label{def:3.7}
A finite dimensional $\mc H$-representation $(\pi,V)$ is essentially tempered if the following 
equivalent conditions are satisfied:
\begin{itemize}
\item $|t(x)| \leq 1$ for all $t \in \mr{Wt}(\pi), x \in X^+ \cap W_\af$,
\item $\mr{Wt}(V) \subset T^\Delta T_\un T^-$,
\item $|\mr{Wt}(V)| \subset T^\Delta T^-$.
\end{itemize}
\end{defn}
\noindent
When the root datum $\mc R$ is semisimple, essentially tempered is the same as tempered.

\begin{lem} \textup{(see \cite[Lemma 3.5]{SolThe} and \cite[Lemma 2.3]{SolComp})} \\
\label{lem:3.8}
For any irreducible essentially tempered $\mc H$-representation $\pi$, there exists
$t \in T^\Delta$ such that $\pi \circ \psi_t$
is tempered and arises by inflation from a representation of $\mc H_\Delta$.
\end{lem}

\begin{ex}
Consider $\mc R_n$ and $\mc H_n (\mb{q})$ with $\mb q > 1$. Then
\[
X^+ \cap W_\af = \{ x \in \Z^n : x_1 \geq x_2 \geq \cdots \geq x_n, 
x_1 + x_2 + \cdots + x_n = 0 \} .
\]
For $z \in \C^\times$ the $\mc H_n (\mb{q})$-representation $\mr{St} \otimes t_z$ from \eqref{eq:2.14} 
has a unique $\C[X]$-weight 
\[
\big( z \mb{q}^{(1-n)/2}, z \mb{q}^{(3-n)/2}, \ldots, z \mb{q}^{(n-1)/2} \big) \in T^\Delta T_\un T^- .
\]
It is essentially tempered for all $z \in \C$, and tempered if and only if $|z| = 1$.
\end{ex}

In the Langlands classification we employ irreducible representations of $\mc H^P$,
where $P \subset \Delta$. We need some further notations:
\[
\begin{array}{l}
\mf t^P = Y^P \otimes_\Z \C = \mr{Lie}(T^P), \qquad \mf a^P = Y^P \otimes_\Z i \R = \mr{Lie}(T^P_\un), \\
\mf a^{P,+} = \{ \nu \in \mf a^P : \inp{\alpha}{\nu} \geq 0 \; \forall \alpha \in \Delta \setminus P \}, \\
T^{P,+} = \exp (\mf a^{P,+}) = \{ t \in T^P \cap \exp (\mf a) : |t(\alpha)| \geq 1 \;
\forall \alpha \in \Delta \setminus P \}, \\
\mf a^{P,++} = \{ \nu \in \mf a^P : \inp{\alpha}{\nu} > 0 \; \forall \Delta \setminus P \}
\qquad T^{P,++} = \exp (\mf a^{P,++}) .
\end{array}
\]
We say that $(\pi,V) \in \Irr (\mc H^P)$ is in positive position if cc$(\pi) = t W_P r$ with 
$t \in T^{P,++} T_\un$ and $r \in T_P$. By Lemma \ref{lem:3.8} this is equivalent to requiring that 
$\pi = \pi' \circ \psi_t$ for some $t \in T^{P,++} T^P_\un$ and $\pi' \in \Irr (\mc H_P)$. 

\begin{defn}\label{def:3.9}
A Langlands datum for $\mc H$ consists of a subset $P \subset \Delta$ and an irreducible
essentially tempered representation $\sigma$ of $\mc H^P$ in positive position. 

Equivalently, it can be given by a triple $(P,\tau,t)$, where $P \subset \Delta, \tau \in \Irr (\mc H_P)$ 
is tempered and $t \in T^{P,++} T^{P,\un}$. Then the associated $\mc H^P$-representation is
$\sigma = \tau \circ \psi_t$.

The $\mc H$-representations $\ind_{\mc H^P}^{\mc H}(\sigma)$ and 
$\ind_{\mc H^P}^{\mc H}(\tau \circ \psi_t)$ are called standard.
\end{defn}

By Lemma \ref{lem:3.3} every standard $\mc H$-module admits a central character.

Now we can finally state the Langlands classification for affine Hecke algebras.

\begin{thm} \textup{\cite[Theorem 2.4.4]{SolAHA}} \\
\label{thm:3.10}
Let $(P,\sigma)$ be a Langlands datum for $\mc H$. 
\enuma{
\item The $\mc H$-representation $\ind_{\mc H^P}^{\mc H}(\sigma)$ has a unique irreducible quotient,
which we call $L(P,\sigma)$.
\item For every irreducible $\mc H$-representation $\pi$ there exists a Langlands datum
$(P,\sigma)$ with $L(P,\sigma) \cong \pi$.
\item If $(P',\sigma')$ is another Langlands datum and $L(P',\sigma') \cong L(P,\sigma)$, then
$P' = P$ and the $\mc H^P$-representations $\sigma'$ and $\sigma$ are equivalent.
}
\end{thm}

Some consequences can be drawn immediately:
\begin{itemize}
\item $L(P,\sigma)$ is tempered if and only if $P = \Delta$ and $\sigma \in \Irr (\mc H)$
is tempered (because $L(\Delta,\sigma) = \sigma$ and Langlands data are unique).
\item In terms of Langlands data $(P,\tau,t)$ and $(P',\tau',t')$, the irreducible\\
$\mc H$-representations $L(P,\tau,t) = L(P,\tau \circ \psi_t)$ and $L(P',\tau',t') = 
L(P',\tau' \circ \psi_{t'})$ are equivalent if and only if $P = P'$ and
$\tau \circ t \cong \tau' \circ t'$. (The latter does not imply $t = t'$, for $t$ and $t'$
may still differ by an element of $T^P \cap T_P$.)
\item For $q=1$, Corollary \ref{cor:2.1} shows that every standard module is irreducible.
Hence the notions of irreducible representations and standard modules coincide for $\C [X \rtimes W]$.
\end{itemize}
\begin{ex} 
We work out the Langlands classification for $\mc H (\mc R,\mb{q})$ with
$\mc R$ of type $\widetilde{A_1}$. Its irreducible representations were already listed in
paragraph \ref{par:A1t}. The irreducible tempered representations are:
\begin{itemize}
\item the Steinberg representation,
\item the parabolically induced representations $\ind_{\C[X]}^{\mc H}(\C_t)$ with
$t \in T_\un = S^1$ but $t \neq -1$ (where we recall that $\ind_{\C[X]}^{\mc H} (\C_{t^{-1}}) \cong
\ind_{\C[X]}^{\mc H} (\C_t )$),
\item the two representations $\pi (-1,\triv)$ and $\pi (-1,\mr{St})$ which sum to
$\ind_{\C[X]}^{\mc H} (\C_{-1})$.
\end{itemize}
These representations exhaust the Langlands data $(P,\sigma)$ with $P = \Delta = \{\alpha\}$.
For $P = \emptyset$ we have $\mc H_\emptyset = \C[X]$. Its irreducible essentially tempered
representations in positive position are the $\C_t$ with $t \in T^{++} T_\un =
\{ z \in \C^\times : |z| > 1 \}$. We have
\begin{itemize}
\item $L(\emptyset,\C_t) = \ind_{\C[X]}^{\mc H} (\C_t)$ unless $t = \mb q$,
\item $L(\emptyset,\C_{\mb q}) =$ triv.
\end{itemize}
In this way we obtain every element of $\Irr (\mc H (\mc R,\mb{q}))$--as listed at the end of
paragraph \ref{par:A1t}--exacly once, because $T^{++} T_\un$ is a fundamental domain for the
action of $W = \langle s_\alpha \rangle$ on $T \cong \C^\times$.
\end{ex}

Theorem \ref{thm:3.10} provides a quick and beautiful way to classify $\Irr (\mc H)$ in terms of
the irreducible tempered representations of its parabolic subquotient algebras $\mc H_P$. 

On the downside, it conceals the topological structure of $\Irr (\mc H)$. For instance, with
$\mc R$ of type $\widetilde{A_1}$ as discussed above, there is a family of $\mc H$-representations\\
$\pi (\emptyset,\triv,t) = \ind_{\C[X]}^{\mc H}(\C_t)$ with $t \in T$, irreducible for 
almost all $t$. The Langlands classification breaks it into two families, one with $t \in T_\un$
and one with $t \in T \setminus T_\un$. In the next paragraph will see how this can be improved. 

We end this paragraph with some useful extras about the Langlands classification.
The central character of an irreducible $\mc H_P$-representation $\tau$ is an element
cc$(\tau)$ of $T_P / W_P$. Its absolute value $|\mr{cc}(\pi)|$, with respect to the polar
decomposition \eqref{eq:3.4}, lies in $\exp (\mf a_P) / W_P$, and $\log |\mr{cc}(\pi)| \in 
\mf a_P / W_P$. We fix a $W$-invariant inner product on $\mf a$. The norm of $\log |t|$ is
the same for all representatives $t \in T_P$ of cc$(\tau)$. That enables us to write
\[
\| \mr{cc}(\tau) \| = \| \log |t| \,\| 
\qquad \text{for any } t \in T_P \text{ with } W_P t = \mr{cc}(\tau) .
\]

\begin{lem} \textup{\cite[Lemma 2.2.6]{SolAHA}} \\
\label{lem:3.11}
Let $(P,\tau,t)$ be a Langlands datum.
\enuma{
\item $\End_{\mc H} (\pi (P,\tau,t)) = \C \: \mr{id}$.
\item The representation $L(P,\tau,t)$ appears with multiplicity one in \\
$\pi (P,\tau,t) = \ind_{\mc H^P}^{\mc H}(\tau \circ \psi_t)$. All other constituents $L(P',\tau',t')$ 
of $\pi (P,\tau,t)$ are larger, in the sense that $\| \mr{cc}(\tau') \| > \| \mr{cc}(\tau) \|$.
\item Let $Ws \in T / W$. Both $\{ \pi \in \Irr (\mc H) : \mr{cc}(\pi) = Ws\}$ and 
\[
\{ \pi (P,\tau,t) : (P,\tau,t) \text{ Langlands datum with } \mr{cc}(\tau) t \subset Ws \}
\]
are bases of the Grothendieck group of the category of finite dimensional $\mc H$-module all whose 
$\mc O (T)$-weights are in $Ws$. With respect to a total ordering that extends the partial ordering 
defined by part (b), the transition matrix between these two bases is unipotent and upper triangular.
}
\end{lem}

\subsection{Discrete series representations} \
\label{par:discrete}

In the representation theory of a reductive group $G$ over a local field, Harish-Chandra showed
that every irreducible tempered $G$-representation $\tau$ can be obtained from an irreducible 
square-integrable modulo centre representation $\delta$ of a Levi subgroup $M$ of $G$, see
\cite[Theorem 8.5.3]{Kna} and \cite[Proposition III.4.1.i]{Wal}. More precisely, $\tau$ is a direct
summand of the parabolic induction of $\delta$. When the centre of $M$ is compact, $\delta$ is an
isolated point of the space of irreducible tempered $M$-representations, and it is called a 
discrete series representation. Then it is a subrepresentation of $L^2 (M)$. 

Like for the Langlands classification, these results can be formulated and proven for affine Hecke
algebras as well. For these purposes it is essential that $q_s \in \R_{>0}$ for all $s \in S_\af$.
To achieve that, we require not only that $\mb q \in \R_{>1}$ (as we already did), but also that
\begin{equation}\label{eq:3.5}
\lambda (\alpha) \in \R, \lambda^* (\alpha) \in \R \qquad \forall \alpha \in \R .
\end{equation}
This condition will be in force in the remainder of this paragraph. \\
We note that the space $\mf a^{--}$ from \eqref{eq:3.6} is empty unless $R$ spans $\mf a^*$, and
in that case 
\[
\mf a^{--} = \big\{ \sum\nolimits_{\alpha \in \Delta} x_\alpha \alpha^\vee : x_\alpha < 0 \big\} . 
\]
For $P \subset \Delta$ we have
\[
\mf a_P = Y_P \otimes_\Z \R ,\quad \mf t_P = Y_P \otimes_\Z \C,
\]
and that gives rise to $\mf a_P^+, \mf a_P^-, T_P^+, T_P^-$. Let 
\[
\mf a_P^{--} = \big\{ \nu \in \mf a_P : \inp{\delta}{\nu} < 0\;
\forall \delta \in \mf a_P^{*+} \setminus \{0\} \big\} 
\]
and $T_P^{--} = \exp (\mf a_P^{--})$ be the versions of $\mf a^{--}$ and 
$T^{--} = \exp (\mf a^{--})$ for $\mc R_P$. 

\begin{defn}\label{def:3.12}
Let $(\pi,V)$ be a finite dimensional $\mc H$-representation. We say that $\pi$ belongs to the
discrete series if the following equivalent conditions hold:
\begin{itemize}
\item $|x(t)| < 1$ for all $t \in \mr{Wt}(V)$ and all $x \in X^+ \setminus \{0\}$,
\item $|\mr{Wt}(V)| \subset T^{--}$,
\item $\mr{Wt}(V) \subset T_\un T^{--}$.
\end{itemize}
Further, we can $\pi$ essentially discrete series if the following equivalent conditions hold:
\begin{itemize}
\item $|x(t)| < 1$ for all $t \in \mr{Wt}(V)$ and all $x \in W_\af X^+ \setminus \{0\}$,
\item $|\mr{Wt}(V)| \subset T^\Delta T^{--}$,
\item $\mr{Wt}(V) \subset T^\Delta T_\un T^{--}$.
\end{itemize}
Here $U T^{--}$ (for any $U \subset T$) is considered as empty if $T^{--} = \emptyset$.
\end{defn}

\begin{ex} \begin{itemize}
\item Let $\mc R$ be of type $\widetilde{A_1}$ and consider $\mc H = \mc H (\mc R,
\lambda,\lambda^*,\mb q)$ with $\lambda (\alpha) > 0$ and $\lambda^* (\alpha) > 0$.
Here $\mf a^{--} = \R_{<0}$ and $T^{--} = (0,1)$. The only weight of the Steinberg
representation is $q^{(\lambda (\alpha) + \lambda^* (\alpha))/2}$, so it is discrete series.
The trivial representation and all the twodimensional irreducible representation of $\mc H$ are
not discrete series. From the classification in paragraph \ref{par:A1t}, especially
\eqref{eq:2.11} and \eqref{eq:2.12}, we see that St is the only irreducible discrete series
$\mc H$-representation if $\lambda (\alpha) = \lambda^* (\alpha)$. When $\lambda (\alpha) >
\lambda^* (\alpha)$, $\pi (-1,\mr{St})$ is the only other irreducible discrete series
representation, while $\pi (-1,\triv)$ is discrete series if $\lambda (\alpha) < 
\lambda^* (\alpha)$.
\item The affine Hecke algebra $\mc H_n (\mb{q})$ of type $GL_n$ has no discrete series,
because its root datum $\mc R_n$ is not semisimple. Here
\[
\mf a_\Delta^{--} = \{ \nu \in \R^n : \nu_1 < 0, \nu_1 + \nu_2 < 0,\ldots,
\nu_1 + \cdots + \nu_{n-1} < 0, \nu_1 + \cdots + \nu_{n-1} + \nu_n = 0 \} .
\]
For any $z \in \C^\times$, the twist $\mr{St} \otimes t_z = \mr{St} \circ \psi_{t_z}$ of 
the Steinberg representation is essentially discrete series. In fact these are all 
irreducible essentially discrete series representations of $\mc H_n (\mb{q})$ 
(recall that $\mb{q} > 1$).
\item An affine Hecke algebra with $\lambda = \lambda^* = 0$ does not have
discrete series representations, apart from the case $\mc R = (0,\emptyset,0,\emptyset)$,
when the trivial representation of $\mc H = \C$ is regarded as discrete series.
\end{itemize} \end{ex}

The relation between Definition \ref{def:3.12} and representations of reductive $p$-adic groups
goes via Casselman's criterium for square-integrability, see \cite[Theorem 4.4.6]{Cas} and
\cite[\S VII.1.2]{Ren}. Opdam \cite[Lemma 2.22]{Opd-Sp} translated this to various criteria
for $\mc H$-representations, which are equivalent with Definition \ref{def:3.12}. 

We condition \eqref{eq:3.5} at hand, we can define a Hermitian inner product on $\mc H$ by
declaring that $\{ q(w)^{-1/2} T_w \}$ is an orthonormal basis. Let $L^2 (W(\mc R),q)$ be the
Hilbert space completion of $\mc H$ with respect to this inner product--it is canonically 
isomorphic to $L^2 (W(\mc R))$. By \cite[Lemma 2.22]{Opd-Sp}, every irreducible 
discrete series representation of $\mc H$ is isomorphic to a subrepresentation of the regular
representation of $\mc H$ on $L^2 (W(\mc R),q)$.\\

In terms of representations of a reductive group $G$, "essentially discrete series" means that
a representation has finite length, and that its restriction to the derived group of $G$
is square-integrable. An essentially discrete series representation is tempered if and only if
$\{ \theta_x : x \in X \cap (\Delta^\vee)^\perp \}$ acts on it by characters from $T^\Delta_\un$.

\begin{thm}\label{thm:3.13} 
\textup{\cite[Theorem 3.22]{DeOp1} and \cite[Lemma 1.3]{SolComp}} 
\enuma{
\item Let $\pi'$ be an irreducible tempered $\mc H$-representation. There exist $P \subset \Delta$
and a tempered essentially discrete series representation $\delta'$ of $\mc H^P$ such that 
$\pi'$ is a direct summand of $\ind_{\mc H^P}^{\mc H}(\delta')$.
\item Let $\pi$ be an irreducible essentially discrete series $\mc H^P$-representation. 
There exist $t \in T^P$ and a discrete series $\mc H_P$-representation $\delta$ such that 
$\pi \cong \delta \circ \psi_t$.
}
\end{thm}

Clearly, a combination of Theorems \ref{thm:3.10} and \ref{thm:3.13} yields some description of
$\Irr (\mc H)$ in terms of parabolic induction and the discrete series of the subquotient algebras 
$\mc H_P$ with $P \subset \Delta$. We work this out in detail. 

\begin{defn}\label{def:3.14}
An induction datum for $\mc H$ is a triple $\xi = (P,\delta,t)$, where
\begin{itemize}
\item $P \subset \Delta$,
\item $\delta$ is an irreducible discrete series representation of $\mc H_P$,
\item $t \in T^P$.
\end{itemize}
We regard two triples $\xi$ and $\xi' = (P',\delta',t')$ as isomorphic (notation $\xi \cong \xi'$)
if $P = P', t = t'$ and $\delta \cong \delta'$. Let $\Xi$ be the space of such induction data,
topologized by regarding $P$ and $\delta$ as discrete variables and $T^P$ as a complex analytic
variety. We say that $\xi = (P,\delta,t)$ is positive, written $\xi \in \Xi^+$, when $|t| \in T^{P+}$.
\end{defn}

We already associated to such an induction datum the parabolically induced representation
\[
\pi (P,\delta,t) = \ind_{\mc H^P}^{\mc H}(\delta \circ \psi_t). 
\]
By Proposition \ref{prop:3.6}
\begin{equation}
\pi (P,\delta,t) \text{ is tempered } \Longleftrightarrow \; t \in T^P_\un .
\end{equation}
For $\xi \in \Xi^+$ we write 
\[
P(\xi) = \{ \alpha \in \Delta : |t(\alpha)| = 1 \} .
\]
Then $P \subset P(\xi)$. This set of simple roots is useful because it allows to break the process
of parabolic induction in two steps: the first dealing only with essentially tempered representations
and the second similar to the Langlands classification. More concretely, by Proposition \ref{prop:3.6} 
$\ind_{(\mc H_{P(\xi)})^P}^{\mc H_{P(\xi)}} (\delta \circ \psi_{|t|})$ is tempered, while
\[
\pi^{P(\xi)} (\xi) := \ind_{\mc H^P}^{\mc H^{P(\xi)}} (\delta \circ \psi_t)
\]
is essentially tempered.

\begin{prop}\label{prop:3.15} 
\textup{\cite[Proposition 3.1.4]{SolAHA}} \\
Let $\xi = (P,\delta,t) \in \Xi^+$ and pick $t^{P(\xi)} \in T^{P(\xi)}$ such that 
$t^{P(\xi)} t^{-1} \in T_{P(\xi)}$.
\enuma{
\item The $\mc H^{P(\xi)}$-representation $\pi^{P(\xi)}(\xi)$
is completely reducible and $\pi^{P(\xi)}(\xi) \circ \psi_{t^{P(\xi)}}^{-1}$ is tempered.
\item Every irreducible summand of $\pi^{P(\xi)}(\xi)$ is of the form
$\pi^{P(\xi)} (P(\xi), \tau, t^{P(\xi)})$, where $(P(\xi), \tau, t^{P(\xi)})$ 
is a Langlands datum for $\mc H$.
\item The irreducible quotients of $\pi (\xi)$ are the representations 
$L (P(\xi), \tau, t^{P(\xi)})$, with $(P(\xi), \tau, t^{P(\xi)})$ coming from part (b).
\item Every irreducible $\mc H$-representation is of the form described in part (c).
\item The functor $\mr{Ind}_{\mc H^{P(\xi)}}^{\mc H}$ induces an isomorphism
\[
\mr{End}_{\mc H^{P(\xi)}} \big( \pi^{P(\xi)}(\xi) \big) \isom \mr{End}_{\mc H} (\pi (\xi)) .
\] }
\end{prop}

For a given $\pi \in \Irr (\mc H)$, there are in general several induction data $\xi \in \Xi^+$
such that $\pi$ is a quotient of $\pi (\xi)$. So, in contrast with the Langlands classificaiton,
Proposition \ref{prop:3.15} does not provide an actual parametrization of $\Irr (\mc H)$.
To bring that goal closer, one has to analyse the relations between the various representations
$\pi (\xi)$ with $\xi \in \Xi$.

For $u$ in the finite group
\[
T^P \cap T_P = \Hom_\Z \big( X / (X \cap \Q P) \oplus (X \cap (P^\vee)^\perp), \C^\times \big) ,
\]
we have an automorphism $\psi_u$ of $\mc H^P$ and a similar automorphism $\psi_{P,u}$ of $\mc H_P$,
given by
\[
\psi_{P,u} (\theta_{x_P} T_w) = u(x_P) \theta_{x_P} T_w \qquad x_P \in X_P, w \in W_P.
\]
Then $(\delta \circ \psi_{P,u}^{-1}) \circ \psi_{ut} = \delta \circ \psi_t$, so
\begin{equation}\label{eq:3.9}
\pi (P,\delta \circ \psi_{P,u}^{-1},ut) = \pi (P,\delta,t) .
\end{equation}
Suppose that $w \in W$ and $w(P) = P' \subset \Delta$. Then
\begin{equation}\label{eq:3.11}
\begin{array}{cccc}
\psi_w : & \mc H^P & \to & \mc H^{P'} \\
& \theta_x T_{w'} & \mapsto & \theta_{w(x)} T_{w w' w^{-1}}
\end{array}
\end{equation}
is an algebra isomorphism, and it descends to an algebra isomorphism $\psi_w : \mc H_P \to \mc H_{P'}$.
Moreover, $\psi_w$ can be implemented as conjugation by the element 
\begin{equation}\label{eq:3.12}
\imath^\circ_w \in \C (X)^W \otimes_{\C [X]^W} \mc H
\end{equation}
from Proposition \ref{prop:1.9}, see \cite[(3.124)]{SolAHA}. There is a bijection 
\begin{equation}\label{eq:3.7}
\begin{array}{cccc}
I_w : & \big( \C (X)^W \otimes_{\C [X]^W} \mc H \big) \otimes_{\mc H^P} V_\delta & \to &
\big( \C (X)^W \otimes_{\C [X]^W} \mc H \big) \otimes_{\mc H^{P'}} V_\delta \\
& h \otimes v & \mapsto & h \imath_{w^{-1}}^\circ \otimes v 
\end{array}.
\end{equation}
By \cite[Theorem 4.33]{Opd-Sp} it is isomorphism between that $\mc H$-representations
\[
\ind_{\mc H^P}^{ \C (X)^W \otimes_{\C [X]^W} \mc H}(\delta \circ \psi_t) \quad \text{and} \quad
\ind_{\mc H^{P'}}^{ \C (X)^W \otimes_{\C [X]^W} \mc H}(\delta \circ \psi_w^{-1} \circ \psi_{w(t)}) .
\]
For $t$ in a Zariski-open dense subset of $T^P$, $I_w$ specializes to an $\mc H$-isomorphism
\begin{equation}\label{eq:3.8}
\pi (P,\delta,t) \isom \pi (w(P),\delta \circ \psi_w^{-1},w(t)) .
\end{equation}
As such, $I_w : \mc H \otimes_{\mc H^P} V_\delta \to \mc H \otimes_{\mc H_{P'}} V_\delta$ is
a rational map in the variable $t \in T^P$ (possibly with poles for some $t \in T^P$).

\begin{lem}\label{lem:3.16}
For all $t \in T^P$ and all $w \in W$ with $w (P) \subset \Delta$, the representations
$\pi (P,\delta,t)$ and $\pi (w(P),\delta \circ \psi_w^{-1}, w(t))$ have the same irreducible
constituents, with the same multiplicities.
\end{lem}
\begin{proof}
The above implies that, for $t$ in a Zariski-open subset of $T^P$, these two $\mc H$-representations
have the same character. The vector space $\mc H \otimes_{\mc H^P} V_\delta$ does not depend on $t$
and the character of $\pi (P,\delta,t)$ depends algebraically on $t \in T^P$, so in fact 
$\pi (P,\delta,t)$ and $\pi (w(P),\delta \circ \psi_w^{-1}, w(t))$ have the same character for all
$t \in T^P$. 

Since $\mc H$ is of finite rank as module over its centre, the Frobenius--Schur theorem 
\cite[Theorem 27.8]{CuRe} applies, and says that the characters of inequivalent irreducible
$\mc H$-representation are linearly independent functionals on $\mc H$. As the character of
$\pi (P,\delta,t)$ determines that representation up to semisimplification, it carries enough
information to determines the multiplicities with which the irreducible representations appear
in $\pi (P,\delta,t)$.
\end{proof}

The next results are much deeper, for their proofs involve a study of topological completions of 
affine Hecke algebras \cite{Opd-Sp,DeOp1}. For a systematic bookkeeping of the $\mc H$-isomorphisms
\eqref{eq:3.9} and \eqref{eq:3.8}, we put them in a groupoid $\mc W_\Xi$. Its base space is the 
power set of $\Delta$, the collection of morphisms from $P$ to $P'$ is
\[
\mc W_{\Xi,P P'} = \{ (w,u) \in W \times (T^P \cap T_P) : w (P) = P' \} ,
\]
and the composition comes from the group $T \rtimes W$. This groupoid acts on the space of 
induction data $\Xi$ as
\[
(w,u) (P,\delta,t) = w \cdot (P,\delta \circ \psi_u^{-1}, ut) =
(w(P),\delta \circ \psi_u^{-1} \circ \psi_w^{-1}, w(ut)) .
\]

\begin{thm}\label{thm:3.17}
Let $\xi = (P,\delta,t), \xi' = (P',\delta',t') \in \Xi^+$. The $\mc H$-representations $\pi (\xi)$
and $\pi (\xi')$ have a common irreducible quotient if and only if there exists a
$(w,u) \in \mc W_\Xi$ with $(w,u) \xi = \xi'$.
\end{thm}
When $t,t' \in T_\un$, Proposition \ref{prop:3.15}.a says that $\pi (\xi)$ and $\pi (\xi')$ are
completely reducible. Then the statement of the theorem becomes: $\pi (\xi)$ and $\pi (\xi')$
have a common irreducible subquotient if and only if $\xi' \in \mc W_\Xi \xi$. This is an analogue
of Langlands' disjointness theorem (see \cite[Theorem 14.90]{Kna} for real reductive groups and
\cite[Proposition 4.1.ii]{Wal} for $p$-adic reductive groups), and it was proven in 
\cite[Corollary 5.6]{DeOp1}. The generalization to $\xi, \xi' \in \Xi^+$ was established in
\cite[Theorem 3.3.1.a]{SolAHA}.

Every element $(w,u) \in \mc W_\Xi$ gives rise to an intertwining operator
\begin{equation}\label{eq:3.33}
\pi ((w,u),P,\delta,t) : \pi (P,\delta,t) \to \pi (w(P),\delta \circ \psi_u^{-1} \circ \psi_w^{-1}, w(ut)) .
\end{equation}
It is rational as a function of $t \in T^P$ and regular for almost all $t \in T^P$. Namely, the
isomorphisms \eqref{eq:3.8} come from \eqref{eq:3.7}, while \eqref{eq:3.9} is just the identity on
the underlying vector space. For isomorphic induction data $(P,\delta,t)$ and $(P,\sigma,t)$ we 
also need an $\mc H$-isomorphism 
\[
\ind_{\mc H^P}^{\mc H}(\delta \circ \psi_t) \to \ind_{\mc H^P}^{\mc H}(\sigma \circ \psi_t).  
\]
To that end we pick (independently of $t$) an $\mc H_P$-isomorphism $\delta \cong \sigma$
(which anyway is unique up to scalars) and apply $\ind_{\mc H^P}^{\mc H}$ to that. In this way we
get an intertwining operator
\[
\pi ((w,u),\xi) : \pi (\xi) \to \pi (\xi')
\]
whenever $(w,u) \xi$ and $\xi'$ are isomorphic. This operator is unique up to scalars and
\begin{equation}\label{eq:3.10}
\pi ((w',u'), (w,u)\xi) \circ \pi ((w,u),\xi) = \natural ((w',u'),(w,u)) \pi ((w',u')(w,u),\xi)
\end{equation}
for some $\natural ((w',u'),(w,u)) \in \C^\times$.

\begin{thm}\label{thm:3.18}
Let $\xi, \xi'  \in \Xi^+$. The operators 
\[
\{ \pi ((w,u),\xi) : (w,u) \in \mc W_\Xi, (w,u) \xi \cong \xi' \}
\] 
are regular and invertible, and they span $\Hom_{\mc H} (\pi (\xi), \pi (\xi'))$.
\end{thm}
In case the coordinates $t,t'$ of $\xi,\xi'$ lie in $T_\un$, this is shown in 
\cite[Corollary 5.4]{DeOp1}. That version is analogue of Harish-Chandra's completeness theorem,
see \cite[Theorem 14.31]{Kna} and \cite[Theorem 5.5.3.2]{Sil2}.
For the version with arbitrary $\xi, \xi' \in \Xi^+$ we refer to \cite[Theorem 3.3.1.b]{SolAHA}.

The relations between parabolically induced representations $\pi (\xi)$ and $\pi (\xi')$ with
$\xi, \xi'  \in \Xi$ $\mc W_\Xi$-associate but not positive are more complicated, and not
understood well. For the principal series $\pi (\emptyset,\triv,t)$ this issue was 
investigated in \cite{Ree1}.

Finally, everything is in place to formulate an extension of the Langlands classification
that incorporates results about discrete series representations. The outcome is similar to 
L-packets in the local Langlands correspondence. Recall that \eqref{eq:3.5} is in force.

\begin{thm}\label{thm:3.19}
\textup{\cite[Theorem 3.3.2]{SolAHA}} \\
Let $\pi$ be an irreducible $\mc H$-representation. There exists a unique
$\mc W_\Xi$-association class $(P,\delta,t) \in \Xi / \mc W_\Xi$ such that the 
following equivalent statements hold:
\enuma{
\item $\pi$ is isomorphic to an irreducible quotient of $\pi (\xi^+)$, for some
$\xi^+ \in \Xi^+ \cap \mc W_\Xi (P,\delta,t)$;
\item $\pi$ is a constituent of $\pi (P,\delta,t)$, and $\| cc (\delta) \|$ 
is maximal for this property.
}
Further, $\pi$ is tempered if and only if $t \in T_\un$.
\end{thm}

Theorem \ref{thm:3.19} associates to every irreducible $\mc H$-representation $\pi$
an essentially unique positive induction datum $(P,\delta,t)$. If 
$\xi' = (P',\delta',t')$ is another positive induction datum associated to $\pi$, then $\xi'$
$\mc W_\Xi$-associate to $(P,\delta,t)$ and Theorem \ref{thm:3.18} entails that
$\pi (\xi') \cong \pi (P,\delta,t)$. Hence the parabolically induced representation
$\pi (P,\delta,t)$ is uniquely determined by $\pi$, up to isomorphism of $\mc H$-representations.

Conversely, however, $\pi (P,\delta,t)$ can have more than one irreducible quotient.
So, like an L-packet for a reductive group can have more than one element, $(P,\delta,t)$ might 
be associated (in Theorem \ref{thm:3.19}) to several irreducible $\mc H$-representations.\\

\begin{ex} \begin{itemize}
\item Consider $\mc R$ of type $\widetilde{A_1}$, $\mc H = \mc H(\mc R,\mb{q})$. Then
$\mc W_{\Xi, \emptyset \emptyset} = \{1,s_\alpha\}, \mc W_{\Xi,\Delta \Delta} = \{1\}$
and Theorem \ref{thm:3.19} works out as follows, where we take $t$ or $t^{-1}$ 
depending on whether $|t| \geq 1$ or $|t| \leq 1$: 
\[
\begin{array}{lcl}
\Irr (\mc H) & & \text{induction data} \\
\hline
\mr{St} & \mapsto & (\Delta = \{\alpha\},\mr{St},1) \\
\triv & \mapsto & (\emptyset,\triv,\mb{q}) \\
\pi (-1,\mr{St}), \pi (-1,\triv) & \mapsto & (\emptyset,\triv,-1) \\
\ind_{\C[X]}^{\mc H}(\C_t) & \mapsto & (\emptyset,\triv,t^{\pm 1})
\end{array}
\]
\item Keep $\mc R$ of type $\widetilde{A_1}$, but consider $\mc H = \mc H (\mc R,1) =
\C [W_\af]$. In this case $\mc H$ does not have any discrete series representations,
and the effect of Theorem \ref{thm:3.19} is:
\[
\begin{array}{lcl}
\Irr (\mc H) & & \text{induction data} \\
\hline
\mr{St}, \triv & \mapsto & (\emptyset,\triv,1) \\
\pi (-1,\mr{St}), \pi (-1,\triv) & \mapsto & (\emptyset,\triv,-1) \\
\ind_{\C[X]}^{\mc H}(\C_t) & \mapsto & (\emptyset,\triv,t^{\pm 1})
\end{array}
\]
\item For $\mc R$ of type $GL_n$, $\mc H = \mc H_n (\mb{q})$, and $P \subset \Delta$
given by a partition $\vec n = (n_1,n_2,\ldots,n_d)$ of $n$, we have
\[
\mc W_{\Xi,PP} / (T^P \cap T_P) = \prod_{m \geq 1} S (\{ i : n_i = m\}) \cong
N_{S_n} \Big( \prod_{i=1}^d S_{n_i} \Big) \Big/ \prod_{i=1}^d S_{n_i} .
\] 
The only irreducible discrete series representations of $\mc H (\mc R_{n,P},\mb{q})$
are $\mr{St} \otimes t_z = \mr{St} \circ \psi_{P,t_z}$ with $t_z \in T^P \cap T_P$.
Recall that 
\[
T^P \cong \prod\nolimits_{i=1}^d ( (\C^\times)^{n_i} )^{S_{n_i}} \cong (\C^\times )^d .
\]
An induction datum $(P,\mr{St} \otimes t_z,\vec{z})$ with $\vec{z} = (z_1,\ldots,z_d) 
\in T^P$ is positive if and only if $|z_1| \geq |z_2| \cdots \geq |z_n|$, 
a condition which is preserved by the action of $T^P \cap T_P$ on such induction data.
Hence, up $\mc W_\Xi$-association, it suffices to consider only positive induction data
$(P,\mr{St},t)$--so with $t_z = 1$.

It is easily checked that every pair $(\vec n, \vec z)$ as in Theorem \ref{thm:2.6}.b is
$S_n$-associate to a pair $(\vec n, \vec z)$ which is positive in the above sense, and
that the latter is unique up to $\mc W_\Xi$-association. The irreducible 
$\mc H_n (\mb{q})$-representation attached to $(\vec n, \vec z)$ in paragraph \ref{par:GLn}
is a quotient of $\pi (P,\mr{St},\vec z) = \pi (\vec n, \vec z)$, just as in Theorem 
\ref{thm:3.19}. Thus, for $\mc H_n (\mb{q})$ Theorem \ref{thm:3.19} recovers Theorem
\ref{thm:2.6}: both parametrize $\Irr (\mc H_n (\mb{q}))$ bijectively with basically
the same data.
\end{itemize} \end{ex}

\subsection{Lusztig's reduction theorems} \
\label{par:reduction}

In Paragraph \ref{par:defGHA} we already hinted at a link between affine Hecke algebras and
graded Hecke algebras. The connection is established with two reduction steps, whose original 
versions are due to Lusztig \cite{Lus-Gr}. These simplifications do not work in the same way for all 
representations, they depend on the central characters. The first reduction step limits the set
of central characters that has to be considered to understand all finite length $\mc H$-modules.

By \cite[Lemma 3.15]{Lus-Gr}, for $t \in T$ and $\alpha \in R$:
\begin{equation}\label{eq:3.17}
s_\alpha (t) = t \; \Longleftrightarrow \; \alpha (t) = \left\{\begin{array}{cc}
1 & \alpha^\vee \notin 2 Y \\
\pm 1 & \alpha^\vee \in 2 Y 
\end{array} \right.  .
\end{equation}
Fix $t \in T$. We will exhibit an algebra, almost a subalgebra of $\mc H$, that captures the behaviour
of $\mc H$-representations with central character close to $Wt \in T / W$. 
We consider the root system
\[
R'_t = \{ \alpha \in R : s_\alpha (t) = t \} .
\]
It fits in a root datum $\mc R'_t = (X,R'_t,Y,R_t^{'\vee})$, and $\lambda$ and $\lambda^*$ restrict 
to parameter functions for $\mc R'_t$. The affine Hecke algebra $\mc H (\mc R'_t,\lambda,\lambda^*,
\mb{q})$ naturally embeds in $\mc H = \mc H (\mc R,\lambda,\lambda^*,\mb{q})$. 

But we have to be careful, this construction is not suitable when $c_\alpha (t) = 0$ for some
$\alpha \in R$. For instance, when $\mc R$ is of type $\widetilde{A_1}$ and $t = \mb{q} \neq 1$,
the algebra $\mc H (\mc R'_t,\lambda,\lambda^*,\mb{q})$ is just $\C [X]$. That is hardly helpful to
describe all $\mc H$-modules with weights in $W \mb{q}$ (e.g. $\ind_{\C[X]}^{\mc H}(\C_{\mb q})$, 
triv and St).

When $c_\alpha (t) \neq 0$ for all $\alpha \in R$, $\mc H (\mc R'_t,\lambda,\lambda^*,\mb{q})$ 
detects most of what can happen with $\mc H$-representations with central character $W t$, but 
still not everything.

\begin{ex}
Consider the root datum $\mc R$ of type $\widetilde{A_2}$, with
\[
X = \{ x \in \Z^3 : x_1 + x_2 + x_3 = 0 \}, R = \{ e_i - e_j : i \neq j \} \cong A_2,
\Delta = \{ e_1 - e_2, e_2 - e_3 \} .
\]
The point $t \in T$ with $t(e_1 - e_2) = t(e_2 - e_3) = e^{2 \pi i / 3}$ satisfies
$R_t = \emptyset$ but $W_t = \{ \mr{id}, (123), (132) \}$. We have
$\mc H (\mc R'_t,\lambda,\lambda^*,\mb{q}) = \C [X]$, which possesses only one irreducible
representation with central character $t$. On the other hand, the R-group for $\mc H$ and 
$\xi = (\emptyset,\triv,t)$ is $W_t$. Theorem \ref{thm:3.21} entails that 
$\ind_{\C[X]}^{\mc H}(\C_t)$ splits into three inequivalent irreducible $\mc H$-representations, 
all with central character $W t$.
\end{ex}

The set
\begin{equation}\label{eq:3.20}
R_t = \{ \alpha \in R : s_\alpha (t) = t, c_\alpha (t) \neq 0 \}
\end{equation}
is a root system, because $\lambda$ and $\lambda^*$ are $W$-invariant. When $t \in T_\un$ and 
\eqref{eq:3.5} holds, $R_t$ coincides with $R'_t$ because $c_\alpha (t)$ cannot be 0. We define
\[
\mc R_t = (X,R_t,Y,R_t^\vee) .
\]
Let $R_t^+ = R^+ \cap R_t$ be the set of positive roots, and let $\Delta_t$ be the unique basis of
$R_t$ contained in $R_t^+$. We warn that $\Delta_t$ need not be a subset of $\Delta$. The group
\[
\Gamma_t = \{ w \in W_t : w (R_t^+) = R_t^+ \},
\]
satisfies $W_t = W(R_t) \rtimes \Gamma_t$.
For every $w \in \Gamma_t$ there exists an algebra automorphism like \eqref{eq:3.11}
\[
\begin{array}{cccc}
\psi_w : & \mc H (\mc R_t,\lambda,\lambda^*,\mb{q}) & \to &
\mc H (\mc R_t,\lambda,\lambda^*,\mb{q}) \\
& \theta_x T_{w'} & \mapsto & \theta_{w(x)} T_{w w' w^{-1}} 
\end{array} .
\]
This is a group action of $\Gamma_t$, so we can form the crossed product
$\mc H (\mc R_t,\lambda,\lambda^*,\mb{q}) \rtimes \Gamma_t$. We recall that this means
$\mc H (\mc R_t,\lambda,\lambda^*,\mb{q}) \otimes \C [\Gamma_t]$ as vector spaces, with multiplication rule
\[
(h \otimes \gamma) (h' \otimes \gamma') = h \psi_\gamma (h') \otimes \gamma \gamma' .
\]
The group $\Gamma_t$ can be embedded in $\big( \C (X)^W \otimes_{\C [X]^W} \mc H (\mc R,\lambda,
\lambda^*,\mb{q}) \big)^\times$ with the elements $\imath^\circ_w$ from Proposition \ref{prop:1.9}.
In view of \eqref{eq:3.11} and \eqref{eq:3.12}, this realizes 
$\mc H (\mc R_t,\lambda,\lambda^*,\mb{q}) \rtimes \Gamma_t$ as a subalgebra of
$\C (X)^W \otimes_{\C [X]^W} \mc H (\mc R,\lambda,\lambda^*,\mb{q})$.

Unfortunately tensoring with $\C (X)^W$ kills many interesting representations, so the above
does not yet suffice to relate the module categories of $\mc H (\mc R,\lambda,\lambda^*,
\mb{q})$, of $\mc H (\mc R_t,\lambda,\lambda^*,\mb{q}) \rtimes \Gamma_t$ and of a graded Hecke algebra.
To achieve that, some technicalities are needed. 

For an (analytically) open $W$-stable subset $U$ of $T$, let $C^{an} (U)$ be the algebra of complex
analytic functions on $U$. As $Z(\mc H)$ injects in $C^{an}(U)^W$, we can form the algebra
\[
\mc H^{an}(U) := C^{an}(U)^W \otimes_{\C [X]^W} \mc H (\mc R,\lambda,\lambda^*,\mb{q}) .
\]
It can also be obtained from Definition \ref{def:1.7} by using $C^{an}(U)$ instead of
$\mc O (T)$, and as a vector space it is just $C^{an}(U) \otimes_\C \mc H (W,q)$.

Let $\Mod_{f,U}$ be the category of finite length $\mc H$-modules, all whose $\mc O (T)$-weights
lie in $U$.

\begin{prop}\label{prop:3.24}
The inclusion $\mc H \to \mc H^{an}(U)$ induces an equivalence of categories
\[
\Mod_f (\mc H^{an}(U)) \to \Mod_{f,U}(\mc H) .
\]
\end{prop}

Let $U_t$ be a "sufficiently small" $W_t$-invariant open neighborhood of $t$ in $T$ and put $U = W U_t$. 
Then
\[
C^{an}(U) = \bigoplus\nolimits_{w \in W / W_t} C^{an}(w U_t) 
\] 
and there is a well-defined group homomorphism 
\[
\Gamma_t \to \mc H^{an}(U)^\times : w \mapsto 1_{U_t} \imath^\circ_w . 
\]
This realizes $\mc H (\mc R_t,\lambda,\lambda^*,\mb{q}) \rtimes \Gamma_t$ as a subalgebra of
$\mc H^{an}(U)$. Our version of Lusztig's first reduction theorem \cite[\S 8]{Lus-Gr} is:

\begin{thm}\label{thm:3.25}
Assume that \eqref{eq:3.5} holds and that $t \in T$ satisfies $W_{|t|^{-1}t} = W_t$. Then there
exist open $W_t$-stable neighborhoods $U_t$ of $t$ with the following properties.
\enuma{
\item There is a natural embedding of $C^{an}(U)^W$-algebras
\[
\mc H (\mc R_t, \lambda, \lambda^*, \mb{q})^{an}(U_t) \rtimes \Gamma_t \to \mc H^{an}(U) .
\]
\item Part (a) and Proposition \ref{prop:3.24} induce equivalences of categories
\begin{multline*}
\Mod_{f,U_t}(\mc H (\mc R_t, \lambda, \lambda^*, \mb{q}) \rtimes \Gamma_t) \cong
\Mod_f (\mc H (\mc R_t, \lambda, \lambda^*, \mb{q})^{an}(U_t) \rtimes \Gamma_t ) \\
\cong \Mod_{f}(\mc H^{an}(W U_t)) \cong \Mod_{f,W U_t} (\mc H) .
\end{multline*}
\item When $t \in T_\un$, we may take $U_t$ of the form $U'_t \times \exp (\mf a)$, where
$U'_t \subset T_\un$ is a small open $W_t$-stable ball around $t$. In that case the 
equivalences of categories between the outer terms in part (b) preserve temperedness and 
(essentially) discrete series. 
}
\end{thm}
Parts (a) and (b) of Theorem \ref{thm:3.25} can be found in \cite[Theorem 3.3]{BaMo2} and 
\cite[Theorem 2.1.2]{SolAHA}. For part (c) we refer to \cite[Proposition 2.7]{AMS3}.

The conditions in Theorem \ref{thm:3.25} avoid possible unpleasantness caused by non-invertible
intertwining operators.
Algebras of the form $\mc H (\mc R_t, \lambda, \lambda^*, \mb{q}) \rtimes \Gamma_t$ behave
almost the same as affine Hecke algebras. The difference can be handled with Clifford theory, as
in \cite[Appendix]{RaRa}. In fact everything we said so far in this paper
can be generalized to such crossed product algebras, see \cite{SolAHA,AMS3}. For simplicity,
we prefer to keep the finite groups $\Gamma_t$ out of our presentation.

An advantage of Theorem \ref{thm:3.25} is that it reduces the study of $\Mod_f (\mc H)$ to those
modules of $\mc H (\mc R_t, \lambda, \lambda^*, \mb{q}) \rtimes \Gamma_t$ all whose $\C [X]$-weights
belong to a small neighborhood of the point $t$, which is fixed by $W(R_t) \rtimes \Gamma_t$. We
will use this in a loose sense, suppressing $\Gamma_t$. Then Theorem \ref{thm:3.25} says that
it suffices to consider those finite dimensional modules of an affine Hecke algebra $\mc H$, all whose
$\C [X]$-weights lie in a small neighborhood of a $W$-fixed point $t \in T$.

The second reduction theorem will transfer such $\mc H$-representations to representations of
graded Hecke algebras. 
In the remainder of this paragraph we assume that $u$ is fixed by $W$. 
We define a parameter function $k_u$ for the root system $R_u$ by
\begin{equation}\label{eq:3.24}
k_u (\alpha) = (\lambda (\alpha) + \alpha (u) \lambda^* (\alpha)) \log (\mb{q}) / 2 .
\end{equation}
By \eqref{eq:3.17} $\alpha (u) \in \{\pm 1\}$, so $k_u$ is real-valued whenever the positivity 
condition \eqref{eq:3.5} for $q$ holds. This gives a graded Hecke algebra $\mh H (\mf t,W,k_u)$.

Let $V \subset \mf t$ be an analytically open $W$-stable subset. We can form the algebra 
\[
\mh H (\mf t,W,k_u )^{an}(V) = C^{an}(V)^W \otimes_{\mc O (\mf t)^W} \mh H (\mf t,W,k_u) ,
\]
which as vector space is just
\[
C^{an}(V) \otimes_{\mc O (\mf t)} \mh H (\mf t,W,k_u) = C^{an}(V) \otimes_\C \C [W] .
\]
Let $\Mod_{f,V}(\mh H (\mf t,W,k_u))$ be the category of those finite dimensional
$\mh H (\mf t,W,k_u)$-modules, all whose $\mc O (\mf t)$-weights lie in $V$. An analogue of
Proposition \ref{prop:3.24} says that the inclusion 
\[
\mh H (\mf t,W,k_u) \to \mh H (\mf t,W,k_u)^{an}(V)
\]
induces an equivalence of categories
\begin{equation}\label{eq:3.18}
\Mod_f (\mh H (\mf t,W,k_u)^{an}(V)) \isom \Mod_{f,V}(\mh H (\mf t,W,k_u)) .
\end{equation}
The analytic map 
\[
\exp_u : \mf t \to T, \quad \lambda \mapsto u \exp (\lambda)
\]
is $W$-equivariant, because $u$ is fixed by $W$. It gives rise to algebra homomorphisms
\[
\begin{array}{cccc}
\exp_u^* : & C^{an}(\exp_u (V)) & \to & C^{an}(V) \\
 & f & \mapsto & f \circ \exp_u \\
\Phi_u : & \C (X)^W \otimes_{\C [X]^W} \mc H^{an}(\exp_u (V)) & \to &
Q(S(\mf t^*))^W \otimes_{S(\mf t^*)^W} \mh H (\mf t,W,k_u)^{an}(V) \\
& f \imath^\circ_w & \mapsto & (f \circ \exp_u) \tilde{\imath}_w 
\end{array}.
\]
In good circumstances, the works already without involving rational (non-regular) functions.

\begin{thm}\label{thm:3.26}
\textup{\cite[Theorem 9.3]{Lus-Gr}, \cite[\S 4]{BaMo2} and \cite[Theorem 2.1.4]{SolAHA}} \\
Let $V \subset \mf t$ be an (analytically) open subset such that
\begin{itemize}
\item $V$ is $W$-stable,
\item $\exp_u$ is injective on $V$,
\item for all $\alpha \in R, \lambda \in V$ the numbers $\inp{\alpha}{\lambda}, 
\inp{\alpha}{\lambda} + k_u (\alpha)$ do not lie in $\pi i \Z \setminus \{0\}$.
\end{itemize}
\enuma{
\item $\exp_u^* : C^{an}(\exp_u (V)) \to C^{an}(V)$ is a $W$-equivariant algebra isomorphism,
and makes $\mc H^{an}(\exp_u (V))$ into a $C^{an}(V)^W$-algebra.
\item $\Phi_u$ restricts to an isomorphism of $C^{an}(V)^W$-algebras 
\[
\mc H^{an}(\exp_u (V)) \to \mh H (\mf t,W,k_u)^{an} (V).
\]
}
\end{thm}

Notice that the conditions in Theorem \ref{thm:3.26} always hold when $V$ is a small 
neighborhood of $0$ in $\mf t$. When $k_u$ is real-valued, for instance whenever \eqref{eq:3.5} 
holds, these conditions also hold for $V$ of the form $\mf a + V'$ with $V' \subset i \mf a$ 
a small ball around 0. 

Combining Theorems \ref{thm:3.25}, \ref{thm:3.26} and Proposition \ref{prop:3.24}, we can draw
important consequences for the category of finite dimensional $\mc H$-modules

\begin{cor}\label{cor:3.27}
\textup{\cite[Corollary 2.15]{SolAHA}} \\
Assume that \eqref{eq:3.5} holds and that $u \in T_\un$.
\enuma{
\item For $\lambda \in \mf a$ the categories $\Mod_{f,W u \exp (\lambda)}(\mc H)$ and
$\Mod_{f,W_u \lambda}(\mh H (\mf t,W (R_u), k_u) \rtimes \Gamma_u)$ are naturally equivalent.
\item The categories $\Mod_{f,W u \exp (\mf a)}(\mc H)$ and $\Mod_{f,\mf a}(\mh H (\mf t,W(R_u),k_u)
\rtimes \Gamma_u)$ are naturally equivalent.
\item The equivalences of categories from parts (a) and (b) are compatible with parabolic induction.
}
\end{cor}

Sometimes it is not enough that $k_u$ is real-valued, we want to have parameters in $\R_{\geq 0}$.
Suppose that $\lambda, \lambda^* :  R \to \R_{\geq 0}$ are parameter functions. Then
$q_{s_\alpha} = \mb q^{\lambda (\alpha)}$ and $q_{s'_\alpha}^{\lambda^* (\alpha)}$ belong to
$\R_{\geq 1}$ for all $\alpha \in R$, but $\lambda (\alpha) - \lambda^* (\alpha)$ could
be $<0$. In that case Theorem \ref{thm:3.26} could produce a graded Hecke algebra with a
negative parameter $k_u (\alpha)$, which could lead to unnecessary complications.

\begin{lem}\label{lem:6.11}
Let $\mc H$ be an affine Hecke algebra constructed from a root datum $\mc R$ and a parameter
function $q : S_\af \to \R_{\geq 1}$. Then $\mc H$ admits a presentation $\mc H (\mc R,\lambda,
\lambda^*, \mb q)$ with $\lambda (\alpha), \lambda^* (\alpha), \lambda (\alpha) -
\lambda^* (\alpha) \in \R_{\geq 0}$ for all $\alpha \in R$.
\end{lem}
\begin{proof}
Choose $\lambda, \lambda^* : R \to \R_{\geq 0}$ as in \eqref{eq:1.28}, so that
$\mc H \cong \mc H (\mc R,\lambda,\lambda^*,\mb q)$.

When $\lambda^* (\alpha) > \lambda (\alpha)$, $\alpha$ must be a short root in a type $B_n$
component $R_i$ of $R$ and $\R R_i^\vee \cap Y = C_n$. Use the presentation from
Definition \ref{def:1.7}. Replace the basepoint 1 of $T = Y \otimes_\Z \C^\times$ by
$\sum_{i=1}^n e_i \otimes -1 \in T^W$. This produces a new torus $T'$, and the algebra
$\mc H$ can be presented as $\mc H (T',\lambda',\lambda^{*'},\mb q)$. Here
$\lambda' (\alpha) = \lambda^* (\alpha), \lambda^{*'}(\alpha) = \lambda (\alpha)$ and
$\lambda' = \lambda, \lambda^{*'} = \lambda^*$ on $R \setminus W \alpha$. This translates
to a Bernstein presentation $\mc H (\mc R,\lambda', \lambda^{*'},\mb q)$ of $\mc H$,
with $\lambda^{*'}(\beta) > \lambda' (\beta)$ for fewer $\beta \in R$ than before.
Repeating the procedure, we can achieve $\lambda (\alpha) \geq \lambda^* (\alpha)$
for all $\alpha \in R$.
\end{proof}

\subsection{Analogues for graded Hecke algebras} \
\label{par:analogues}

Motivated by Corollary \ref{cor:3.27} we investigate all finite dimensional representations of
graded Hecke algebras. The representation theory of graded Hecke algebras has been developed together
with that of affine Hecke algebras, and they are very similar. As far as the topics in this section
are concerned, these two kinds of algebras behave analogously. Instead of translating the
paragraphs \ref{par:induction}--\ref{par:Rgroups} to graded Hecke algebras, we will be more sketchy
here, just providing the necessary definitions and pointing out the analogies. For more background
and proofs we refer to \cite{BaMo2,Eve,KrRa,SolHomGHA,SolGHA}. 

Compared to Paragraph \ref{par:defGHA}, we assume in addition that $W$ is a crystallographic Weyl
group. Equivalently, the data $\mf t, \mf a, R, W,S$ for a graded Hecke algebra now come from a
based root datum $\mc R = (X,R,Y,R^\vee,\Delta)$. Let $k : R \to \C$ be a $W$-invariant parameter
function and consider the algebra $\mh H = \mh H (\mf t,W,k)$.

A parabolic subalgebra of $\mh H$ is by definition of the form $\mh H^P = \mh H (\mf t, W_P,k)$
for a subset $P \subset \Delta$. The associated "semisimple" quotient algebra is 
$\mh H_P = \mh H (\mf t_P,W_P,k)$. The relation between these two subquotients of $\mh H$ is simple:
\[
\mh H^P = \mh H_P \otimes_\C S(\mf t^{P*}) = \mh H_P \otimes_\C \mc O (\mf t^P) .
\]
In particular every irreducible representation of $\mh H^P$ is of the form $\pi_P \otimes \lambda$
for unique $\pi_P \in \Irr (\mh H_P)$ and $\lambda \in \mf t^P$. Parabolic induction for $\mh H$
is the functor $\ind_{H^P}^{\mh H}$.

\begin{lem}\label{lem:3.30}
\textup{\cite[Theorems 2.5.b and 2.11.b]{AMS3}} \\
The equivalences of categories from Theorems \ref{thm:3.25} and \ref{thm:3.26} and Corollary
\ref{cor:3.27} are compatible with parabolic induction.
\end{lem}

Finite dimensional $\mh H$-modules have weights with respect to the commutative subalgebra
$S(\mf t^*) = \mc O (\mf t)$. With Theorem \ref{thm:3.26} one can translate many notions for 
$\mc H$-modules to $\mh H$-modules. In terms of weights, this boils down to replacing a subset
$U \subset T$ by $\exp^{-1}(U) \subset \mf t$. More concretely, we say that a finite dimensional
$\mh H$-representation $\pi$ is 
\begin{itemize}
\item tempered if Wt$(\pi) \subset \mf i a + \mf a^-$,
\item essentially tempered if Wt$(\pi) \subset \mf t^\Delta + i \mf a + \mf a^-$,
\item discrete series if Wt$(\pi) \subset i \mf a + \mf a^{--}$,
\item essentially discrete series if Wt$(\pi) \subset \mf t^\Delta + i \mf a + \mf a^{--}$.
\end{itemize}
Here $V + \mf a^{--}$ (for any $V \subset \mf t$) is considered as empty when $\mf a^{--} = \emptyset$.

Then a Langlands datum for $\mh H$ is a triple $(P,\tau,\lambda)$, where $P \subset \Delta$,
$\tau \in \Irr (\mh H_P)$ is tempered and $\lambda \in i \mf a^P + \mf a^{P++}$. Then
$\ind_{\mh H^P}^{\mh H}(\tau \otimes \lambda)$ is called a standard $\mh H$-module.
With these definitions, the Langlands classification (as in Theorem \ref{thm:3.10} and Lemma
\ref{lem:3.11}) holds true for graded Hecke algebras \cite{Eve}.

As expected, and proven in \cite[Theorem 2.11.d]{AMS3}, the equivalence of categories between 
$\Mod_{f,V}(\mh H (\mf t,W,k_u))$ and $\Mod_{f,\exp_u (V)}(\mc H)$ resulting from Theorem 
\ref{thm:3.26} and Proposition \ref{prop:3.24} preserves temperedness and (essentially) 
discrete series. Combining that with Theorem \ref{thm:3.25}, we find:

\begin{lem}\label{lem:3.28}
Assume that \eqref{eq:3.5} holds. The equivalence of categories between 
$\Mod_{f,W u \exp (\mf a)}(\mc H)$ and $\Mod_{f,\mf a}(\mh H (\mf t,W(R_u),k_u)\rtimes \Gamma_u)$
from Corollary \ref{cor:3.27} preserves temperedness and (essentially) discrete series. 
\end{lem}

From now on we assume that $k : R \to \C$ has values in $\R$, so that $\mh H$ is related to an 
affine Hecke algebra with parameters in $\R_{>0}$.
As induction data for $\mh H$ we take triples $\tilde \xi = (P,\delta,\lambda)$ where
$P \subset \Delta$, $(\delta,V_\delta) \in \Irr (\mh H_P)$ is discrete series and 
$\lambda \in \mf t^P$. The space of such triples (with $\delta$ considered up to isomorphism) 
is denoted $\tilde \Xi$. The parabolically induced representation attached to an induction datum is
\[
\pi (P,\delta,\lambda) = \ind_{\mh H^P}^{\mh H}(\delta \otimes \lambda) .
\]
\begin{lem}\label{lem:3.29}
Let $\tilde \xi = (P,\delta,\lambda) \in \tilde \Xi$.
\enuma{
\item $\pi (\tilde \xi)$ is tempered if and only if $\lambda \in i \mf a^P$.
In that case $\pi (\tilde \xi)$ is completely reducible.
\item Let $W_P \mr{cc}(\delta)$ be the central character of $\delta$.
The central character of $\pi (\tilde \xi)$ is $W (\mr{cc}(\delta) + \lambda)$.
It lies in $\mf a / W$ if and only if $\lambda \in \mf a^P$.
}
\end{lem}
\begin{proof}
The arguments from these statement will be given in the setting of affine Hecke algebras.
From there they can be translated to graded Hecke algebras with Paragraph \ref{par:reduction}.\\
(a) follows from Propositions \ref{prop:3.6} and \ref{prop:3.15}.a.\\
(b) The expression for the central character comes from Lemma \ref{lem:3.3}. 
Since $k$ is real-valued and $\delta$ is discrete series, $W_P \mr{cc}(\delta)$ lies in 
$\mf a_P / W_P$ \cite[Lemma 2.13]{Slo2}. These two facts imply the second statement.
\end{proof}

The collection of intertwining operators between the representation $\pi (\tilde{\xi})$ with
$\tilde \xi \in \tilde \Xi$ is simpler than for affine Hecke algebras, because $\mf t^P \cap
\mf t_P = \{0\}$ does not contribute to it. There is a groupoid $\mc W_{\tilde \Xi}$ over 
$\Delta$, with
\[
\mc W_{\tilde \Xi,P P'} = \{ w \in W : w (P) = P' \} .
\]
To every $w \in \mc W_{\tilde \Xi, P P'}$ one can associate an algebra isomorphism
\[
\begin{array}{ccccc}
\psi_w : & \mh H^P & \to & \mh H^{w(P)} \\
& x \otimes w' & \mapsto & w(x) \otimes w w' w^{-1} & 
\qquad x \in \mf t^*, w' \in W_P . \end{array}
\]
Then $w (\tilde \xi) = (w(P), \delta \circ \psi_w^{-1}, w(\lambda))$ is another induction
datum, and there is an intertwining operator
\[
I_w : \pi (P,\delta,\lambda) \to \pi (w(P),\delta \circ \psi_w^{-1},w(\lambda)) .
\] 
The latter is rational as a function of $\lambda \in \mf t^P$ and comes from
\[
\begin{array}{ccc}
\big( Q(S (\mf t^*))^W \otimes_{S(\mf t^*)^W} \mh H \big) \otimes_{\mh H^P} V_\delta & \to &
\big( Q(S (\mf t^*))^W \otimes_{S(\mf t^*)^W} \mh H \big) \otimes_{\mh H^{w(P)}} V_\delta \\
h \otimes v & \mapsto & h \tilde{\imath}_w \otimes v 
\end{array},
\]
where $\tilde{\imath}_w$ is as in Proposition \ref{prop:1.10}.
We call an induction datum $\tilde \xi = (P,\delta,\lambda)$ positive if 
$\lambda \in i \mf a^P + \mf a^{P+}$, and we define 
\begin{equation}\label{eq:3.32}
\begin{array}{lll}
P(\tilde \xi) & = & \{ \alpha \in \Delta : \Re \inp{\alpha}{\lambda} = 0 \} ,\\
\tilde{\Xi}^+ & = & \{ (P,\delta,\lambda) \in \tilde{\Xi} : \lambda \in i \mf a^P + \mf a^{P+} \} .
\end{array}
\end{equation}
Using these notions, the whole of paragraph \ref{par:discrete} holds for graded Hecke algebras.
see \cite{SolGHA}.

\begin{ex}\label{ex:3.1}
Consider $\mf a = \mf a^* = \R, \mf t = \mf t^* = \C, R = \{\pm 1\}, \Delta = \{\alpha = 1\},
W = \langle s_\alpha \rangle$. The graded Hecke algebra $\mh H = \mh H (\mf t,W,k)$ with 
$k(\alpha) = k > 0$
has a unique discrete series representation. It is the analogue of the Steinberg representation,
given in this setting by $\mr{St}|_{\mc O (\mf t)} = \C_{-k}$ and $\mr{St}|_{\C [W]} =$ sign.

Apart from $(\emptyset,\mr{St},0)$, the positive induction data are $(\emptyset,\triv,\lambda)$
with $\lambda \in i \mf a + \mf a^+ = i \R + \R_{\geq 0}$. For $\lambda \neq k$, 
$\pi (\emptyset,\triv,\lambda) = \mr{ind}_{\mc O (\mf t)}^{\mh H}$ is irreducible, while 
$\pi (\emptyset,\triv,k)$ has the "trivial representation" as unique irreducible quotient.
It is given by $\triv|_{\mc O (\mf t)} = \C_k$ and $\triv |_{\C [W]} = \triv$.
\end{ex}

\section{Classification of irreducible representations}
\label{sec:class}

As in paragraph \ref{par:discrete}, we work with an affine Hecke algebra 
$\mc H = \mc H (\mc R, \lambda, \lambda^*,\mb{q})$ where $\mb{q} > 1$ and $\lambda, \lambda^*$
are real-valued. We abbreviate $\mh H = \mh H (\mf t,W,k)$ where $W = W(R), \mf t = \mr{Lie}(T)$
and $k$ is a real-valued parameter function.

In Theorem \ref{thm:3.19} we reduced the classification of irreducible $\mc H$-representations
to a little combinatorics with a groupoid $\mc W_\Xi$ and two substantial subproblems:
\begin{itemize}
\item classify the irreducible discrete series representations $\delta$ of the parabolic
subquotient algebras $\mc H_P$ (modulo the action of $T^P \cap T_P$ via the automorphisms
$\psi_{P,u}$),
\item determine the irreducible quotients of $\pi (\xi)$ for $\xi = (P,\delta,t) \in \Xi^+$.
\end{itemize}
In this section we address both these issues.

\subsection{Analytic R-groups} \
\label{par:Rgroups}

By Proposition \ref{prop:3.15}.a--c the second subproblem above is equivalent to classifying 
the irreducible summands of the completely reducible $\mc H^{P(\xi)}$-representation 
\[
\pi^{P(\xi)} (\xi) := \ind_{\mc H^P}^{\mc H^{P(\xi)}}(\delta \circ \psi_t) .
\]
For that we have to analyse $\End_{\mc H^{P(\xi)}}(\pi^{P(\xi)} (\xi))$, which by 
Proposition \ref{prop:3.15}.e boils down to investigating $\End_{\mc H}(\pi (\xi))$.

For $\xi = (P,\delta,t) \in \Xi^+$ we let $\mc W_\xi$ be the subgroup of $\mc W_{\Xi,PP}$
that stabilizes $\xi$ (up to isomorphism of induction data). From Theorem \ref{thm:3.18}
we know that the intertwining operators $\pi (w,\xi)$ with $w \in \mc W_\xi$ span
$\End_{\mc H}(\pi (\xi))$, but they need not be linearly independent. Knapp and Stein
exhibited a subgroup $\mf R_\xi$ of $\mc W_\xi$ such that the $\pi (w,\xi)$ with 
$w \in \mf R_\xi$ do form a basis of $\End_{\mc H}(\pi (\xi))$.

Let $R_P^+$ be the set of positive roots in $R_P$, with respect to the basis $P$.
Suppose that $\alpha \in R^+ \setminus R_P^+$ and that $P \cup \{ \alpha \}$ is a basis
of a parabolic root subsystem $R_{P \cup \{\alpha\}}$ of $R$. Then we put
\[
\alpha^P = \alpha |_{\mf a^{P*}} \qquad \text{and} \qquad
c_\alpha^P = \prod\nolimits_{\beta \in R^+_{P \cup \{\alpha\}}} c_\beta \in \C (X).
\]
We note that $c_\alpha^P$ is $W_P$-invariant because $W_P$ stabilizes $R_P$ and does not
make positive roots outside $R_P$ negative. Let $\delta \in \Irr (\mc H_P)$ be discrete
series, with central character cc$(\delta) = W_P r$. Then $t \mapsto c_\alpha^P (rt)$
is a rational function on $T^P$, independent of the choice of the representative $r$
for cc$(\delta)$. For $\xi = (P,\delta,t)$ we consider the following subset of $\mf a^{P*}$:
\[
R_\xi = \big\{ \pm \alpha^P : \alpha \in R^+ \setminus R_P^+ \text{ as above, } 
c_\alpha^P \text{ has a non-removable pole at } rt \big\} .
\]
This generalizes the root system $R_t$ from \eqref{eq:3.20}.

\begin{prop}\label{prop:3.20}
\textup{\cite[Proposition 4.5]{DeOp2}}
\enuma{
\item $R_\xi$ is a reduced root system in $\mf a^{P*}$,
\item the Weyl group $W(R_\xi)$ is naturally a normal subgroup of $\mc W_\xi$.
}
\end{prop}

The group $\mc W_\xi$ acts naturally on $\mf a^{P*}$ and stabilizes all the data used to
construct $R_\xi$, so it also acts naturally on $R_\xi$. Clearly $R^+$ determines a set
of positive  roots $R_\xi^+$ in $R_\xi$. We define
\begin{equation}\label{eq:3.16}
\mf R_\xi = \{ w \in \mc W_\xi : w (R_\xi^+) = R_\xi^+ \} .
\end{equation}
As $W(R_\xi)$ acts simply transitively on the collection of positive systems of $R_\xi$:
\begin{equation}\label{eq:3.14}
\mc W_\xi = W(R_\xi) \rtimes \mf R_\xi .
\end{equation}

\begin{thm}\label{thm:3.21}
Let $\xi \in \Xi^+$.
\enuma{
\item For $w \in \mc W_\xi$, the intertwining operator $\pi (w,\xi)$ is scalar if
and only if $w \in W(R_\xi)$.
\item There exists a 2-cocycle $\natural_\xi : \mf R_\xi \times \mf R_\xi \to \C^\times$
(depending on the normalization of the operators $\pi (w,\xi)$ with $w \in \mc W_\xi$)
such that 
\[
\End_{\mc H}(\pi (\xi)) = \mr{span} \{ \pi (w,\xi) : w \in \mf R_\xi \}
\]
is isomorphic to the twisted group algebra $\C [\mf R_\xi ,\natural_\xi]$. The 
multiplication in $\C [\mf R_\xi ,\natural_\xi]$ is as in \eqref{eq:3.10}.
\item Given the normalization of these intertwining operators, we write
\[
\pi^{P(\xi)}(\xi,\rho) = \Hom_{\C [\mf R_\xi ,\natural_\xi]} \big( \rho, \pi^{P(\xi)}(\xi) \big) .
\]
There are bijections
\[
\begin{array}{ccccc}
\Irr (\C [\mf R_\xi ,\natural_\xi]) & \to & \{ \text{irreducible summands of } 
& \to & \{ \text{irreducible quotients of }\\
& & \pi^{P(\xi)}(\xi), \text{ up to isomorphism} \} & &  
\pi (\xi), \text{ up to isomorphism} \} \\
\rho & \mapsto & \pi^{P(\xi)}(\xi,\rho) & \mapsto & L \big( P(\xi), \pi^{P(\xi)}(\xi,\rho) \big)
\end{array}
\]
}
\end{thm}
\begin{proof}
For $\xi = (P,\delta,t)$ with $t \in T^P_\un$, all this (and more) was shown in 
\cite[Theorems 5.4 and 5.5]{DeOp2}. Using Proposition \ref{prop:3.15}.a, the same proofs
work for $\pi^{P(\xi)}(P,\delta,t)$ with any $\xi \in \Xi$, they show the theorem on the
level of $\mc H^{P(\xi)}$. Finally, we apply Proposition \ref{prop:3.15}.c,e.
\end{proof}

\begin{ex} \label{ex:Rgroups}
\begin{itemize}
\item $\mc R$ of type $\widetilde{A_1}$, $\mc H = \mc H (\mc R,\mb q)$. The root system
$R_\xi$ is nonempty only for $\xi = (\emptyset,\triv,1)$, and in that case
$R_\xi = R, \mc W_\xi = W = W(R_\xi), \mf R_\xi = 1$ and $\pi (\xi) =
\ind_{\C [X]}^{\mc H}(\C_1)$ is irreducible.

The R-group $R_\xi$ is nontrivial only for $\xi = (\emptyset,\triv,-1)$, and in that
case $R_\xi = \emptyset, \mc W_\xi = \mf R_\xi = W$. Further $\pi (\xi) = 
\ind_{\C [X]}^{\mc H}(\C_{-1})$ is reducible and $\End_{\mc H}(\pi (\xi)) \cong
\C [\mf R_\xi]$. With the appropriate normalization of the intertwining operator
$\pi (s_\alpha,\xi)$, we have
\[
\pi (\xi,\rho = \triv) = \pi (-1,\triv) ,\quad
\pi (\xi,\rho = \mr{sign}) = \pi (-1,\mr{St}) .
\]
\item $\mc R = \mc R_n, \mc H = \mc H_n (\mb{q})$. For all $\xi = (P,\delta,t) \in \Xi$,
\[
\mc W_\xi = \{ w \in S_n : w (P) = P , w(t) = t \} = W(R_\xi)
\]
and $\mf R_\xi = 1$. Hence 
\[
\End_{\mc H^{P(\xi)}}(\pi^{P(\xi)} (\xi)) \cong \End_{\mc H}(\pi (\xi)) \cong \C
\]
and $\pi (\xi)$ has only one irreducible quotient (as we already saw in several ways).
\item $\mc R$ arbitrary, $\lambda = \lambda^* = 0, \mc H = \mc H (\mc R,1) = \C [X \rtimes W]$.
The only discrete series representation of a parabolic subquotient algebra of this $\mc H$
is the trivial representation of $\mc H_\emptyset = \C$. Hence 
\[
\Xi = \{ (\emptyset,\triv,t) : t \in T \}.
\] 
Further $c_\alpha = 1$ for all $\alpha \in R$, so $R_\xi$ is empty for all $\xi \in \Xi$. 
As $T^\emptyset \cap T_\emptyset = T_\emptyset = \{1\}$,
\[
\mc W_\xi = \mc W_{(\emptyset,\triv,t)} = W_t = R_\xi .
\] 
Here $\End_{\mc H}(\ind_{\C[X]}^{\mc H} (\C_t)) \cong \C [W_t]$ acts on the vector space
$\ind_{\C[X]}^{\mc H} (\C_t) \cong \C [W]$ as the induction, from $W_t$ to $W$ of the right 
regular representation of $W_t$, as can be inferred from \eqref{eq:3.7}.
\end{itemize} \end{ex}

The last example shows that R-groups can be as complicated as $W$ itself. This is in sharp contrast
with the situations for real reductive groups and for classical $p$-adic groups, where all
R-groups are abelian 2-groups. In all examples that we are aware of, the 2-cocycle $\natural_\xi$
of $\mf R_\xi$ is a coboundary, so that $\C[\mf R_\xi,\natural_\xi]$ is isomorphic to $\C [\mf R_\xi]$.
It would be interesting to know whether or not this is always true for affine Hecke algebras.

By Proposition \ref{prop:3.15}.b, the $\mc H$-module
\begin{multline}\label{eq:3.15}
\ind_{\mc H^{P(\xi)}}^{\mc H}(\pi^{P(\xi)} (\xi,\rho)) = \ind_{\mc H^{P(\xi)}}^{\mc H}
\big( \Hom_{\C [\mf R_\xi,\natural_\xi]}(\rho, \pi^{P(\xi)}(\xi)) \big) = \\
\Hom_{\C [\mf R_\xi,\natural_\xi]} \big( \rho, \ind_{\mc H^{P(\xi)}}^{\mc H} \pi^{P(\xi)} (\xi) \big) 
= \Hom_{\C [\mf R_\xi,\natural_\xi]} \big( \rho,\pi (\xi)) = \pi (\xi,\rho)
\end{multline} 
from Theorem \ref{thm:3.21}.c is standard. Its unique irreducible quotient is $L(P(\xi),\pi^{P(\xi)}
(\xi,\rho) )$. By Proposition \ref{prop:3.15}.c,e, every standard $\mc H$-module is of the form
\eqref{eq:3.15}, for some $\xi \in \Xi^+$ and $\rho \in \Irr (\C [\mf R_\xi,\natural_\xi])$.

When $\xi,\xi' \in \Xi^+, \rho \in \Irr (\C [\mf R_\xi,\natural_\xi])$ and $w \in \mc W_\Xi$
with $w (\xi) \cong \xi'$, we can define $\rho' \in  \Irr (\C [\mf R_{w(\xi)},\natural_{w(\xi)}])$ by
\[
\rho' \big( \pi (w',\xi') \big) := \rho \big( \pi (w,\xi)^{-1} \pi (w',\xi') \pi (w,\xi) \big) .
\]
Although $\pi (w,\xi)$ is only defined up to a scalar, the formula for $\rho'$ is independent of
the choice of a normalization. 

We denote this $\rho'$ by $w(\rho)$, and we say that $(\xi,\rho)$ and $(w(\xi),w(\rho))$ are 
$\mc W_\Xi$-associate. It is not clear whether this comes from a groupoid action on a set containing
all $(\xi,\rho)$ as above, because $\mc W_\Xi$ does not stabilizes the set of positive induction 
data $\Xi^+$ and we did not define R-groups for non-positive induction data.

Let us summarise some properties of standard $\mc H$-modules.

\begin{cor}\label{cor:3.23}
Write $\Xi^+_e = \{ (\xi,\rho) : \xi \in \Xi^+, \rho \in \Irr (\C [\mf R_\xi,\natural_\xi]) \}$. 
\enuma{
\item The set of standard $\mc H$-modules (up to isomorphism) is 
parametrized by $\Xi^+_e$ up to $\mc W_\Xi$-association.
\item Every standard $\mc H$-module has a unique irreducible quotient.
\item For every irreducible $\mc H$-representation $\pi$ there is a unique (up to isomorphism)
standard $\mc H$-module that has $\pi$ as quotient.
\item There are bijections}
\[
\begin{array}{ccccc}
\Xi^+_e / \mc W_\Xi & \to & \{ \text{standard } \mc H\text{-modules} \} & \to & \Irr (\mc H) \\
(\xi,\rho) & \mapsto & \pi (\xi,\rho) & \mapsto & L ( P(\xi), \pi^{P(\xi)}(\xi,\rho) ) .
\end{array}
\]
\end{cor}
\begin{proof}
(a) follows from Theorem \ref{thm:3.18}. \\
(b) is already contained in Theorem \ref{thm:3.21}.\\
(c) Recall from the remark after Theorem \ref{thm:3.19} that $\pi$ determines a unique 
parabolically induced representation $\pi (\xi)$, with $\xi \in \Xi^+$, that has $\pi$ as quotient.\\
(d) is a consequence of parts (a), (b) and (c).
\end{proof}

Corollary \ref{cor:3.23} provides a classification of $\Irr (\mc H)$ in terms of induction data
and R-groups. The role of the R-groups is quite subtle, firstly because it can be hard to determine
them, secondly because potentially a non-trivial 2-cocycle can be involved in the $\mc H$-endomorphism
algebra of a parabolically induced representation. 

Sometimes it is easier to work with standard modules than with irreducible representations, for their
structure is more predictable.

\begin{ex} \begin{itemize}
\item For $\mc H (\mc R,\mb{q})$ with $\mc R$ of type $\widetilde{A_1}$, almost all standard 
modules are irreducible. The only reducible standard $\mc H$-module is $\pi (\emptyset,\triv,\mb{q}) 
= \ind_{\C[X]}^{\mc H}(\C_{\mb{q}})$, which has triv as irreducible quotient.

\item For $\mc H_n (\mb{q})$ all the groups $\mf R_\xi$ are trivial, so the standard modules 
are just the parabolically induced representations $\pi (\xi)$ with $\xi \in \Xi^+$.

\item When $\mc H = \mc H (\mc R,1) = \C [X \rtimes W]$ (any $\mc R$), there is a standard module
$\pi (\emptyset,\triv,t,\rho)$ for every $t \in T$ and $\rho \in \Irr (W_t)$. In the notation
from Paragraph \ref{par:q=1} it equals $\pi (t,\rho^*)$. In view of Theorem \ref{thm:2.2}, these
$\mc H$-representations are irreducible and there are no other standard modules.
\end{itemize} \end{ex}

To the best of our knowledge, the theory of R-groups for graded Hecke algebras has never been
written down explicitly. It can be deduced readily from \cite{DeOp2} and the algebra
isomorphisms from Theorem \ref{thm:3.26}. In this setting the $c_\alpha$-function for one
root $\alpha \in R$ becomes
\[
\tilde{c_\alpha} (\lambda) = \frac{\inp{\alpha}{\lambda} + k(\alpha)}{\inp{\alpha}{\lambda}}
\qquad \lambda \in \mf t.
\]
For every $\tilde \xi = (P,\delta,\lambda) \in \tilde{\Xi}^+$ we obtain an R-group 
$\mf R_{\tilde \xi}$ and a 2-cocycle $\natural_{\tilde \xi}$ such that 
\begin{equation}\label{eq:4.4}
\End_{\mh H}(\pi (\tilde \xi)) \cong \C [\mf R_{\tilde \xi}, \natural_{\tilde \xi}] .
\end{equation}
We will see in Lemma \ref{lem:6.13} that $\natural_{\tilde \xi}$ is always trivial,
except maybe for some instances with root systems of type $F_4$. However, this triviality 
does not automatically generalize to affine Hecke algebras (as in Theorem \ref{thm:3.21}.b)
because the reduction to graded Hecke algebras could pick up groups $\Gamma_t$ in Theorem 
\ref{thm:3.25}.

The standard module associated to $\tilde \xi \in \tilde{\Xi}^+$ and $\rho \in \Irr
(\C [\mf R_{\tilde \xi}, \natural_{\tilde \xi}] )$ is
\begin{equation}\label{eq:3.31}
\pi (\tilde \xi, \rho) = 
\Hom_{\C [\mf R_{\tilde \xi}, \natural_{\tilde \xi}]} (\rho,\pi (\tilde \xi)) .
\end{equation}
Now Theorem \ref{thm:3.21} and Corollary \ref{cor:3.23} apply to $\mh H$, and they provide
bijections
\begin{align} \nonumber
\big\{ (\tilde \xi,\rho) : \tilde \xi \in \tilde{\Xi}^+, \rho \in \Irr
(\C [\mf R_{\tilde \xi}, \natural_{\tilde \xi}] ) \big\} \big/ \mc W_{\tilde \Xi}  
\; \to \; \{ \text{standard } \mh H\text{-modules} \} \; \to \;\Irr (\mh H) \\
\label{eq:3.19} (\tilde \xi,\rho) \quad \mapsto \quad \pi (\tilde \xi,\rho) 
\quad \mapsto \quad \text{irreducible quotient of } \pi (\tilde \xi,\rho) .
\end{align}

\subsection{Residual cosets} \

The most significant step towards the classification of discrete series $\mc H$-\\
representations is the determination of their central characters, which was achieved by 
Opdam \cite{Opd-Sp}. Consider the following rational function on $T$:
\begin{equation}\label{eq:4.1}
c_R = \prod_{\alpha \in R} c_\alpha = \prod_{\alpha \in R}
\frac{\big( \theta_\alpha - \mb{q}^{(- \lambda^* (\alpha) - \lambda (\alpha))/2}\big)
\big( \theta_\alpha + \mb{q}^{(\lambda^* (\alpha) - \lambda (\alpha))/2}\big)}{
(\theta_\alpha - 1) (\theta_\alpha + 1)} .
\end{equation}
Its counterpart for $\mh H (\mf t,W,k)$ is
\[
\tilde{c}_R = \prod_{\alpha \in R} \tilde{c}_\alpha = 
\prod_{\alpha \in R} (\alpha + k (\alpha)) \alpha^{-1} .
\]
\begin{defn}\label{def:4.1}
Let $L \subset T$ be a coset of a complex algebraic subtorus of $T$. We call $L$ a residual
coset (with respect to $R$ and $q$) if the zero order of $c_R$ along $L$ is at least the
(complex) codimension of $L$ in $T$. 

An affine subspace $\mf l \subset \mf a$ is called a residual subspace (with respect to 
$R$ and $k$) if the zero order of $\tilde{c}_R$ along $\mf l$ is at least the (real)
codimension of $\mf l$ in $\mf a$. 

A residual point is a residual coset/subspace of dimension zero.
\end{defn}

For any $L$ or $\mf l$ as above, its zero order is always at most its codimension
\cite[Corollary A.12]{Opd-Sp}. Hence we may replace "is at least" by "equals" in Definition
\ref{def:4.1}. This also implies that residual points can only exist if $R$ spans $\mf a^*$.
The collection of residual cosets/subspaces is stable under $W$, because $R$ and $q/k$ are so.

\begin{ex}
\begin{itemize}
\item $T$ itself is always a residual coset, and $\mf a$ itself is always a residual subspace.
\item There are no residual points for $\mc H$ (resp. for $\mh H$) if $R \neq \emptyset$ and
$\lambda = \lambda^* = 0$ (resp. $k = 0$).
\item Consider $\mf a = \mf a^* = \{x \in \R^n : x_1 + \cdots + x_n = 0\},
R = A_{n-1} = \{ e_i - e_j : i \neq j \}, W = S_n, k (\alpha) = k \in \R^\times$.
There is just one $S_n$-orbit of residual points for $\mh H$, and it contains
\[
(k(1-n)/2,k(3-n)/2,\ldots,k(n-1)/2) .
\]
This point is the $\mc O (\mf t)$-character of the Steinberg representation of $\mh H$, which
by definition is onedimensional and restricts on $\C[W]$ to the sign representation.
\item Take $\mf a = \mf a^* = \R^2, R = B_2 = \{\pm e_1 \pm e_2\} \cup \{ \pm e_i \},
W = W(B_2) \cong D_4, k(\pm e_1 \pm e_2) = k_1, k (\pm e_i) = k_2$. There are at most two 
$W$-orbits of residual points for $\mh H (\C^2,W(B_2),k)$, represented by $(k_1 + k_2,k_2)$
and $(k_1 - k_2,k_2)$. These points are indeed residual if 
\[
k_1 k_2 (k_1 + 2k_2) (k_1 + k_2) (k_1 - 2 k_2) (k_1 - k_2) \neq 0.
\]
\end{itemize} \end{ex}
The crucial property of residual points is:

\begin{thm}\label{thm:4.2}
\textup{\cite[Lemma 3.31]{Opd-Sp}} \\
Let $\delta \in \Irr (\mc H)$ be discrete series. Then all its $\C [X]$-weights are residual
points for $(R,q)$. Conversely, if $t \in T$ is a residual point for $(R,q)$, then there 
exists a discrete series $\mc H$-representation with central character $W t$.
\end{thm}

\begin{ex}
Consider the root datum $\mc R$ of $PGL_n (\C)$, with 
\[
X = \Z^n / \Z (1,1,\ldots,1), Y = \{ y \in \Z^n : y_1 + \cdots + y_n = 0\},
\] 
$R = R^\vee = A_{n-1}$ and $W = S_n$. Notice that $T^W \cong \Z / n \Z$, generated by 
\[
\zeta_n : x \mapsto \exp (2 \pi i (x_1 + \cdots + x_n)).
\]
For $\mb q \neq 1$, $\mc H (\mc R,\mb q)$ admits a unique $S_n \times T^W$-orbit of residual
points, one such point being 
\[
t_{\mb q} = \big( \mb{q}^{(1-n)/2},\mb{q}^{(3-n)/2},\ldots,\mb{q}^{(n-1)/2} \big).
\]
For $\mb q > 1$, this is the unique $\C [X]$-weight of the Steinberg representation of
$\mc H (\mc R,\mb{q})$--which is defined just like St for $\mc H_n (\mb{q})$ in \eqref{eq:2.10}
and \eqref{eq:2.14}. Similarly $\zeta_n t_{\mb q}$ is the $\C[X]$-weight of the discrete series
representation St$\otimes \zeta_n$.
\end{ex}

There is a general method to construct residual cosets from residual points for subquotient
algebras. Namely, let $P \subset \Delta$ and let $r \in T_P$ be a residual point for $\mc H_P$.
Then $T^P r$ is a residual coset for $\mc H$ \cite[Proposition A.4]{Opd-Sp}. Up to the action
of $W$, every residual coset is of this form.

From that, Theorem \ref{thm:4.2} and Lemma \ref{lem:3.3} we deduce: for any induction datum
$(P,\delta,t) \in \Xi$, every weight of $\pi (P,\delta,t)$ lies in a residual coset of the same
dimension as $T^P$.

Now we relate residual cosets to residual subspaces for graded Hecke algebras.
By definition every residual coset for $\mc H$ can be written as $L = u \exp (\lambda) T_L$,
where $u \in T_un, \lambda \in \mf a$ and $T_L \subset T$ is a complex algebraic subtorus.

\begin{prop}\label{prop:4.3}
\textup{\cite[Theorem A.7]{Opd-Sp}}
\enuma{
\item With the above notations, $\lambda + \log |T_L| \subset \mf a$ is a residual subspace for\\
$\mh H (\mf t, W(R_u),k_u)$. 
\item Every residual subspace for $\mh H (\mf t, W(R_u),k_u)$ arises in this way. 
\item When $T_L + T^P$ and $u \exp (\lambda)$ is a residual point for $\mc H_P$,
$\lambda \in \mf a_P$ is a residual point for $\mh H (\mf t_P, W(R_{P,u}),k_u)$. In this
case $R_{P,u}$ has the same rank as $R_P$, namely $|P|$.
}
\end{prop}

\begin{ex} \label{ex:4.1}
Take $X = Y = \Z^2, R = B_2, R^\vee = C_2 = \{ \pm e_1 \pm e_2 \} \cup \{\pm 2 e_1, \pm 2 e_2\}$
and $\Delta = \{e_1 - e_2, e_2\}$. We write
\[
q_1 = \mb{q}^{\lambda (\pm e_1 \pm e_2)}, q_2 = \mb{q}^{\lambda (\pm e_i)}, 
q_0 = \mb{q}^{\lambda^* (\pm e_i)} .
\]
The points $u \in T_\un$ with $R_u$ of rank 2 are $(1,1), (-1,-1), (1,-1)$ and $(-1,1)$, the last
two being $W$-associate. There are at most 5 $W$-orbits of residual points, summarised in the 
following table:
\[
\begin{array}{c|c|c|c}
u & (1,1) & (-1,-1) & (1,-1) \\
R_u & B_2 & B_2 & A_1 \times A_1 \\
k_u & \log (q_1), k_2 = {\ds \frac{\log (q_2 q_0)}{2}} & 
\log (q_1), k_2 = {\ds \frac{\log (q_2 q_0^{-1})}{2}} 
& {\ds \frac{\log (q_2 q_0)}{2}}, {\ds \frac{\log (q_2 q_0^{-1})}{2}} \\
\!\! \text{residual} & ( q_1^{-1} e^{-k_2}, e^{-k_2}) & 
(-q_1^{-1} e^{-k_2},-e^{-k_2}) & ((q_2 q_0)^{-1/2}, \hspace{1cm} \\
\text{points} & ( q_1^{-1} e^{k_2},e^{-k_2}) & 
 ( -q_1^{-1} e^{k_2},-e^{-k_2}) & \hspace{1cm} - q_2^{-1/2} q_0^{1/2}) 
\end{array}
\]
Here we give $k_u$ in terms of its values on a basis of $R_u$. For generic parameters $q_1, q_2,
q_0$, all the five points of $T$ in the above table are indeed residual, and each of them
represents the central character of a unique discrete series representation \cite[\S 5.5]{SolAHA}.
\end{ex}

As a consequence of Proposition \ref{prop:4.3}, everything we said before about residual cosets
for $\mc H$ can be translated to residual subspaces for graded Hecke algebras. Combining that with
Lemma \ref{lem:3.28}, we find that every $\mc O (\mf t)$-weight of a discrete series 
$\mh H$-representation is a residual point in $\mf a$.\\

\begin{ex}
We continue Example \ref{ex:4.1}, but now for the graded Hecke algebra $\mh H$ built from
$\mf a = \mf a^* = \R^2, \mf t = \mf t^* = \C^2$, $R = B_2, W = W(B_2)$, $\Delta = \{ \alpha = 
e_1 - e_2, \beta = e_2\}$, $k(\pm e_1 \pm e_2) = k_1 > 0, k (\pm e_i) = k_2 > 0$. 
The residual point $(-k_1 - k_2, -k_2)$ is the $\mc O (\mf t)$-character
of the Steinberg representation of $\mh H$, which is discrete series and restricts to the sign
character of $\C[W]$. When $k_1 < 2 k_2$, the residual point $(k_1 - k_2, -k_2)$ is the
$\mc O (\mf t)$-character of a onedimensional discrete series $\mh H$-representation $\delta$.
Its restriction to $\C [W]$ is given by $\delta (s_{e_2}) = -1, \delta (s_{e_1 - e_2)} = 1$.

With Theorem \ref{thm:3.26} and Corollary \ref{cor:3.23} we can complete the classification of 
$\Irr (\mh H)$. To this end we note that
\[
\mh H_{\{\alpha\}} = \mh H (\C \alpha, \langle s_\alpha \rangle, k_1) \quad \text{and} \quad
\mh H_{\{\beta\}} = \mh H (\C \beta, \langle s_\beta \rangle, k_2) .
\]
These algebras have a unique discrete series representation, namely St. Further
\[
\mf t^{\{\alpha\}} = \C (e_1 + e_2), \mf a^{\{\alpha\}+} = \R_{\geq 0} (e_1 + e_2),
\mf t^{\{\beta\}} = \C e_1, \mf a^{\{\beta\}+} = \R_{\geq 0} e_1
\]
and $\mf a^+ = \R_{\geq} e_1 + \R_{\geq 0}(e_1 + e_2)$.
All the R-groups $\mf R_{\tilde \xi}$ for $\mh H$ are trivial, so \eqref{eq:3.19} provides a
bijection from $\tilde{\Xi}^+ / \mc W_{\tilde \Xi}$ to $\Irr (\mh H)$, where
\begin{multline*}
\tilde{\Xi}^+ = \big\{ (\emptyset,\triv,\lambda) : \lambda \in i \mf a + \mf a^+ \big\} \; \cup \;
\big\{ (\{\alpha\},\mr{St},\lambda) : \lambda \in i \mf a^{\{\alpha\}} + \mf a^{\{\alpha\}+} \big\} \\
\cup \; \big\{ (\{\beta\},\mr{St},\beta) : \lambda \in i \mf a^{\{\beta\}} + \mf a^{\{\beta\}+} \big\}
\; \cup \; \{\mr{St},\delta\} .
\end{multline*}
The groupoid $\mc W_{\tilde \Xi}$ consists of the groups
\[
\mc W_{\tilde \Xi,\emptyset \emptyset} = W,
\mc W_{\tilde \Xi,\{\alpha\} \{\alpha\}} = \langle s_{e_1 + e_2} \rangle,
\mc W_{\tilde \Xi,\{\beta\} \{\beta\}} = \langle s_{e_1} \rangle,
\mc W_{\tilde \Xi,\Delta \Delta} = \{\mr{id}\} .
\]
The action of $\mc W_{\tilde \Xi}$ on $\tilde \Xi$ makes some of the $(P,\delta,\lambda) \in 
\tilde{\Xi}^+$ with $\Re(\lambda) \in \partial (\mf a^{P+})$ $\mc W_{\tilde \Xi}$-associate,
for instance $(\emptyset,\triv, (1,i))$ and $(\emptyset,\triv, (1,-i))$.
\end{ex}

In general it is easy to classify all points $u \in T_\un$ for which $R_u$ has full rank in $R$,
in terms of the affine Dynkin diagram of $\mc R$ \cite[Lemma A.8]{Opd-Sp}.
Recall that the relation between the representations of $\mh H (\mf t,W(R_u),k_u)$ and of
$\mh H (\mf t,W(R_u),k_u) \rtimes \Gamma_u$ is well-understood from Clifford theory. Thus the
classification of discrete series $\mc H$-representation boils down to two tasks:
\begin{itemize}
\item classify all residual points for $\mh (\mf t,W,k)$, where $k$ is any real-valued
parameter function,
\item for a given residual point $\lambda \in \mf a$, classify the discrete series 
$\mh H$-representations with central character $W \lambda$. 
\end{itemize}
In view of the isomorphism \eqref{eq:1.22}, it suffices to this when $R$ is irreducible.
The residual points for $\mh H(\mf t,W,k)$ with such $R$ and $k$ have been classified completely 
in \cite[\S 4]{HeOp}. They are always linear expressions $f(k)$ in the parameters $k(\alpha)$
for $\alpha \in R$. For a given such $f$, $f(k)$ is residual with respect to $R$ and $k$ for 
almost all $k : R \to \R$. We say that a parameter function $k$ is generic if all potentially 
residual points $f(k)$, for the $\mh H (\mf t,W(R_P),k)$ with $P \subset \Delta$, are 
really residual for this $k$, and are all different.

\begin{thm}\label{thm:4.4}
\textup{\cite[Theorems 3.4 and 7.1]{OpSo}} \\
Let $R \subset \mf a^*$ be an irreducible root system that spans $\mf a^*$.
\enuma{
\item Let $k : R \to \R$ be a generic parameter function. The central character map gives 
a bijection from the set of irreducible discrete series representations of $\mh H (\mf t,W(R),k)$
to the set of $W$-orbits of residual points for $R$ and $k$ (except when $R \cong F_4$, then one fibre
of this map has two elements).
\item For a non-generic parameter function $k' : R \to \R$ and a residual point $\lambda \in \mf a$,
consider the collection of generic residual points $\{ f_i (k) \}_i$ that specialize to $\xi$
at $k' = k$. For $k$ close to $k'$ in the space of all parameter functions $R \to \R$, 
there is a natural bijection between:
\begin{itemize}
\item the set of irreducible discrete series representations of $\mh H (\mf t,W(R),k)$ 
with central character in $\{ W f_i (k) \}_i$,
\item the set of irreducible discrete series representation of $\mh H (\mf t,W(R),k')$ 
with central character $W \lambda$.
\end{itemize}
More explicitly, every $\mh H (\mf t,W(R),k)$-representation of the indicated kind is of part
of a unique continuous family of such representations, one representation for each $k$ in some
neighborhood of $k$. The above bijection matches the members of such a continuous family at 
$k$ and at $k'$.
}
\end{thm}

From Theorem \ref{thm:4.4} one obtains a complete classification of discrete series representations
of affine Hecke algebras with positive parameters. However, we have to point out that this does not
yet achieve an actual classification of all irreducible representations. The problem is that
it can remain difficult to effectively compute the R-groups $\mf R_{P,\delta,t}$ and their 
2-cocycles $\natural_{P,\delta,t}$ from Paragraph \ref{par:Rgroups}.\\

We conclude this section with a discussion of the residual points for $\mh H$ in the most
intricate case, for root systems of type $B_n$. 
We take $\mf a = \mf a^* = \R^n, R = B_n = \{ \pm e_i : i = ,1\ldots,n \} \cup 
\{ \pm e_i \pm e_j : i \neq j \}$ and we write $k(\pm e_i \pm e_j) = k_1, k (\pm e_i) = k_2$.
For every partition $\vec{n} = (n_1,n_2,\ldots,n_d)$ of $n$ we construct a point 
$\lambda (\vec{n},k) \in \mf a$, or rather a $S_n$-orbit in $\mf a$, in the following way. 
Draw the Young diagram, with first column of $n_1$ boxes, second column of $n_2$ boxes and so on.
Label the boxes from $b_1$ to $b_n$ in some way (to does not matter how, for another labelling 
will produce a point in the same $S_n$-orbit in $\mf a$). We define the height of a box $b$
in column $i$ and row $j$ to be $h(b) = j - i$ and we write
\begin{equation}\label{eq:4.3}
\lambda (\vec{n},k) = (h(b_1) k_1 + k_2, h(b_2) k_1 + k_2,\ldots, h(b_n) k_1 + k_2) .
\end{equation}
For example, when $n=2$ we have
\[
\lambda ((2),k) = (k_1 + k_2,k_2) ,\quad \lambda( (1,1),k) = (-k_1 + k_2,k_2)
\]
Every residual point for $\mh H (\C^n,W(B_n),k)$ is $W(B_n)$-associate to a $\lambda (\vec{n},k)$
\cite[\S 4]{HeOp}. For most parameters $k_i$ indeed all these points of $\mf a$
are residual, but not for all parameters. An extreme case is $k_1 = k_2 = 0$, then there are no
residual points. 

A parameter function $k : B_n \to \R$ is called generic if 
\begin{equation}\label{eq:4.2}
k_1 k_2 \prod\nolimits_{j=1}^{2(n-1)} (j k_1 + 2 k_2) (j k_1 - 2 k_2) \neq 0 .
\end{equation}
For generic $k$, all the $\lambda (\vec{n},k)$ are residual \cite[Proposition 4.3]{HeOp}
and they belong to different $W(B_n)$-orbits \cite[proof of Theorem 7.1]{OpSo}. 
On the other hand, when $k$ is not generic, some of the $\lambda (\vec{n},k)$ 
are not residual, and some of them may belong to the same $W(B_n)$-orbit.

\section{Geometric methods}

We survey some of results on affine Hecke algebras obtained with methods from complex algebraic
geometry. In many cases, these provide a complete classification of standard modules and of
irreducible representations. 

\subsection{Equivariant K-theory} \

In this paragraph we discuss equal label affine Hecke algebras, that is, with a single
parameter $\mb q$.
Recall from Paragraph \ref{par:defIHA} that these algebras are especially important 
because they classify representations of reductive $p$-adic groups with vectors fixed by an
Iwahori subgroup. Lusztig \cite{Lus-KT} discovered that such 
affine Hecke algebras can be realized as the equivariant K-theory of a suitable complex algebraic
variety. Then its representations can be analysed in algebro-geometric terms, and that leads to 
a beautiful construction and parametrization of all irreducible representations.

Let $G$ be a connected complex reductive group with a maximal torus $T$, and let $\mc R (G,T)$ 
be the associated root datum. We define $\mc H (G,T)$ to be like $\mc H (\mc R (G,T),\mb q)$, but 
with $\mb q$ replaced by an invertible formal variable $\mb z$. As vector spaces
\[
\mc H (G,T) = \C [X^* (T)] \otimes_\C \C [W] \otimes_\C \C [\mb z, \mb z^{-1}] ,
\]
where $W = W (G,T)$. Let $\mc B$ be the variety of Borel subgroups of $G$, it is isomorphic to 
$G/B$ for one Borel subgroup $B$. 

The upcoming constructions work best when the derived group $G_\der$ of $G$ 
is simply connected, so we assume that in this paragraph (unless explicitly mentioned otherwise). 
A main role is played by the Steinberg variety of $G$ from \cite[\S 3.3]{KaLu}:
\[
\mc Z := \{ (B,u,B') \in \mc B \times G \times \mc B : u \in B \cap B' \text{ unipotent} \} .
\]  
The group $G \times \C^\times$ acts on $\mc Z$ by 
\[
(g,\lambda) (B,u,B') = (g B g^{-1}, g u^{\lambda^{-1}} g^{-1}, g B' g^{-1}) .
\]
Note that $u^{\lambda^{-1}}$ is defined because $u$ is unipotent. This $\C^\times$-action might appear
ad hoc, but it is indispensable to obtain Hecke algebras. Without it, we could at best build the
$G$-equivariant K-group $K^G (\mc Z)$, which turns out be isomorphic to $\C [X] \rtimes W$
\cite[Theorem 7.2]{ChGi}. According to \cite[Theorem 3.5]{KaLu} and \cite[Theorem 7.2.5]{ChGi},
there is a natural isomorphism
\begin{equation}\label{eq:5.1}
K^{G \times \C^\times}(\mc Z) \cong \mc H (G,T).
\end{equation}
The $\mb z$'s in $\mc H (G,T)$ are due to the $\C^\times$-action on $\mc Z$. The ring of regular
class functions on $G$ is
\[
R (G) = \mc O (G)^G \cong \mc O (T/W) .
\]
When we regard $\mb z$ as the identity representation of $\C^\times$, we can write the rings of
regular class functions on $\C^\times$ and on $G \times \C^\times$ as
\[
R (\C^\times) = \C [\mb z,\mb z^{-1}], \qquad
R(G \times \C^\times) = \mc O (G)^G \otimes_\C \C [\mb z, \mb z^{-1}] .
\]
By construction \cite[\S 5.2]{ChGi}, $R(G \times \C^\times)$ acts naturally on 
$K^{G \times \C^\times} (\mc Z)$. On the other hand, a variation on \eqref{eq:1.14} shows that
$R(G \times \C^\times)$ is also naturally isomorphic to the centre of $\mc H (G,T)$. With these
identifications the isomorphism \eqref{eq:5.1} is $R(G \times \C^\times)$-linear. For any
$\mb q \in \C^\times$ we can specialize \eqref{eq:5.1} to an isomorphism
\begin{equation}\label{eq:5.2}
K^{G \times \C^\times} (\mc Z) \otimes_{\C [\mb z, \mb z^{-1}]} \C_{\mb q} 
\cong \mc H (\mc R(G,T),\mb q) .
\end{equation}
Further, let $t \in G$ be a semisimple element and denote the associated onedimensional
representation of $R(G \times \C^\times)$ by $\C_{t,\mb q}$. Let $\mc Z^{t,q}$ be the subvariety of
$\mc Z$ fixed by $(t,\mb q) \in G \times \C^\times$. By \cite[p. 414]{ChGi} there is an isomorphism
\begin{equation}\label{eq:5.3}
K^{G \times \C^\times} (\mc Z) \otimes_{R(G \times \C^\times)} \C_{t,\mb q} \cong 
K (\mc Z^{t,\mb q}) \otimes_\Z \C \cong H_* (\mc Z^{t,\mb q}, \C) .
\end{equation}
The construction of $K^{G \times \C^\times} (\mc Z)$-modules is performed most naturally with
Borel--Moore homology (that is equivalent to the constructions with equivariant K-theory in
\cite{KaLu}). 

Let $u \in G$ be a unipotent element such that $t u t^{-1} = u^{\mb q}$ and let $\mc B^{t,u} \subset
\mc B$ be the subvariety of Borel subgroups that contain
$t$ and $u$. The convolution product in Borel--Moore homology \cite[Corollary 2.7.42]{ChGi}
provides an action of $H_* (\mc Z^{t,\mb q},\C)$ on $H_* (\mc B^{t,u},\C)$. This and \eqref{eq:5.3} 
make $H_* (\mc B^{t,u},\C)$ into a $K^{G \times \C^\times}(\mc Z)$-module, usually reducible. 

By \cite[Lemma 8.1.8]{ChGi} these constructions commute with the $G$-action, in the sense that
\[
H_* (\mr{Ad}_g)^* : H_* (\mc B^{t,u},\C) \to H_* (\mc B^{g t g^{-1}, g u g^{-1}},\C) 
\] 
intertwines the $K^{G \times \C^\times}(\mc Z)$-actions. In particular $Z_G (t,u)$, the 
centralizer of $\{t,u\}$ in $G$, acts on $H_* (\mc B^{t,u},\C)$ by $K^{G \times \C^\times}(\mc Z)
$-intertwiners. The neutral component of $Z_G (t,u)$ acts trivially, so we may regard it as an 
action of the component group $\pi_0 (Z_G (t,u))$. That can be used to decompose the module
$H_* (\mc B^{t,u},\C)$. Let $\rho$ be an irreducible representation of $\pi_0 (Z_G (t,u))$
which occurs in $H_* (\mc B^{t,u},\C)$. Then
\[
K_{t,u,\rho} := \Hom_{\pi_0 (Z_G (t,u))} \big( \rho, H_* (\mc B^{t,u},\C)\big)
\]
is a nonzero $K^{G \times \C^\times}(\mc Z)$-module, called standard in \cite[5.12]{KaLu} and
\cite[Definition 8.1.9]{ChGi}. Since the action factors via \eqref{eq:5.3}, $K_{t,u,\rho}$
can be regarded as a\\ $\mc H (\mc R (G,T),\mb q)$-representation with central character $W t$.
In view of their role in \cite{KaLu}, data $(t,u,\rho)$ with the above properties are usually
called Kazhdan--Lusztig triples for $(G,\mb q)$.

\begin{thm}\label{thm:5.1}
\textup{\cite[Theorem 7.12]{KaLu}} \\
Recall that $G_\der$ is simply connected and let $\mb q \in \C^\times$ be of infinite order. 
\enuma{
\item For every Kazhdan--Lusztig triple $(t,u,\rho)$, the $\mc H (\mc R (G,T),\mb q)$-module 
$K_{t,u,\rho}$ has a unique irreducible quotient $L_{t,u,\rho}$.
\item Every irreducible $\mc H (\mc R (G,T),\mb q)$-module is of the form $L_{t,u,\rho}$, for 
a suitable Kazhdan--Lusztig triple.
\item Let $(t',u',\rho')$ be another Kazhdan--Lusztig triple. Then
$L_{t,u,\rho} \cong L_{t',u',\rho'}$ if and only if there exists a $g \in G$ such that
$t' = g t g^{-1}, u' = g u g^{-1}$ and $\rho' = \rho \circ \mr{Ad}(g^{-1})$.
}
\end{thm}

This major result, also shown later in a somewhat different way in \cite[Theorem 8.1.16]{ChGi},
comes with a lot of extras. 
Firstly, suppose that $L$ is a standard Levi subgroup of $G$ and that it contains $\{t,u\}$. Then 
\[
K^{L \times \C^\times} (\mc Z_L) \otimes_{\C [\mb z,\mb z^{-1}]} \C_{\mb q} 
\cong \mc H (\mc R (L,T),\mb q)
\]
embeds naturally in $\mc H (\mc R(G,T),\mb q)$ and, by \cite[Theorem 6.2]{KaLu}:
\[
H_* (\mc B^{t,u},\C) \cong 
\ind_{\mc H (\mc R(L,T),\mb q)}^{\mc H (\mc R(G,T),\mb q)} H_* (\mc B_L^{t,u},\C) .
\]
For the second extra we suppose that $\mb q \in \R_{>1}$. By \cite[Theorem 8.3]{KaLu} and
\cite[Proposition 9.3]{ABPSprin} the $\mc H (\mc R(G,T),\mb q)$-module $L_{t,u,\rho}$ is 
essentially discrete series if and only if $\{t,u\}$ is not contained in any Levi subgroup 
of any proper parabolic subgroup of $G$. 

With these two extras at hand, we can compare $K_{t,u,\rho}$ with the more analytic approach from
Sections \ref{sec:rep} and \ref{sec:class}. Let $L$ be a Levi subgroup of $G$ which contains 
$\{t,u\}$ and is minimal for that property. Upon conjugating everything by an element of $G$,
we may assume that $L$ is standard, that $t \in T$ and that $\log |t|$ is as positive as possible
in its $W$-orbit. Let $\rho_L \in \Irr \big( \pi_0 (Z_L (t,u)) \big)$ be an irreducible constituent
of $\rho |_{\pi_0 (Z_L (t,u))}$. Then 
\begin{equation}\label{eq:5.4}
\Hom_{\pi_0 (Z_L (t,u))} \big( \rho_L, H_* (\mc B_L^{t,u},\C) \big) 
\end{equation}
is an irreducible essentially discrete series $\mc H (\mc R (L,T),\mb q)$-representation. By 
Theorem \ref{thm:3.13} it is of the form $\delta \circ \psi_{t'}$, where $\delta$ is discrete
series. Then $\xi = (\Delta_L,\delta,t') \in \Xi_+$ by the assumption on $t$. The induction of
\eqref{eq:5.4} to $\mc H (\mc R(G,T),\mb q)$ contains $K_{t,u,\rho}$ as a direct summand. In view 
of Theorem \ref{thm:3.21}.b, this summand must be picked out by an irreducible representation 
$\rho'$ of $\C [\mf R_\xi,\natural_\xi]$. We conclude that $K_{t,u,\rho} \cong \pi (\xi,\rho')$, 
a standard module in the sense of Definition \ref{def:3.9}.

This argument also works in the opposite direction, and then it shows that the standard modules
from Definition \ref{def:3.9} are precisely the standard modules from \cite{KaLu} and \cite{ChGi}.\\

Recall that in Theorem \ref{thm:5.1} the complex reductive group $G$ has simply connected derived
group $G_\der$. Without that assumption on $G$, $K^{G \times \C^\times}$ behaves less well.
Nevertheless, the parametrization of irreducible $\mc H (\mc R(G,T),\mb q)$-representations obtained
in Theorem \ref{thm:5.1} is valid for any complex reductive group $G$. This was shown by Reeder
\cite[Theorem 3.5.4]{Ree2}, via reduction to the case with simply connected $G_\der$.

When $\mb q$ is a root of unity, Theorem \ref{thm:5.1} can definitely fail, in particular when $\mb q$
is a zero of the Poincar\'e polynomial of $W$. On the other hand, if $\mb q$ is not a zero of the
polynomials \eqref{eq:1.24} for any finite reflection subgroup of $X \rtimes W$, then it seems
likely that large parts of the K-theoretic approach are still valid.

A very special case arises when $\mb q = 1$. Then Theorem \ref{thm:5.1} holds, in a slightly different,
simpler way \cite{Kat2}. In this case the action of 
\[
\mc H (\mc R (G,T),1) = \C [X^* (T)] \rtimes W
\]
on $K_{t,u,\rho}$ preserves the homological degree, so 
$\Hom_{\pi_0 (Z_G (t,u))} \big( \rho, H_d (\mc B^{t,u},\C)\big)$ is a subrepresentation for any
$d \in \Z_{\geq 0}$. Then $K_{t,u,\rho}$ may obviously have many irreducible quotients, so the
previous definition of $L_{t,u,\rho}$ cannot be used anymore. Instead we define
\begin{equation}\label{eq:5.5}
L_{t,u,\rho} = \Hom_{\pi_0 (Z_G (t,u))} \big( \rho, H_{\dim_\R (\mc B^{t,u})} 
(\mc B^{t,u},\C)\big) \qquad \text{when } \mb q = 1, 
\end{equation}
that is, we only use the homology in the largest possible degree. It can be shown that 
\eqref{eq:5.5} is the canonical irreducible quotient of $K_{t,u,\rho}$ \cite[\S 12]{ABPSprin}.
Thus modified, Theorem \ref{thm:5.1} becomes valid for $\mb q = 1$ without any restriction on 
$G$ \cite[Theorem 4.1]{Kat2}. Notice that here the triples
$(t,u,\rho)$ satisfy $t u = u t$, so $u$ is a unipotent element of $Z_G (t)$. This classification
of $\Irr (X \rtimes W)$ can be regarded as a Springer correspondence for affine Weyl groups.
The proof is much shorter than that of Theorem \ref{thm:5.1}, it mainly relies on the Springer
correspondence for finite Weyl groups.\\

As already observed in \cite[\S 2]{KaLu}, there are several equivalent ways to present 
Kazhdan--Lusztig parameters (for arbitrary complex reductive groups $G$), when we consider them 
modulo $G$-conjugation (as in Theorem \ref{thm:5.1}.c). One alternative shows the connection 
between different $\mb q$'s very nicely.

Let $(t_1,u,\rho_1)$ be a Kazhdan--Lusztig triple for $(G,1)$.
Pick an algebraic homomorphism $\phi : SL_2 (\C) \to Z_G (t_1)$ with $\phi ( \matje{1}{1}{0}{1} ) = u$. 
By the Jacobson--Morozov theorem, such a $\phi$ exists and is unique up to conjugation by $Z_G (t_1,u)$. 
Assume that we have a preferred square root of $\mb q$. We put $t_{\mb q} = t_1 \phi ( 
\matje{\mb q^{1/2}}{0}{0}{\mb q^{-1/2}} )$, so that $t_{\mb q} u t_{\mb q}^{-1} = u^{\mb q}$. Then
$\mc B^{t_1,u}$ and $\mc B^{t_{\mb q},u}$ are homotopy equivalent and $\rho_1$ gives rise to a unique
$\rho_{\mb q}$ \cite[Lemma 6.1]{ABPSprin}. This provides a bijection between $G$-conjugacy classes
of Kazhdan--Lusztig tripes for $(G,1)$ and for $(G,{\mb q})$ \cite[Lemma 7.1]{ABPSprin}. 
In combination with Theorem \ref{thm:5.1} we obtain:

\begin{cor}\label{cor:5.3}
Let $G$ be a complex reductive group and let $\mb q \in \C^\times$ be either 1 or not a root of unity.
There exists a canonical bijection
\[
\begin{array}{ccc}
\{ \text{Kazhdan--Lusztig triples for } (G,1) \} / G & \longleftrightarrow &
\Irr \big( \mc H (\mc R (G,T),\mb q) \big) \\
(t_1,u,\rho_1) & \mapsto & L_{t_{\mb q},u,\rho_{\mb q}}
\end{array}
\]
\end{cor}

One advantage of this parametrization is that the pair $(t_1,u)$ is the Jordan decomposition
of an arbitrary element of $G$, so that $(t_1,u)$ up to $G$-conjugacy parametrizes the
conjugacy classes of $G$.

\begin{ex}
Let $\mc R$ be of type $GL_n$ and consider $\mc H_n (\mb q)$. As $Z_{GL_n (\C)}(t_1,u) =
Z_{GL_n (\C)}(t_1 u)$ is always connected, $\rho$ is necessarily trivial and may be ignored.
Corollary \ref{cor:5.3} recovers the parametrization of $\Irr (\mc H_n (\mb q))$ summarised
in Theorem \ref{thm:2.6}.
\end{ex}

In addition to the already discussed properties of these bijections, we mention that 
temperedness can be detected easily in Corollary \ref{cor:5.3}. Namely, by
\cite[Proposition 9.3]{ABPSprin}, for $\mb q \in \R_{\geq 1}$:
\begin{equation}\label{eq:5.6}
L_{t_{\mb q},u,\rho_{\mb q}} \text{ is tempered} \quad \Longleftrightarrow \quad
t_1 \text{ lies in a compact subgroup of } G.
\end{equation}

Finally, we mention an interesting extension of the method from \cite{ChGi} to an affine 
Hecke algebra $\mc H$ of type $B_n / C_n$ with three unequal parameters \cite{Kat3}.
Like in \cite{KaLu} it is shown that $\mc H$ is isomorphic to the equivariant K-theory of
a certain algebraic variety, and that is used to parametrize $\Irr (\mc H)$ when the 
parameters are generic (in a sense related to Theorem \ref{thm:4.4}). We return to this
in Paragraph \ref{par:BnCn}.

\subsection{Equivariant homology} \
\label{par:homology}

The material in the previous paragraph learns us a lot about equal label affine Hecke
algebras, but very little about the cases with several parameters $q_s$.
Faced with this problem, Lusztig discovered that Hecke algebras with multiple parameters
can still be studied geometrically, if we accept two substantial modifications:
\begin{itemize}
\item replace affine Hecke algebras by their graded versions,
\item replace equivariant K-theory by equivariant homology. 
\end{itemize}
Not all combinations of $q$-parameters can be obtained in this way, but still a considerable
number of them. Our treatment of this method is based on the papers 
\cite{Lus-Int,LusCusp1,LusCusp2,LusCusp3,AMS2}.

Let $G$ be a connected complex reductive group and let $P$ be a parabolic subgroup of $G$
with Levi factor $L$ and unipotent radical $U$. We denote the Lie algebras of these groups
by $\mf g, \mf p, \mf l$ and $\mf u$. Let $v \in \mf l$ be nilpotent and let $\mc C_v^L$
be its adjoint orbit. Let $\mc L$ be an irreducible $L$-equivariant cuspidal local system
on $\mc C_v^L$. We refrain from explaining these notions here, instead we refer to
\cite{Lus-Int}, where cuspidal local systems are introduced and classified.

To the above data we will associate a graded Hecke algebra. We take $T = Z(L)^\circ$,
a (not necessarily maximal) torus in $G$. According to \cite[Proposition 2.5]{LusCusp1},
the cuspidality of $\mc L$ implies that:
\begin{itemize}
\item the set of weights of $T$ acting on $\mf g$ is a (possibly nonreduced) root system
$R(G,T)$ in $X^* (T)$,
\item the Weyl group of $R(G,T)$ is $W_L := N_G (L) / L = N_G (T) / Z_G (T)$.
\end{itemize}
The parabolic subgroup $P$ determines a basis $\Delta_L$ of $R(G,T)$.
Let $\mf t = X_* (T) \otimes_\Z \C$ be the Lie algebra of $T$, so that 
\[
R(G,T) \subset \mf t^* = X^* (T) \otimes_\Z \C.
\]
The action of $W_L$ on $T$ induces actions on $\mf t$ and $\mf t^*$, which stabilize $R(G,T)$.
The definition of the parameter function $k : R(G,T) \to \Z$ involves the
nilpotent element $v$. Let $\mf g_\alpha \subset \mf g$ be the root space and let $s_\alpha \in 
W_L$ be the reflection associated to $\alpha \in R(G,T)$. Since $v \in \mf l$ commutes with
$\mf t$, ad$(v)$ stabilizes each $\mf g_\alpha$. For $\alpha \in \Delta_L$ one defines
$k(\alpha) \in \Z_{\geq 2}$ by
\[
\begin{array}{ccccl}
\mr{ad}(v)^{k(\alpha) - 2} : & \mf g_\alpha \oplus \mf g_{2 \alpha} & \to &
\mf g_\alpha \oplus \mf g_{2 \alpha} & \text{is nonzero,} \\
\mr{ad}(v)^{k(\alpha) - 1} : & \mf g_\alpha \oplus \mf g_{2 \alpha} & \to &
\mf g_\alpha \oplus \mf g_{2 \alpha} & \text{is zero.}
\end{array}
\]
Then $k(\alpha) = k(\beta)$ whenever $\alpha, \beta \in \Delta_L$ are $W_L$-associate. 
Now we can define
\[
\mh H (G,L,\mc L,\mb r) = \mh H (\mf t, W_L, k, \mb r) .
\]
Suppose that $G$ is an almost direct product of connected normal subgroups $G_1$ and $G_2$.
Then $L, \mc C_v^L$ and $\mc L$ decompose accordingly and
\[
\mh H (G,L,\mc L,\mb r) = 
\mh H (G_1,L_1,\mc L_1,\mb r) \otimes_{\C [\mb r]} \mh H (G_2,L_2,\mc L_2,\mb r).
\]
If $G$ is a torus, then necessarily $L = T = G$ and $v = 0$. In that case $\mc L$ is trivial
and $\mh H (T,T,\mc L,\mb r)$ is just $\mc O (\mf t) \otimes_\C \C[\mb r]$. Hence the study 
of $\mh H (G,L,\mc L)$ can be reduced to simple $G$. Then it becomes feasible 
to classify the data, and indeed this has been done in \cite[2.13]{LusCusp1}. 
We tabulate the possibilities for $\mf g, \mf l, R(G,T)$ and $k$:
\begin{equation}\label{table}
\begin{array}{cccc}
\mf g & \mf l & R(G,T) & k \\
\hline 
\text{simple} & \text{Cartan} & \text{irreducible} & k(\alpha) = 2 \\
\mf{sl}_{(d+1)p} & {\mf sl}_p^{d+1} \oplus \C^d & A_d & k(\alpha) = 2p \\
\mf{sp}_{2n + 2d}, 2n = p(p+1) & \mf{sp}_{2n} \oplus \C^d & BC_d & k(\alpha) = 2,
k(\beta) = 2 p + 1\\
\mf{so}_{n + 2d}, n = p^2 & \mf{so}_n \oplus \C^d & B_d & k(\alpha) = 2,
k(\beta) = 2 p \\
\mf{so}_{n + 4d}, n = p (2p-1) & \mf{sl}_2^d \oplus \C^d & BC_d & k(\alpha) = 4,
k(\beta) = 4 p - 1\\
\mf{so}_{n + 4d}, n = p (2p+1) & \mf{sl}_2^d \oplus \C^d & BC_d & k(\alpha) = 4,
k(\beta) = 4 p + 1\\
E_6 & \mf{sl}_3^2 \oplus \C^2 & G_2 & k(\alpha) = 2, k(\beta) = 6\\
E_7 & \mf{sl}_2^3 \oplus \C^4 & F_4 & k(\alpha) = 2, k(\beta) = 4
\end{array}
\end{equation}
In this table $d,p \in \Z_{>0}$ are arbitrary, $\alpha \in R(G,T)$ is a long root and
$\beta \in R(G,T)$ is a short root. (For $R(G,T)$ of type $BC_d$ we mean that 
$\alpha = \pm e_i \pm e_j$ and $\beta = \pm e_i$.)

Recall from \eqref{eq:1.16} that we can simultaneously rescale all the $k(\alpha)$ without
changing the algebra (up to isomorphism). When $R(G,T)$ has roots of different lengths, 
we can also adjust the parameters for roots of one length in a specific way:

\begin{ex} \label{ex:1.1}
Let $R$ be of type $B_n, F_4$ or $G_2$. Assume that $R$ spans $\mf a^*$. 
Any parameter function $k$ for $R$ has two independent values $k_1 = k(\alpha)$ and
$k_2 = k (\beta)$, which can be chosen arbitrarily. We write $k = (k_1,k_2)$ and we
consider the graded Hecke algebra $\mh H (\mf t,W(R),k_1,k_2,\mb r)$.

Take $\epsilon = 2$ for $B_n$ or $F_4$, and $\epsilon = 3$ for $G_2$. The set 
\[
\{ w \alpha : w \in W \} \cup \{ \epsilon w \beta : w \in W \}
\]
is a root system in $\mf a^*$, of type (respectively) $C_n, F_4$ or $G_2$. Notice that now 
$\alpha$ is short and $\epsilon \beta$ is long. The identity map on the vector space 
underlying $\mh H (\mf t,W(R),k_1,k_2, \mb r)$ provides an algebra isomorphism
\begin{align} 
\label{eq:1.23} & 
\mh H (\C^n, W(B_n), k_1,k_2,\mb r) \to \mh H (\C^n, W(C_n), 2 k_2, k_1, \mb r) \\
\label{eq:1.26} & 
\mh H (\C^4, W(F_4), k_1,k_2,\mb r) \to \mh H (\C^4, W(F_4), 2 k_2, k_1, \mb r) \\
\label{eq:1.27} &
\mh H (\C^2, W(G_2), k_1,k_2,\mb r) \to \mh H (\C^2, W(G_2), 3 k_2, k_1, \mb r)
\end{align}
In particular any graded Hecke algebra of type $C_n$ is also 
a graded Hecke algebra of type $B_n$ (but with different parameters). 
\end{ex}
We will call a parameter function obtained from the above table by a composition of the isomorphisms 
\eqref{eq:1.16} and those from Example \ref{ex:1.1} geometric. Thus we have a large supply 
of geometric parameter functions for type B/C root systems.

Next we describe the geometric realization of $\mh H (G,L,\mc L,\mb r)$. We need the varieties
\begin{align*}
& \dot{\mf g} = \{ (x,gP) \in \mf g \times G / P : 
\mr{Ad}(g^{-1})x \in \mc C_v^L + \mf t + \mf u \}, \\
& \ddot{\mf g_N} = \{ (x,gP,g'P) \in \mf g \times (G/P)^2 : (x,gP) \in \dot{\mf g}, (x,g'P) 
\in \dot{\mf g}, x \text{ nilpotent} \} . 
\end{align*}
The first is a variation on the variety $\mc B$ of Borel subgroups of $G$, while the
second has a flavour of the Steinberg variety of $G$. The group $G \times \C^\times$ acts
on $\dot{\mf g}$ by
\[
(g_1,\lambda) (x,gP) = (\lambda^{-2} \mr{Ad}(g_1) x, g_1 g P) ,
\] 
and similarly on $\ddot{\mf g_N}$. The $L \times \C^\times$-equivariant local system $\mc L$
on $\mc C_v^L$ yields a $G \times \C^\times$-equivariant local system $\dot{\mc L}$
on $\dot{\mf g}$. The two projections $\ddot{\mf g_N} \to \dot{\mf g}$ give rise to an
equivariant local system $\ddot{\mc L} = \dot{\mc L} \boxtimes \dot{\mc L}^*$ on 
$\ddot{\mf g_N}$. Equivariant (co)homology with coefficients in a local system is defined in
\cite[\S 1]{LusCusp1}. In special cases this theory admits a convolution product
\cite[\S 2]{LusCusp2}. That makes $H_*^{G \times \C^\times} (\ddot{\mf g_N},\ddot{\mc L})$
into a graded algebra, see the proof of \cite[Theorem 8.11]{LusCusp2}. 

\begin{thm}\label{thm:5.3}
\textup{\cite[Corollary 6.4]{LusCusp1} and \cite[Theorem 8.11]{LusCusp2}} \\
There exists a canonical isomorphism of graded algebras
\[
\mh H (G,L,\mc L,\mb r) \longrightarrow H_*^{G \times \C^\times} (\ddot{\mf g_N},\ddot{\mc L}).
\]
\end{thm}

With equivariant homology one can construct many modules for $H_*^{G \times \C^\times} 
(\ddot{\mf g_N},\ddot{\mc L})$. Let $y \in \mf g$ be nilpotent and define
\[
\mc P_y = \{ g P \in G / P : \mr{Ad}(g^{-1}) y \in \mc C_v^L + \mf u \} .
\]
This is the appropriate analogue of the variety $\mc B^u$ of Borel subgroups containing $u$.
The group
\[
M(y) := \{ (g_1,\lambda) \in G \times \C^\times : \mr{Ad}(g_1) y = \lambda^2 y \}
\] 
acts on $\mc P_y$ by $(g_1,\lambda) gP = g_1 g P$. The inclusion $\{y\} \times \mc P_y
\to \dot{\mf g}$ is $M(y)$-equivariant, which allows us to restrict $\dot{\mc L}$ to an
equivariant local system on $\mc P_y$. With constructions in equivariant (co)homology
\cite[\S 3.1]{AMS2} one can define an action of $\mh H (G,L,\mc L,\mb r)$ on
$H_*^{M(y)^\circ}(\mc P_y,\dot{\mc L})$. It commutes with the natural actions of
$\pi_0 (M(y))$ and of $H^*_{M(y)^\circ}(\{y\})$ on $H_*^{M(y)^\circ}(\mc P_y,\dot{\mc L})$
which enables us to decompose it as\\ $\mh H (G,L,\mc L,\mb r)$-module. 
It is known that $H_*^{M(y)^\circ}(\mc P_y,\dot{\mc L})$ is projective over 
$H^*_{M(y)^\circ}(\{y\})$. One can naturally identify
\begin{align*}
& H^*_{M(y)^\circ}(\{y\}) = \mc O (\mr{Lie}(M(y)^\circ))^{M(y)^\circ},\\
& \mr{Lie}(M(y)^\circ) = \{ (\sigma,r) \in \mf g \oplus \C : [\sigma,y] = 2 r y \} .
\end{align*}
In particular the characters of $H^*_{M(y)^\circ}(\{y\})$ are parametrized by semisimple
adjoint orbits in Lie$(M(y)^\circ)$. For a semisimple element $(\sigma,r) \in \mr{Lie}
(M(y)^\circ)$ we have the $\mh H (G,L,\mc L,\mb r)$-module
\begin{equation}\label{eq:5.12}
E_{y,\sigma,r} = 
\C_{\sigma,r} \otimes_{H^*_{M(y)^\circ}(\{y\})} H^{M(y)^\circ}_* (\mc P_y,\dot{\mc L}) .
\end{equation}
By the projectivity of $H^{M(y)^\circ}_* (\mc P_y,\dot{\mc L})$, $(\sigma,r) \mapsto
E_{y,\sigma,r}$ is an algebraic family of $\mh H (G,L,\mc L,\mb r)$-modules. In particular
the restriction of $E_{y,\sigma,r}$ to  the finite dimensional semisimple algebra $\C [W_L]$ 
does not depend on $(\sigma,r)$.
As usual, the isomorphism class of $E_{y,\sigma,r}$ depends only on $(y,\sigma,r)$ up to
the adjoint action of $G$ (which fixes $r$). From the action of $\pi_0 (M(y))$ on 
$H^{M(y)^\circ}_* (\mc P_y,\dot{\mc L})$, only the operators that stabilize the adjoint
orbit $[\sigma]$ of $\sigma$ in\\ 
Lie$(M(y)^\circ)$ act on $E_{y,\sigma,r}$. Hence irreducible representations $\rho$ of 
$\pi_0 (M)_{[\sigma]}$ can be used to decompose the $\mh H (G,L,\mc L, \mb r)
$-module $E_{y,\sigma,r}$ further. We define the $\mh H (G,L,\mc L,\mb r)$-representation
\[
E_{y,\sigma,r,\rho} = \Hom_{\pi_0 (M(y))_{[\sigma]}}(\rho, E_{y,\sigma,r}).
\]
We call this a standard module if it is nonzero.
 
Let us improve the bookkeeping for the parameters $(y,\sigma,r,\rho)$ just obtained. With
Jacobson--Morozov we pick an algebraic homomorphism $\gamma_y : SL_2 (\C) \to G$ such that
d$\gamma_y (\matje{1}{1}{0}{1}) = y$. Notice that the semisimple element
\[
\sigma_0 = \sigma + \textup{d}\gamma_y ( \matje{-r}{0}{0}{r} )
\]
commutes with $y$. It is not difficult to see that $\pi_0 (M(y))_{[\sigma]}$ is naturally
isomorphic to $\pi_0 (Z_G (\sigma_0,y))$. By \cite[Lemma 3.6]{AMS2} these constructions
provide a bijection between $G$-association classes of data $(y,\sigma,\rho)$ as above
(for a fixed $r \in \C$) and $G$-association classes of triples $(y,\sigma_0,\rho)$, where
\begin{equation}\label{eq:5.7}
y \in \mf g \text{ nilpotent, } \sigma_0 \in \mf g \text{ semisimple, } [\sigma_0,y] = 0,
\rho \in \Irr \big( \pi_0 (Z_G (\sigma_0,y)) \big) .
\end{equation}
In the previous paragraph we encountered a clear condition on the representation $\rho$
of the component group: it should appear in the homology of a particular variety, otherwise
the associated module would be 0. 

In the current setting the condition on $\rho$ is more
subtle, because $\mc P_y$ can be empty and a local system $\mc L$ is involved. To formulate
it we need the cuspidal support map $\Psi_G$ from \cite[6.4]{Lus-Int}. It associates a
cuspidal support $(L',\mc C_{v'}^{L'},\mc L')$ to every pair $(x,\rho')$ with $x$ nilpotent and
$\rho' \in \Irr \big( \pi_0 (Z_G (x)) \big)$. Giving such $(x,\rho)$ is equivalent to giving
a $G$-equivariant cuspidal local system on a nilpotent orbit in $\mf g$ (which is also
an equivariant perverse sheaf). The cuspidal support map can be expressed with a version of
parabolic induction for equivariant perserve sheaves \cite[\S 4.1]{Aub}.
According to \cite[Proposition 3.7]{AMS2}:
\begin{equation}\label{eq:5.8}
E_{y,\sigma,r,\rho} \neq 0 \quad \Longleftrightarrow \quad
\Psi_{Z_G (\sigma_0)}(y,\rho) \text{ is } G\text{-associate to } (L,\mc C_v^L,\mc L).
\end{equation}
When $L = T$ is a maximal torus of $G$ and $v = 0$, we have $\mc P_y = \mc B^{\exp y}$ and 
$\mc L$ is trivial. Then the condition \eqref{eq:5.8} reduces to:
$\rho$ appears in $H_* (\mc B^{\exp y}_{Z_G (\sigma_0)},\C)$. That is equivalent to the condition
on $\rho$ in the Kazhdan--Lusztig triple $(\exp (\sigma_0), \exp (y), \rho)$ for $(G,1)$.

\begin{thm}\label{thm:5.4}
\textup{\cite[Theorem 1.15]{LusCusp3} and \cite[Theorem 3.11]{AMS2}} \\
Let $(y,\sigma_0,\rho)$ be as in \eqref{eq:5.7}, such that $\Psi_{Z_G (\sigma_0)}(y,\rho)$ is 
$G$-associate to $(L,\mc C_v^L,\mc L)$. For $r \in \C$ we write $\sigma_r = \sigma_0 +
\textup{d}\gamma_y ( \matje{r}{0}{0}{-r} )$, where $\gamma_y : SL_2 (\C) \to Z_G (\sigma_0)$
with d$\gamma_y ( \matje{0}{1}{0}{0} ) = y$.
\enuma{
\item For $r \neq 0$, $E_{y,\sigma_r,r,\rho}$ has a unique irreducible quotient, which we call
$M_{y,\sigma_r,r,\rho}$.
\item For $r = 0$, $E_{y,\sigma_0,0,\rho}$ has a canonical irreducible quotient $M_{y,\sigma_0,0,\rho}$
(the direct summand in one specific homological degree).
\item Parts (a) and (b) set up a bijection between
$\Irr \big( \mh H (G,L,\mc L,\mb r) / (\mb r - r) \big)$ and the $G$-association classes of triples
$(y,\sigma_0,\rho)$ as above. 
\item Every irreducible constituent of $E_{y,\sigma_r,r,\rho}$ different from $M_{y,\sigma_r,r,\rho}$
is isomorphic to $M_{y',\sigma',r,\rho'}$, for data $(y',\sigma',\rho')$ as above that satisfy
$\dim \mc{C}_y^G < \dim \mc{C}_{y'}^G$. 
}
\end{thm} 

Lusztig investigated when the modules $E_{y,\sigma,r,\rho}$ are tempered or discrete series
\cite{LusCusp3}. Unfortunately his notions differ from ours, and as a consequence the resulting
properties are opposite to what we want. To reconcile it, we use the Iwahori--Matsumoto involution
of $\mh H (G,L,\mc L,\mb r)$. It is the algebra automorphism
\begin{align*}
& \mr{IM} : \mh H (G,L,\mc L,\mb r) \to \mh H (G,L,\mc L,\mb r) \\
& \mr{IM}(N_w) = \mr{sign}(w) N_w, \quad \mr{IM}(\mb r) = \mb r, \quad
\mr{IM}(\xi) = - \xi \qquad w \in W_L, \xi \in \mf t^*.
\end{align*}
Clearly composition with IM has the effect $x \mapsto -x$ on $\mc O (\mf t)$-weights of
$\mh H (G,L,\mc L,\mb r)$-representations, and similarly for central characters. Let 
$y,\sigma_0,\rho$ and $\gamma_y$ be as in Theorem \ref{thm:5.4} and \eqref{eq:5.7}. We define
\begin{equation}\label{eq:5.9}
\begin{array}{ccc}
\widetilde{E}_{y,\sigma_0,r,\rho} & = &
\mr{IM}^* E_{y,\textup{d}\gamma_y \matje{r}{0}{0}{-r} - \sigma_0, r,\rho} \\
\widetilde{M}_{y,\sigma_0,r,\rho} & = &
\mr{IM}^* M_{y,\textup{d}\gamma_y \matje{r}{0}{0}{-r} - \sigma_0, r,\rho} 
\end{array}
\end{equation}
The modules \eqref{eq:5.9} enjoy the same properties as their ancestors without tildes in 
Theorem \ref{thm:5.4}. By \cite[Theorem 8.13]{LusCusp2} and \cite[Theorem 3.29]{AMS2} 
all these four modules admit the same central character, namely
\begin{equation}\label{eq:5.10}
(\mr{Ad}(G)(\sigma_r) \cap \mf t,r) = 
(\mr{Ad}(G)(\sigma_0 - \textup{d}\gamma_y \matje{r}{0}{0}{-r}) \cap \mf t, r).
\end{equation}

\begin{thm}\label{thm:5.5}
Let $(y,\sigma_0,\rho)$ be as in Theorem \ref{thm:5.4}.
\enuma{
\item When $\Re (r) \geq 0$, the following are equivalent:
\begin{itemize}
\item $\widetilde{E}_{y,\sigma_0,r,\rho}$ is tempered,
\item $\widetilde{M}_{y,\sigma_0,r,\rho}$ is tempered,
\item $\mr{Ad}(G)\sigma_0$ intersects $i \mf a = i\R \otimes_\Z X_* (T)$.
\end{itemize}
\item When $\Re (r) > 0$, the following are equivalent:
\begin{itemize}
\item $\widetilde{M}_{y,\sigma_0,r,\rho}$ is essentially discrete series,
\item $y$ is distinguished nilpotent in $\mf g$, that is, not contained in any proper
Levi subalgebra of $\mf g$.
\end{itemize}
Moreover in this case $\sigma_0 \in Z(\mf g)$ and $\widetilde{E}_{y,\sigma_0,r,\rho}
= \widetilde{M}_{y,\sigma_0,r,\rho}$.
\item When $r \in \R$, the central character of $\widetilde{E}_{y,\sigma_0,r,\rho}$ lies in
$\mf a / W$ if and only if $\sigma_0 \in \mr{Ad}(G) \mf a$.
}
\end{thm}
\begin{proof}
(a) and (b) See \cite[(84) and (85)]{AMS2}.\\
(c) Upon conjugating the parameters by a suitable element of $G$, we may assume that
$\sigma_0, \textup{d}\gamma_y \matje{1}{0}{0}{-1} \in \mf t$ \cite[Proposition 3.5.c]{AMS2}.
Then $\textup{d}\gamma_y \matje{r}{0}{0}{-r}$ represents the central character of the module
$\widetilde{E}_{y,0,r} = \mr{IM}^* E_{y,\textup{d}\gamma_y \matje{r}{0}{0}{-r},r}$ for
\[
\mh H \big( Z_G (\sigma_0),L,\mc L,\mb r \big) / (\mb r - r) \cong 
\mh H \big( Z_G (\sigma_0)_\der, L \cap Z_G (\sigma_0)_\der, \mc L,\mb r \big) / (\mb r - r) 
\otimes_\C \mc O (Z_{\mf g}(\sigma_0)) .
\] 
Part (b) tells us that the restriction of $\widetilde{E}_{y,0,r}$ to 
$\mh H (Z_G (\sigma_0)_\der, L \cap Z_G (\sigma_0)_\der, \mc L,\mb r)$ is a direct sum of 
discrete series representations. By \cite[Lemma 2.13]{Slo2} the central characters of these 
representations lie in $\Z_{\mf g}(\sigma_0)_\der \cap \mf a / W_{L \cap Z_G (\sigma)_\der}$. 
Hence d$\gamma_y \matje{r}{0}{0}{-r} \in \mf a$, which implies that 
\[
\sigma_0 \in \mf a \quad \Longleftrightarrow \quad
\sigma_0 - \textup{d}\gamma_y \matje{r}{0}{0}{-r} \in \mf a.
\]
Compare that with \eqref{eq:5.10}
\end{proof}

The family of representations $E_{y,\sigma,r}$ is compatible with parabolic induction,
under a mild condition that $(\sigma,r)$ is not a zero of a certain polynomial function 
$\epsilon$ \cite[Corollary 1.18]{LusCusp3}. Namely, let $Q$ be a standard Levi subgroup
of $G$ containing $L$ and suppose that $\{y,\sigma_0\} \subset \mr{Lie}(Q)$. When
$\epsilon (\sigma,r) \neq 0$ (or $r=0$, see \cite[Theorem A.2]{AMS2}), there is a
canonical isomorphism
\begin{equation}\label{eq:5.11}
\ind_{\mh H (Q,L,\mc L,\mb r)}^{\mh H (G,L,\mc L,\mb r)} E^Q_{y,\sigma,r} \to
E_{y,\sigma,r} .
\end{equation}
The Iwahori--Matsumoto involution commutes with parabolic induction, so \eqref{eq:5.11}
also holds for the family of representations $\widetilde{E}_{y,\sigma_0,r,\rho}$. With that 
and Theorem \ref{thm:5.4} one can show that (at least when $\Re (r) > 0$) 
$\widetilde{E}_{y,\sigma_0,r,\rho}$ is a
standard module in the sense of Paragraph \ref{par:analogues}, see \cite[Proposition A.3]{AMS2}.
(This refers to the last arXiv-version of \cite{AMS2}, in which an appendix was added to
deal with a mistake in the published version.) The argument for standardness is analogous
to what we sketched for $K_{t,u,\rho}$ around \eqref{eq:5.4}.

\begin{ex} \label{ex:5.1}
We illustrate the material in this section with an example of rank 1.
Let $G = \mr{Sp}_4 (\C), L = \mr{Sp_2}(\C) \times GL_1 (\C)$ and $v$ regular nilpotent in
$\mf l$. The local system on $\mc C_v^L$ corresponding to the sign representation of
$\pi_0 (Z_L (v)) = \pi_0 (Z(L)) = \{\pm 1\}$ is cuspidal. The root system with respect to
$T = Z(L)^\circ \cong \C^\times$ is $R(G,T) = \{ \pm \alpha, \pm 2 \alpha \}$, of type $BC_1$.
Its Weyl group is $W_L = \{ 1, s_\alpha \}$, with $s_\alpha$ acting on $T$ by inversion.
The parameter function by determined by $\mc L$ satisfies $k(\pm \alpha) = 3$ and
$k(\pm 2 \alpha) = 6$. The associated graded Hecke algebra is 
\[
\mh H (G,L,\mc L,\mb r) = \mh H (\mf t,W_L,k,\mb r) .
\]
The irreducible representations on which $\mb r$ acts by a fixed $r \in \C^\times$ were classified
in Example \ref{ex:3.1} and for $r=0$ in Theorem \ref{thm:2.3}. 

Let us analyse the possibilities for the geometric parameters. It turns out that the condition
on $\rho$ can only be met for nilpotent elements $y$ in two adjoint orbits in $\mf g$: the orbits
of $v$ and of $v + v'$, where $v'$ is regular nilpotent in $Z(\mf l_\der) \cong \mf{sp}_2 (\C)$.
The cuspidal support condition becomes that the subgroup of $\pi_0 (Z_G (\sigma_0,y))$ coming from
$\pi_0 (Z(L))$ must act via the sign representation.

We may assume that $\C v'$ is stable under the adjoint action of $T$. A complete list of 
representatives for the $G$-conjugacy classes of parameters for $\mh H (G,L,\mc L,\mb r)$ is:
\[
\begin{array}{c|c|c|c}
\sigma_r &  \mr{diag}(r,r,-r,-r) & \sigma_0 = 0, r = 0 &  
\mr{diag}(\sigma,r,-r,-\sigma) , \Re(\sigma) \geq 0\\
y & v + v' & v & v \\
\pi_0 (Z_G (\sigma_0,y)) & Z(\mr{Sp}_2 (\C)^2) & Z(\mr{Sp}_2 (\C)) & Z(\mr{Sp}_2 (\C)) \\
\rho & \mr{sign} \boxtimes \triv & \mr{sign} & \mr{sign} \\
E_{y,\sigma_r,r,\rho} & \triv & \mr{St} & \ind_{\mc O (\mf t) \otimes_\C \C[\mb r]}^{
\mh H (G,L,\mc L,\mb r)} (\C_{\mr{diag}(-\sigma,r,-r,\sigma),r}) \\
\widetilde{E}_{y,\sigma_r,r,\rho} & \mr{St} & \triv & \ind_{\mc O (\mf t) \otimes_\C \C[\mb r]}^{
\mh H (G,L,\mc L,\mb r)} (\C_{\mr{diag}(\sigma,r,-r,-\sigma),r})
\end{array}
\]
In the last column we exclude the case $\sigma_0 = 0, r=0$. For almost all $\sigma \in \C$, standard 
module $\ind_{\mc O (\mf t) \otimes_\C \C[\mb r]}^{ \mh H (G,L,\mc L,\mb r)} (\C_{\mr{diag}(\sigma,r,-r,
-\sigma),r})$ is irreducible, and hence equal to both $M_{y,\sigma_r,r,\rho}$ and 
$\widetilde{M}_{y,\sigma_r,r,\rho}$. The exceptions are $\sigma = \pm r$, then 
\[
M_{y,\mr{diag}(r,r,-r,-r),r,\rho} = \mr{St} \quad \text{and} \quad
\widetilde{M}_{y,\mr{diag}(r,r,-r,-r),r,\rho} = \triv.
\]
\end{ex}

\subsection{Affine Hecke algebras from cuspidal local systems} \
\label{par:cusp}

In the previous paragraph we associated a graded Hecke algebra to a cuspidal local system
on a nilpotent orbit in a Levi algebra of $\mf g$. With a process that is more or less
inverse to the reduction theorems in Paragraph \ref{par:reduction}, we can glue a suitable
family of such algebras into one affine Hecke algebra.

Concretely, let $G,P,L,v,T,\mc L$ be as before. Let $R(G,T)_{\mr{red}}^\vee \subset X_* (T)$
be the dual of the reduced root system $R(G,T)_{\mr{red}} \subset X^* (T)$, and consider
the root datum 
\[
\mc R (G,T) = (R(G,T)_{\mr{red}},X^* (T),R(G,T)_{\mr{red}}^\vee,X_* (T),\Delta_L) .
\]
There are unique parameter functions $\lambda,\lambda^* : R(G,T)_{\mr{red}} \to \Z_{\geq 0}$
that meet the requirements sketched above, see \cite[Proposition 2.1 and (28)]{AMS3}. 
With those, we define
\begin{equation}\label{eq:5.13}
\mc H (G,L,\mc L) = \mc H (\mc R(G,T),\lambda,\lambda^*,\mb q) .
\end{equation}
This is slightly different from \cite[\S 2]{AMS3}, where the algebras involved a formal
variable $\mb z$ instead of $\mb q^{1/2}$. To compensate for the square root ($\mb q^{1/2}$
versus $\mb q$) we could replace $\lambda,\lambda^*$ by $2 \lambda,2 \lambda^*$--but that is
not necessary, since we are at liberty to choose $\mb q \in \R_{>0}$ as we like.

When $\lambda,\lambda^* : R \to \R$ arise from a cuspidal local system as in \eqref{eq:5.13},
we call them geometric parameter functions for $\mc R$. As we may choose any $\mb q \in \R_{>0}$,
$\lambda,\lambda^*$ may be scaled by any nonzero real factor, and remain geometric.

For a unitary $t \in T_\un$, we saw in Corollary \ref{cor:3.27} that there is an equivalence
of categories
\[
\Mod_{f,W_L t \exp (\mf a)} \big( \mc H (G,L,\mc L) \big) \quad \cong \quad
\Mod_{f,\mf a} \big( \mh H (\mf t,W(R_t), k_t) \rtimes \Gamma_t \big) .
\]
Let $\widetilde{Z_G}(t)$ be the subgroup of $G$ generated by $Z_G (t)$ and the root subgroups
$U_\alpha$ with $s_\alpha (t) = t$. By \cite[Theorems 2.5 and 2.9]{AMS3} we may identify 
\[
\mh H (\mf t,W(R_t), k_t) \rtimes \Gamma_t = 
\mh H (\widetilde{Z_G}(t), L ,\mc L,\mb r) / (\mb r - \log (\mb q) / 2) .
\]
Here $\widetilde{Z_G}(t)$ has component group $\Gamma_t$, so it can be disconnected. In that 
case these graded Hecke algebras are a little more general than in the previous paragraph,
but that does not matter much.

\begin{ex} \label{ex:5.2}
We continue Example \ref{ex:5.1}. We identify $T$ with $\C^\times$
by means of the map $\mr{diag}(z,1,1,z^{-1}) \mapsto z$. For $t \in T_\un \setminus \{1,-1\}$,
$Z_G (t) = \widetilde{Z_G} (t) = L$ and 
\[
\mh H (\widetilde{Z_G}(t),L,\mc L,\mb r) = \mh H (T,T,\triv,\mb r) = 
\mc O (\mf t) \otimes_\C \C[\mb r] .
\]
The most interesting element of $T_\un$ is $-1 = \mr{diag}(-1,1,1,-1)$. Then $Z_G (-1) = L$
but $\widetilde{Z_G}(-1) = \mr{Sp}_2 (\C) \times \mr{Sp}_2 (\C)$. Now the root system is
$R(\widetilde{Z_G}(-1),T) = \{ \pm 2 \alpha \}$, with parameter $k_{-1} (\pm 2 \alpha) = 2$.
The associated graded Hecke algebra is
\[
\mh H (\widetilde{Z_G}(-1),L,\mc L,\mb r) = \mh H (\mf t, W_L, k_{-1}, \mb r).
\]
These data suffice to determine the parameters for $\mc H (G,L,\mc L)$, they are
\[
\lambda (\alpha) = k (\alpha) + k_{-1}(2 \alpha) / 2 = 4 \quad \text{and} \quad
\lambda^* (\alpha) = k (\alpha) - k_{-1}(2 \alpha) / 2 = 2. 
\]
Thus $\mc H (G,L,\mc L) = \mc H (\mc R (G,T),\lambda,\lambda^*,\mb q)$, an affine Hecke
algebra of type $\widetilde A_1$ with unequal parameters.
\end{ex}

With Theorems \ref{thm:3.25} and \ref{thm:3.26} we can reduce the classification of 
$\Irr (\mc H (G,L,\mc L))$ to Theorem \ref{thm:5.4}. For better additional benefits, we prefer
to use the modules $\widetilde{E}_{y,\sigma_0,r,\rho}$ and $\widetilde{M}_{y,\sigma_0,r,\rho}$. In this 
context it is more natural to use data from $G$ that from $\mf g$. The correct parameters turn
out to be a variation on Kazhdan--Lusztig triples, namely triples $(s,u,\rho)$ where
\begin{itemize}
\item $s \in G$ is semisimple,
\item $u \in Z_G (s)$ is unipotent,
\item $\rho \in \Irr \big( \pi_0 (Z_G (s,u)) \big)$ such that the cuspidal support
$\Psi_{Z_G (s)}(u,\rho)$ is\\ $(L,\exp (\mc C_v^L),\exp_* (\mc L))$ modulo $G$-conjugacy.
\end{itemize}
When $Z_G (s)$ is disconnected, we have to use a generalization of the cuspidal support map,
defined in \cite[\S 4]{AMS1}. By conjugating with a suitable element of $G$, we may always
assume that $s \in T$. Let $\overline{E}_{s,u,\rho}$ (resp. $\overline{M}_{s,u,\rho}$)
be the $\mc H (G,L,\mc L)$-module obtained from 
\[
\widetilde{E}_{\log u, \log |s|, \log (\mb q)/2,\rho} \in 
\Mod \big( \mh H (\widetilde{Z_G}(s |s|^{-1}), L,\mc L,\mb r) / (\mb r - \log (\mb q) /2) \big)
\]
(resp. $\widetilde{M}_{\log u, \log |s|, \log (\mb q)/2,\rho}$) via Theorems \ref{thm:3.26}
and \ref{thm:3.25}, with respect to $s |s|^{-1} \in T$.

\begin{thm}\label{thm:5.6}
\textup{\cite[Theorem 2.11]{AMS3}} \\
Let $\mb q \in \R_{>0}$ and consider triples $(s,u,\rho)$ as above.
\enuma{
\item The maps $(s,u,\rho) \mapsto \overline{E}_{s,u,\rho} \mapsto \overline{M}_{s,u,\rho}$
provide canonical bijections
\[
\{ \text{triples as above} \} / G \longrightarrow \{ \text{standard } 
\mc H (G,L,\mc L)\text{-modules} \} \longrightarrow \Irr (\mc H (G,L,\mc L)) .
\]
\item Suppose that $s \in T$ and let $\gamma_u : SL_2 (\C) \to Z_G (s)$ be an algebraic
homomorphism with $\gamma_u (\matje{1}{1}{0}{1}) = u$. Then $\overline{E}_{s,u,\rho}$ and
$\overline{M}_{s,u,\rho}$ admit the central character $W_L s \gamma_u (\matje{\mb{q}^{1/2}}
{0}{0}{\mb{q}^{-1/2}})$.
\item Suppose that $\mb q \geq 1$. The following are equivalent:
\begin{itemize}
\item $\overline{E}_{s,u,\rho}$ is tempered,
\item $\overline{M}_{s,u,\rho}$ is tempered,
\item $s$ lies in a compact subgroup of $G$.
\end{itemize}
\item Suppose that $\mb q > 1$. Then $\overline{M}_{s,u,\rho}$ is essentially discrete series
if and only if $u$ is distinguished unipotent in $G$ (that is, not contained in any proper
Levi subgroup of $G$).
}
\end{thm}

Notice that $\mb q = 1$ is allowed here. For $\mb q = 1$, Theorem \ref{thm:5.6} provides 
a parametrization of $\Irr (X^* (T) \rtimes W(G,T))$ with $G$-association classes of 
triples $(s,u,\rho)$ as above. That can be regarded as an affine version of the generalized
Springer correspondence from \cite{Lus-Int}.

\begin{ex}
We work out the parametrization from Theorem \ref{thm:5.6} for $(G,L,\mc L)$ as in Examples
\ref{ex:5.1} and \ref{ex:5.2}. Here $\mc H (G,L,\mc L)$ is of type $\widetilde{A_1}$, with 
parameters $\lambda (\alpha) = 4$ and $\lambda^* (\alpha) = 2$. With the example at the end 
of Paragraph \ref{par:homology} at hand, it is easy to determine all geometric parameters 
$(s,u,\rho)$ for $\mc H (G,L,\mc L)$, and the relevant modules can be found by applying 
Theorems \ref{thm:3.26} and \ref{thm:3.25}.
\[
\begin{array}{c|c|c|c}
s & 1 & -1 = \mr{diag}(-1,1,1,-1) & s \in T \cong \C^\times, |s| \geq 1 \\
u & \exp (v + v') & \exp (v + v') & \exp (v) \\
\pi_0 (Z_G (s,u)) & Z(\mr{Sp}_2 (\C)^2) &  Z(\mr{Sp}_2 (\C)^2) &  Z(\mr{Sp}_2 (\C)) \\
\rho & \mr{sign} \boxtimes \triv & \mr{sign} \boxtimes \triv & \mr{sign} \\
\overline{E}_{s,u,\rho} & \mr{St} & \pi (-1,\mr{St}) & \ind_{\C [X^* (T)]}^{\mc H (G,L,\mc L)}(\C_s)
\end{array}
\]
For almost all $s \in T$ the standard module $\ind_{\C [X^* (T)]}^{\mc H (G,L,\mc L)}(\C_s)$ 
is irreducible, and hence equal to $\overline{M}_{s,1,\triv}$. The exceptions are
\[
\overline{M}_{\mr{diag}(\mb q^{3/2},1,1,\mb{q}^{-3/2}),1,\triv} = \triv \quad \text{and} \quad
\overline{M}_{\mr{diag}(-\mb q^{1/2},1,1,-\mb{q}^{-1/2}),1,\triv} = \pi (-1,\triv) . 
\]
A comparison with Paragraph \ref{par:A1t} shows that we indeed found every irreducible
$\mc H (G,L,\mc L)$-representation once in this way.
\end{ex}

We have associated to every complex reductive group $G$ a family of affine Hecke algebras,
one for every cuspidal local system on a nilpotent orbit for a Levi subgroup of $G$.
Every nilpotent orbit admits only a few inequivalent cuspidal local systems, and $G$-conjugate
data $(L,\mc C_v^L,\mc L)$ yield isomorphic Hecke algebras. Thus we have a finite family of
affine Hecke algebras associated to $G$.

The simplest member of this family arises when $L = T, v = 0$ and $\mc L$ is trivial. Then 
the graded Hecke algebras $\mh H (\widetilde{Z_G}(t),T,\mc L = \triv)$ have parameters 
$k(\alpha) = 2$ for all $\alpha \in R (G,T)$. When we specialize $\mb r$ to $\log (\mb q) / 2$,
\eqref{eq:3.24} shows that $\lambda (\alpha) = \lambda^* (\alpha) = 1$ for all 
$\alpha \in R(G,T)$. In other words:
\[
\mc H (G,T,\mc L = \triv) = \mc H (\mc R (G,T),\mb q) .
\]
For this algebra the cuspidal support condition on the triples $(s,u,\rho)$ reduces to the 
condition on $\rho$ in a Kazhdan--Lusztig triple for $(G,1)$. By \cite[Proposition 2.18]{AMS3}, 
the parametrization of standard and irreducible $\mc H (G,T,\mc L = \triv)$-modules in Theorem 
\ref{thm:5.6} coincides with the Kazhdan--Lusztig paramerization from Theorem \ref{thm:5.1}, 
modified as in Corollary \ref{cor:5.3}. That is less obvious than it might seem though, the twist 
with the Iwahori--Matsumoto involution in \eqref{eq:5.9} is necessary to achieve the agreement.

\section{Comparison between different $q$-parameters}

The aim of this section is a canonical bijection between the set of irreducible representations
of an affine Hecke algebra with arbitrary parameters $q_s \in \R_{\geq 1}$, and the set of 
irreducible representations of the same algebra with parameters $q_s = 1$. This will be achieved
in several steps of increasing generality.

\subsection{$W$-types of irreducible tempered representations} \
\label{par:Wtypes}

Consider any graded Hecke algebra $\mh H = \mh H (\mf t,W,k)$. 
The group algebra $\C [W]$ is embedded in $\mh H$, so every $\mh H$-representation can be 
restricted to a $W$-representation. For $k=0$, the isomorphism
\begin{equation}\label{eq:6.1}
\mh H (\mf t,W,0) / (\mf t^*) \cong \C [W]
\end{equation}
shows that a representation on which $\mc O (\mf t)$ acts via evaluation at $0 \in \mf t$ is the
same as a $\C [W]$-representation. From Example \ref{ex:3.7} we know that the irreducible tempered
representations of $\mc H (\mc R,1)$ with central character in $\exp (\mf a)$ are precisely
the irreducible representations which admit the $\mc O (T)$-character $1 \in T$. Via Corollary
\ref{cor:3.27} this implies that the irreducible representations of \eqref{eq:6.1} are precisely
the irreducible tempered $\mh H (\mf t,W,0)$-representations whose central character is real,
that is, lies in $\mf a / W$.

That and the results of Paragraph \ref{par:reduction} indicate that we should focus on
$\mh H$-representations with $\mc O (\mf t)$-weights in $\mf a = X_* (T) \otimes_\Z \R$. 
We say that those have real weights. Let 
\[
\Irr_0 (\mh H (\mf t,W,k)) = \{ \pi \in \Mod_{f,\mf a}(\mh H (\mf t,W,k)) : \pi \text{ is
irreducible and tempered} \}
\]
be the set of irreducible tempered representations with real central character. The above says that
$\Irr_0 (\mh H (\mf t,W,0))$ can be identified with $\Irr (W)$. Since $\mh H (\mf t,W,k)$ is a 
deformation of $\mh H (\mf t,W,0)$, it can be expected that something similar holds for 
$\mh H (\mf t,W,k)$. On closer inspection the parameters $k(\alpha)$ interact with the notion of 
temperedness, and it is natural to require that $k$ takes real values. 

For later applications to affine Hecke
algebras, it will pay off to increase our generality. Let $\Gamma$ be a finite group which
acts on $(\mf a^*, \mf z^*,R, \Delta)$. That is, $\Gamma$ acts $\R$-linearly on $\mf a^*$, and
that action stabilizes $R,\Delta$ and the decomposition $\mf a^* = \R R \oplus \mf z^*$. Suppose
further that $k : R \to \R$ is constant on $\Gamma$-orbits. Then $\Gamma$ acts on $\mh H (\mf t,W,k)$
by the algebra automorphisms
\[
\xi w \mapsto \gamma (\xi) (\gamma w \gamma^{-1}) \qquad \xi \in \mf t^*, w \in W.
\]
The crossed product algebra $\mh H (\mf t,W,k) \rtimes \Gamma = \mh H \rtimes \Gamma$ is of the
kind already encountered in Corollary \ref{cor:3.27}. The $\Gamma$-action on $\mh H$ preserves all
the available structure, so all the usual notions for $\mh H$ also make sense for 
$\mh H \rtimes \Gamma$. 

We denote the restriction of any $\mh H \rtimes \Gamma$-representation $\pi$ to the subalgebra 
$\C [W \rtimes \Gamma]$ by $\Res_{W \rtimes \Gamma}(\pi)$.
An initial result in the direction sketched above is:

\begin{thm} \label{thm:6.1} 
\textup{\cite[Theorem 6.5]{SolHomGHA}} \\
$\Irr_0 (\mh H \rtimes \Gamma)$ and $\Irr (W \rtimes \Gamma)$ have the same cardinality, and the
set $\Res_{W \rtimes \Gamma} (\Irr_0 (\mh H \rtimes \Gamma))$ is linearly independent in the 
representation ring of $W \rtimes \Gamma$.
\end{thm}

In particular it is possible to choose a bijection $\Irr_0 (\mh H \rtimes \Gamma) \to 
\Irr (W \rtimes \Gamma)$ such that the image of $\pi \in \Irr_0 (\mh H \rtimes \Gamma)$ is always
a constituent of $\Res_{W \rtimes \Gamma} (\pi)$. We will establish a much more precise 
version of Theorem \ref{thm:6.1}, for almost all positive parameter functions $k$.

\begin{thm}\label{thm:6.2} 
Let $k : R \to \R_{\geq 0}$ be a $\Gamma$-invariant parameter function whose restriction to any type 
$F_4$ component of $R$ is geometric or has $k(\alpha) = 0$ for a root $\alpha$ in that component.
\enuma{
\item The set $\Res_{W \rtimes \Gamma} (\Irr_0 (\mh H \rtimes \Gamma))$ is a $\Z$-basis of
$\Z \, \Irr (W \rtimes \Gamma)$.
\item There exist total orders on $\Irr_0 (\mh H \rtimes \Gamma)$ and on $\Irr (W \rtimes \Gamma)$
such that the matrix of the $\Z$-linear map
\[
\Res_{W \rtimes \Gamma} : \Z \, \Irr_0 (\mh H \rtimes \Gamma) \to \Z \, \Irr (W \rtimes \Gamma)
\]
is upper triangular and unipotent.
\item There exists a unique bijection
\[
\zeta_{\mh H \rtimes \Gamma} : \Irr_0 (\mh H \rtimes \Gamma) \to \Irr (W \rtimes \Gamma)
\]
such that, for any $\pi \in \Irr_0 (\mh H \rtimes \Gamma)$, $\zeta_{\mh H \rtimes \Gamma} (\pi)$ 
occurs in $\Res_{W \rtimes \Gamma} (\pi)$.
}
\end{thm}
Part (a) is known from \cite[Proposition 1.7]{SolK} and part (c) is a direct consequence of part (b).
Part (b) was already conjectured by Slooten \cite[\S 1.4.5]{Slo1}.

Recall from Table \eqref{table} that the geometric parameter functions for
$R = F_4$ are given by $(k(\alpha),k(\beta))$ equal to
\[
(k,k), (2 k,k), (k,2 k) \text{ or } (4 k,k)
\]
for any $k \in \C^\times$, where $\alpha$ is a short root and $\beta$ is a long root. We expect that
Theorem \ref{thm:6.2}.b also holds for non-geometric parameter functions $k : F_4 \to \R_{>0}$,
but found it too cumbersome to check.

The discussion after \eqref{eq:6.1} shows that the theorem is trivial for $k = 0$, so we assume from 
now on that $k \neq 0$. Our proof of Theorem \ref{thm:6.2} will occupy the entire paragraph.

\begin{lem}\label{lem:6.3}
Theorem \ref{thm:6.2} holds for $\mh H (\mf t,W,k)$ with $k : R \to \R_{>0}$ geometric.
\end{lem}
\begin{proof}
The only irreducible tempered representation with real central character of 
$\mh H (\mf t, R = \emptyset, k) = \mc O (\mf t)$ is $\C_0$, so the result is trivially true
for that algebra. In view of the decomposition of $\mh H$ according to the irreducible components 
of $R$ \eqref{eq:1.22}, we may assume that $R$ is irreducible and spans $\mf a^*$. 

The algebra isomorphisms \eqref{eq:1.23}--\eqref{eq:1.27} are the identity on $\mc O (\mf t)$, so they 
preserve the set $\Irr_0 (\mh H)$. The same holds for $m_z$ from \eqref{eq:1.16}, when 
$z \in \R_{>0}$. Therefore we may just as well suppose that $k$ is one of the parameter functions in 
the table \eqref{table}, and that 
\[
\mh H = \mh H (G,L,\mc L,\mb r) / (\mb r - r) \text{ for some } r \in \R_{>0}.
\] 
By Theorem \ref{thm:5.5} $\Irr_0 (\mh H)$ consists of the representations $\widetilde{M}_{y,0,r,\rho}$
with $y \in \mf g$ nilpotent, $\sigma_0 = 0$, $\rho \in \Irr \big( \pi_0 (Z_G (y)) \big)$ and $
\Psi_G (y,\rho) = (L,\mc C_v^L,\mc L)$ up to $G$-conjugacy. By Lemma \ref{lem:3.29} all the
irreducible constituents of $\widetilde{E}_{y,0,r,\rho}$ are tempered and have central character in
$\mf a / W$. For $r = 0$ all these representations admit the $\mc O (\mf t)$-character 0, so they
can be identified with $W$-representations via \eqref{eq:6.1}. 

Recall from \eqref{eq:5.12} that 
\[
\Res_W \widetilde{E}_{y,0,r,\rho} = \Res_W \widetilde{E}_{y,0,0,\rho} .
\]
Theorem \ref{thm:5.4}.d for $r = 0$ says that all constituents of $\widetilde{E}_{y,0,0,\rho}$
different from $\widetilde{M}_{y,0,0,\rho}$ are of the form $\widetilde{M}_{y',0,0,\rho'}$ with
$\dim \mc C_y^G < \dim \mc C_{y'}^G$. The numbers $\dim \mc C_y^G$ define a partial order on
the set of eligible pairs $(y,\rho)$, considered modulo $G$-conjugation. Refine that to a total order.
We transfer that to a total ordering on $\Irr_0 (\mh H)$ (resp. $\Irr (W)$) via $(y,\rho) \mapsto
\widetilde{M}_{y,0,r,\rho}$ (resp. $\widetilde{M}_{y,0,0,\rho}$). With respect to these orders,
the matrix of
\[
\Res_W : \Z \, \Irr_0 (\mh H ) \to \Z \, \Irr (W)
\]
is unipotent and upper triangular. It follows that  
\[
\zeta_{\mh H} : \widetilde{M}_{y,0,r,\rho} \mapsto \widetilde{M}_{y,0,0,\rho}
\]
is the unique map $\Irr_0 (\mh H ) \to \Irr (W)$ with the required properties. 
\end{proof}

We remark that Lemma \ref{lem:6.3} with respect to Lusztig's alternative version of temperedness
was proven in \cite[\S 3]{Ciu}.

When $R$ is irreducible and all roots have the same length, $k$ is determined by the single number 
$k(\alpha) > 0$, and it is geometric. So we just dealt with $R$ of type $A_n, D_n, E_6, E_7$ or $E_8$.
For $R$ of type $B_n, C_n, F_4$ or $G_2$, let $k_1$ be the $k$-parameter of a long root 
and $k_2$ the $k$-parameter of a short root. We will consider the algebras 
$\mh H (\mf t,W(R),k)$ with $k_1 = 0$ or $k_2 = 0$ in and after \eqref{eq:6.3}.

That settles $R = F_4$ for the moment, because we excluded non-geometric strictly positive parameters.
For $R = G_2$ with strictly positive $k$, Theorem \ref{thm:6.2} was proven in \cite[\S 1.4.4]{Slo1}, 
by working it out completely for all possible cases. 

Amongst the $\mh H$ with $R$ irreducible, that leaves type $B_n$ or $C_n$. In view of the isomorphism 
\eqref{eq:1.23}, it suffices to consider $R = B_n$. Lemma \ref{lem:6.3} proves Theorem \ref{thm:6.2} 
for the following geometric $k$ (and any $p \in \Z_{>0}$):
\[
\begin{array}{ccc}
k_2 / k_1 & \mf g & \mf l \\
\hline
1/2 & \mf{sp}_{2n} & \text{Cartan} \\
p & \mf{so}_{p^2 + 2n} & \mf{so}_{p^2} \oplus \C^n \\
p + 1/2 & \mf{sp}_{p(p+1) + 2n} & \mf{sp}_{p(p+1)} \oplus \C^n \\
p - 1/4 & \mf{sp}_{p(2p-1) + 4n} & \mf{sl}_2^n \oplus \C^n \\
p + 1/4 & \mf{sp}_{2(2p+1) + 4n} & \mf{sl}_2^n \oplus \C^n 
\end{array}
\]
Recall from \eqref{eq:4.2} that a strictly positive $k$ is generic if $\prod_{j=1}^{2(n-1)} (j k_1 - 2 k_2)$ 
is nonzero. In particular we already covered all strictly positive non-generic parameters for $B_n$.

\begin{lem}\label{lem:6.4}
Theorem \ref{thm:6.2} holds for $\mh H (\C^n,W(B_n),k)$ when\\
$k_2 / k_1 \in (p - 1/2, p) \cup (p,p+1/2)$ for a $p \in \Z_{>0}$.
\end{lem}
\begin{proof}
We only consider $k_2 / k_1 \in (p-1/2,p)$, the case $k_2 / k_1 \in (p,p+1/2)$ is completely analogous.
Define $k'$ by $k'_1 = 1$ and $k'_2 = p - 1/4$.

Notice that all $k$ with $k_2 / k_1 \in (p-1/2,p)$ are generic for $B_n$ and for all its parabolic 
root subsystems. Hence the residual subspaces of $\mf a$ for $k$ (Definition \ref{def:4.1}) are canonically 
in bijection with those for $k'$. More precisely, every residual subspace has coordinates that are linear
functions of $k$, see Proposition \ref{prop:4.3} and \eqref{eq:4.3}. If such a linear function gives a
residual subspace for $k'$, then it also gives a residual subspace for $k$, and conversely.
From Theorem \ref{thm:4.4} we obtain a canonical bijection between the sets of irreducible discrete series 
representations of $\mh H_P (\C^n,W(B_n),k')$ and of $\mh H_P (\C^n,W(B_n),k)$, say $\delta' \mapsto \delta$,
where the image depends continuously on $k$. 

By Lemma \ref{lem:3.29} and \eqref{eq:3.19}, $\Irr_0 (\mh H (\C^n, W(B_n),k))$ consists of the 
representations $\pi (P,\delta,0,\rho)$ with $\rho \in \Irr (\C [\mf R_{P,\delta,0},
\natural_{P,\delta,0}])$. 
Recall from Paragraph \ref{par:Rgroups} that the analytic R-group $\mf R_{P,\delta,\lambda}$ is defined in 
terms of the functions $\tilde{c_\alpha}(\lambda)$.  Since $k$ is only allowed to vary among generic 
parameters, the pole order of $\tilde{c_\alpha}|_{\mf a^{P*}}$ at $\lambda = 0$ does not depend on $k$. 
Hence $\mf R_{P,\delta,0}$ does not depend on $k$ either. The intertwining operators $\pi (w,P,\delta,0)
\; (w \in \mf R_{P,\delta,0})$ that span the twisted group
algebra can be constructed so that they depend continuously on $k$. Then $\C [\mf R_{P,\delta,0},
\natural_{P,\delta,0}])$ becomes a continuous family of finite dimensional semisimple $\C$-algebras.
Such algebras cannot be deformed continuously, so the family is isomorphic to a constant family. That 
provides a canonical bijection 
\[
\Irr (\C [\mf R_{P,\delta',0},\natural_{P,\delta',0}]) \to 
\Irr (\C [\mf R_{P,\delta,0},\natural_{P,\delta,0}]) ,
\]
We plug it into \eqref{eq:3.19} and we obtain a bijection
\begin{equation}\label{eq:6.5}
\Irr_0 \big( \mh H (\C^n,W(B_n),k') \big) \to \Irr_0 \big( \mh H (\C^n,W(B_n),k) \big) ,
\end{equation}
where the image depends continuously on $k$. Finite dimensional representations of the finite group $W$
do not admit continuous deformations, so \eqref{eq:6.5} preserves $W$-types.
Knowing that, Theorem \ref{thm:6.2} for $\mh H (\C^n, W(B_n),k')$, as shown in
Lemma \ref{lem:6.3}, immediately implies Theorem \ref{thm:6.2} for $\mh H (\C^n,W(B_n),k)$.
\end{proof}

Now we involve the group $\Gamma$ that acts on $\mh H$ via automorphisms of $(\mf a^*, \mf z^*,R, \Delta)$.

\begin{lem}\label{lem:6.5}
Let $\mh H$ be one of the graded Hecke algebras for which we already proved Theorem \ref{thm:6.2}.
We can choose the total orders on $\Irr_0 (\mh H)$ and on $\Irr (W)$ such that, for any $\pi \in 
\Irr_0 (\mh H)$ and any irreducible constituent $\pi'$ of $\Res_W (\pi)$ different from
$\zeta_{\mh H}(\pi)$, $\gamma^* (\pi') > \zeta_{\mh H} (\pi)$ for all $\gamma \in \Gamma$.
\end{lem}
\begin{proof}
First we assume that $R$ is irreducible and spans $\mf a^*$.

We consider a geometric $k$ and we revisit the proof of Lemma \ref{lem:6.3}. Replacing $\mh H$ by 
an isomorphic algebra, we may assume that $k$ comes from table \eqref{table}. Inspection of the table 
shows that every automorphism of $(R,\Delta)$ can be lifted to an automorphism of $(G,T)$ 
(Recall that there are no automorphisms of $(R,\Delta)$ when $R$ has type $B_d, BC_d, G_2$ or $F_4$, 
apart from the identity.) Hence the function 
\[
\Irr (W) \to \R : \widetilde{M}_{y,0,0,\rho} \mapsto \dim \mc C_y^G
\]
is $\Gamma$-invariant. We replace this function by a function $f_{R,k} : \Irr (W) \to \R$, whose
images differ only slightly from $\dim \mc C_y^G$ and which induces the total order on $\Irr (W)$ 
claimed in Theorem \ref{thm:6.2}.b and exhibited in the proof of Lemma \ref{lem:6.3}.
Via $\zeta_{\mh H}$, we also regard $f_{R,k}$ as a function $\Irr_0 (\mh H)$. Then Theorem
\ref{thm:6.2}.c implies that every constituent of $\Res_{W}(\widetilde{M}_{y,0,r,\rho})$ different
from $\widetilde{M}_{y,0,0,\rho}$ is isomorphic to a $\widetilde{M}_{y',0,0,\rho'}$ with
\begin{equation}\label{eq:6.6}
f_{R,k} \big( \gamma^* (\widetilde{M}_{y',0,0,\rho'}) \big) > f_{R,k}(\widetilde{M}_{y,0,0,\rho})
\qquad \text{for all } \gamma \in \Gamma .
\end{equation}
When $R = G_2$ and $k$ is not geometric, we can use the analysis from \cite[\S 1.4.4]{Slo1} to find
a function $f_{R,k} : \Irr (W) \to \R$ with analogous properties. For other non-geometric $k$, we
may assume that $R = B_n$ (recall we imposed that $k$ is geometric for $R = F_4$). Then $k$ is 
one of the parameter functions considered in Lemma \ref{lem:6.4}. In the proof of that lemma we saw
that $k$ can be deformed continuously to a geometric parameter function $k'$, while staying generic. 
That led to a canonical bijection 
\begin{equation}\label{eq:6.7}
\Irr_0 (\mh H) \to \Irr_0 (\mh H (\C^n,W(B_n),k')),
\end{equation} 
which preserves $W$-types. We define $f_{R,k}$ to be the composition of $f_{R,k'}$ with that
bijection, and we transfer it to a function on $\Irr (W)$ via $\zeta_{\mh H}$. The bijection 
\eqref{eq:6.7} and $\zeta_{\mh H}$ are $\Gamma$-equivariant, if nothing else because $(B_n,\Delta)$
does not admit nontrivial automorphisms. Hence all the properties of $f_{R,k'}$ transfer to
$f_{R,k}$. 

So far we proved the lemma in all cases where $R$ is irreducible and we already had Theorem 
\ref{thm:6.2}, and we made Theorem \ref{thm:6.2}.b more explicit by associating the total order 
to a real-valued function $f_{R,k}$. For a general $R$ we use the decomposition \eqref{eq:1.22}
of $\mh H$. It provides a natural bijection
\[
\begin{array}{ccc}
\Irr_0 (\mh H(\mf t_1,W(R_1),k) \times \cdots \times \Irr_0 (\mh H (\mf t_d,W(R_d),k)) & \to &
\Irr_0 (\mh H (\mf t,W(R),k) \\
(V_1, \ldots, V_d) & \mapsto & V_1 \otimes \cdots \otimes V_d \otimes \C_0 
\end{array}
\]
where $\{ \C_0 \} = \Irr (\mc O (\mf z^*))$. We may assume that all values of the $f_{R_i,k}$
constructed above are algebraically independent and differ from an integer by at most $(8d)^{-1}$
(if not, we can adjust them a bit). Now we define $f_{R,k} : \Irr_0 (\mh H) \to \R$ by
\[
f_{R,k} (V_1 \otimes \cdots V_d \otimes \C_0) = \sum\nolimits_{i=1}^d f_{R_i,k}(V_i).
\] 
and we order $\Irr_0 (\mh H)$ accordingly. Via $\zeta_{\mh H}$ we transfer $f_{R,k}$ and the total 
order to $\Irr (W)$. Let $\pi'$ be a constituent of $\Res_W (V_1 \otimes \cdots \otimes V_d \otimes \C_0)$
different from $\zeta_{\mh H}(V_1 \otimes \cdots \otimes V_d \otimes \C_0)$ and let 
$\pi'_i \in \Irr (W(R_i)), i = 1,\ldots,d$ be its tensor components. 
At least one of the $\pi'_i$ is not isomorphic to 
$\zeta_{\mh H (\mf t_i,W(R_i),k)}(V_i)$. By \eqref{eq:6.6} and its analogues for other irreducible $R$:
\begin{equation}\label{eq:6.8}
f_{R,k}(\gamma^* (\pi')) > f_{R,k}(\pi) + 1 - 2d /8d = f_{R,k}(\zeta_{\mh H}(\pi)) + 3/4
\qquad \text{for all } \gamma \in \Gamma. \qedhere
\end{equation}
\end{proof}

Before we continue, we quickly recall how Clifford theory relates the irreducible representations of 
$\mh H$ and of $\mh H \rtimes \Gamma$. For $(\pi,V_\pi) \in \Irr (\mh H))$ we write
\[
\Gamma_\pi = \{ \gamma \in \Gamma : \gamma^* (\pi) \cong \pi \}.
\]
For every $\gamma \in \Gamma_\pi$ we pick a nonzero intertwining operator 
$I_\gamma : \pi \to \gamma^* (\pi)$. By Schur's lemma $I_\gamma$ is unique up to scalars, 
so there exist $\natural_\pi \in \C^\times$ such that
\begin{equation}\label{eq:6.13}
I_{\gamma \gamma'} = \natural_\pi (\gamma,\gamma') I_\gamma I_{\gamma'}
\qquad \text{for all } \gamma, \gamma' \in \Gamma_\pi .
\end{equation}
Then $\natural_\pi^{\pm 1}$ is a 2-cocycle $\Gamma_\pi \times \Gamma_\pi \to \C^\times$ and the 
twisted group algebra $\C [\Gamma_\pi, \natural_\pi^{-1}]$ 
acts on $V_\pi$ via the $I_\gamma$. For every representation $(\sigma,V_\sigma)$ of 
$\C [\Gamma_\pi, \natural_\pi]$, the vector space
$V_\pi \otimes_\C V_\sigma$ becomes a representation of $\mh H \rtimes \Gamma_\pi$ by
\[
h \gamma \cdot (v_\pi \otimes v_\sigma) = \pi (h) I_\gamma (v_\pi) \otimes \sigma (\gamma) v_\sigma .
\]
When $\sigma$ is irreducible, $V_\pi \otimes V_\sigma$ is also irreducible. Moreover
\[
\pi \rtimes \sigma := \ind_{\mh H \rtimes \Gamma_\pi}^{\mh H \rtimes \Gamma} (V_\pi \otimes V_\sigma)
\]
is an irreducible $\mh H \rtimes \Gamma$-representation. By \cite[Appendix]{RaRa} every irreducible
$\mh H \rtimes \Gamma$-representation is of the form $\pi \rtimes \sigma$, for a pair 
$(\pi,\sigma)$ that is unique up to the $\Gamma$-action. 

The restriction of $\pi \rtimes \sigma$ to $\mh H$ has constituents $\gamma^* (\pi)$ for
$\gamma \in \Gamma / \Gamma_\pi$, each appearing with multiplicity $\dim V_\sigma$. Since $\Gamma$
stabilizes $\mf a$, $\pi \rtimes \sigma$ has all $\mc O (\mf t)$-weights in $\mf a$ if and only if
that holds for $\pi$. As $\Gamma$ stabilizes $\Delta$, it preserves temperedness of 
$\mh H$-representations. Consequently $\pi \rtimes \sigma$ is tempered if and only if $\pi$ is
tempered. In particular $\Irr_0 (\mh H \rtimes \Gamma)$ consists of the representations
$\pi \rtimes \sigma$ with $\pi \in \Irr_0 (\mh H)$ and $\sigma \in \Irr (\C [\Gamma_\pi, \natural_\pi])$.

\begin{lem}\label{lem:6.6}
Let $\mh H$ be one of the graded Hecke algebras for which we already proved Theorem \ref{thm:6.2}
and Lemma \ref{lem:6.5}. Then Theorem \ref{thm:6.2} holds for $\mh H \rtimes \Gamma$.
\end{lem}
\begin{proof}
The same Clifford theory as above can also be used to relate the irreducible representations of
$W$ and of $W \rtimes \Gamma$. Recall that $\Gamma$ acts on $W$ by automorphisms of the Coxeter system
$(W,S)$. In this setting it is known from \cite[Proposition 4.3]{ABPSprin} that the 2-cocycle 
$\natural_{\pi_W}$ associated to any $\pi_W \in \Irr (W)$ is trivial in $H^2 (\Gamma_{\pi_W},\C^\times)$.
Hence we can find $I_\gamma$ for $\pi_W$ such that
\begin{equation}\label{eq:6.9}
\Gamma_{\pi_W} \to \mr{Aut}_\C (V_{\pi_W}) : \gamma \mapsto I_\gamma
\end{equation}
is a group homomorphism. Then $\Irr (W \rtimes W)$ can be parametrized by $\Gamma$-orbits of pairs
$(\pi_W,\sigma_W)$ with $\pi_W \in \Irr (W)$ and $\sigma_W \in \Irr (\Gamma_{\pi_W})$.

We consider any $\pi \in \Irr_0 (\mh H)$. By the uniqueness in Theorem \ref{thm:6.2}.c, $\zeta_{\mh H}$
is $\Gamma$-equivariant and $\Gamma_\pi = \Gamma_{\zeta_{\mh H} (\pi)}$. The intertwiners 
$I_\gamma : \pi \to \gamma^* (\pi)$ also qualify as intertwiners $I_\gamma : \zeta_{\mh H}(\pi) \to 
\gamma^* (\zeta_{\mh H} (\pi))$, because $\zeta_{\mh H}(\pi)$ is contained in $\Res_W (\pi)$ with
multiplicity one. Therefore we can specify a unique $I_\gamma : \pi \to \gamma^* (\pi)$ by the 
requirement that its restriction to $\zeta_{\mh H}(\pi)$ equals the $I_\gamma$ from \eqref{eq:6.9}. Then
\begin{equation}\label{eq:6.27}
\Gamma_\pi \to \mr{Aut}_\C (V_\pi) : \gamma \mapsto I_\gamma
\end{equation}
is a group homomorphism. Now Clifford theory parametrizes $\Irr_0 (\mh H \rtimes \Gamma)$ via
$\Gamma$-orbits of pairs $(\pi,\sigma)$ with $\pi \in \Irr_0 (\mh H)$ and $\sigma \in \Irr (\Gamma_\pi)$.
In particular we obtain a bijection (which will be $\zeta_{\mh H \rtimes \Gamma}$)
\begin{equation}\label{eq:6.10}
\Irr_0 (\mh H \rtimes \Gamma) \to \Irr (W \rtimes \Gamma) : 
\pi \rtimes \sigma \mapsto \zeta_{\mh H} \rtimes \sigma .
\end{equation}
Let $f_{R,k} : \Irr_0 (\mh H) \to \R$ be as in the proof of Lemma \ref{lem:6.5}. We define
\begin{equation}\label{eq:6.12}
f_{R,k} (\pi \rtimes \sigma) = \min \{ f_{R,k}(\gamma^* (\pi)) : \gamma \in \Gamma \} .
\end{equation}
Notice that the irreducible $W \rtimes \Gamma$-representation $\zeta_{\mh H} \rtimes \sigma$ appears
in $\Res_{W \rtimes \Gamma}(\pi \rtimes \sigma)$. For any other irreducible constituent 
$\pi'_{W \rtimes \Gamma}$ of $\Res_{W \rtimes \Gamma}(\pi \rtimes \sigma)$, every irreducible 
$W$-subrepresentations of $\pi'_{W \rtimes \Gamma}$ is contained in $\Res_W (\gamma^* (\pi))$ for some
$\gamma \in \Gamma$. Since $\zeta_{\mh H }(\pi)$ appears with multiplicity one in $\Res_W (\pi)$,
the subspace $\zeta_{\mh H}(\pi) \rtimes \sigma$ of $\Res_{W \rtimes \Gamma}(\pi \rtimes \sigma)$
exhausts the $W$-subrepresentations $\Res_W \big(\gamma^* (\zeta_{\mh H}(\pi)) \big)$ in
$\Res_{W \rtimes \Gamma}(\pi \rtimes \sigma)$. Hence $\pi'_{W \rtimes \Gamma}$ has $W$-constituents
$\gamma^* (\pi'_W)$ with $\pi'_W \subset \Res_W (\pi)$ but $\pi'_W \not\cong \zeta_{\mh H}(\pi)$.
Lemma \ref{lem:6.5} and \eqref{eq:6.8} say that
\begin{equation}\label{eq:6.11}
f_{R,k} (\gamma^* (\pi'_W)) > 3/4 + f_{R,k} (\zeta_{\mh H} (\pi)) \qquad \text{for all } \gamma \in \Gamma.
\end{equation}
Take a total order on $\Irr_0 (\mh H \rtimes \Gamma)$ that refines the partial order defined by $f_{R,k}$.
We transfer $f_{R,k}$ and this total order to $\Irr (W \rtimes \Gamma)$ via the bijection \eqref{eq:6.10}.
Then the above verifies parts (a) and (b) of Theorem \ref{thm:6.2} for $\mh H \rtimes \Gamma$.
It follows this there is a unique $\zeta_{\mh H \rtimes \Gamma}$ that fulfills the requirements,
namely \eqref{eq:6.10}
\end{proof}

Notice that in Lemma \ref{lem:6.4} we did not allow $k_2 / k_1 \in (0,1/2)$. 
Other special cases that we skipped in Lemmas \ref{lem:6.3} and \ref{lem:6.4} are:
\begin{equation}\label{eq:6.3}
\begin{array}{lll@{\qquad}l}
\mh H (\C^n ,W(B_n) ,k) & = & \mh H (\C^n ,W(D_n) ,k_1) \rtimes \langle s_{e_n} \rangle & 
\text{when } k_2 = 0,\\
\mh H (\C^n ,W(B_n) ,k) & = & \mh H (\C^n ,W(A_1)^n ,k_2) \rtimes S_n & \text{when } k_1 = 0, \\
\mh H (\C^4 ,W(F_4), k) & \cong & \mh H (\C^4, W(D_4), k_i) \rtimes S_3 & \text{when } k_{3-i} = 0, \\
\mh H (\C^2, W(G_2), k) & \cong & \mh H (\C^2, W(A_2), k_i) \rtimes S_2 & \text{when } k_{3-i} = 0. 
\end{array}
\end{equation}
Let us relate some of these cases.

\begin{lem}\label{lem:6.7}
Let $k : B_n \to \R_{>0}$ be a parameter function with $k_2 / k_1 \in (0,1/2)$. There exists a canonical
bijection 
\[
\Irr_0 \big( \mh H (\C^n,W(B_n),k) \big) \to 
\Irr_0 \big( \mh H (\C^n,W(D_n),k_1) \rtimes \langle s_{e_n} \rangle \big)
\]
which preserves $W(B_n)$-types.
\end{lem}
\begin{proof}
We abbreviate $\mh H = \mh H (\C^n,W(B_n),k)$ and $\mh H' = \mh H (\C^n,W(D_n),k_1) \rtimes 
\langle s_{e_n} \rangle$. As explained in the proof of Lemma \ref{lem:6.4}, for $k$ within the range
of parameters considered in this lemma, $\Irr_0 (\mh H)$ is essentially independent of $k$. By varying
$k_2$ continuously, we can reach the algebra $\mh H'$, which however may behave differently.
By Clifford theory $\Irr_0 (\mh H')$ consists of representations of the following kinds:

\textbf{(i)} $\ind_{\mh H (\C^n,W(D_n),k_1)}^{\mh H'}(\pi)$, where $\pi \in \Irr_0 (\mh H (\C^n,W(D_n),k_1))$
is not equivalent with $s_{e_n}^* (\pi)$,

\textbf{(ii)} $V_\pi \otimes V_\sigma$, where $\pi \in \Irr_0 (\mh H (\C^n,W(D_n),k_1))$ is fixed by
$s_{e_n}^*$ and $\sigma \in \Irr (\langle s_{e_n} \rangle) = \{ \triv, \mr{sign} \}$.

Let us investigate what happens when we deform $k_2 = 0$ to a positive but very small real number.
Accordingly we replace 
\[
\pi'_2 = \ind_{\mh H (\C^n,W(D_n),k_1)}^{\mh H'}(\pi) \quad \text{by} \quad 
\pi_2 = \ind_{\mh H (\C^n,W(D_n),k_1)}^{\mh H}(\pi). 
\]
The map
\[
w f \mapsto w f \qquad f \in \mc O (\C^n), w \in W(B_n)
\]
is a linear bijection $\mh H' \to \mh H$, so $\Res_{W(B_n)} \pi'_2 = \Res_{W(B_n)} \pi_2$.

\textbf{(i')} We claim that in case (i) $\pi_2$ is still irreducible.

As vector spaces $V_{\pi_2} = V_\pi \oplus s_{e_n} V_\pi$. For any nonzero linear subspace $V$ of $V_{\pi_2}$,
the irreducibility of $\pi'_2$ tells us that there exists an $h \in \mh H'$ such that $\pi'_2 (h) V
\not\subset V$. For $k_2 > 0$ very small, the corresponding element of $\mh H$ still satisfies
$\pi_2 (h) V \not\subset V$. This verifies the claim (i').

\textbf{(ii')} In case (ii), we claim that $\pi_2$ is reducible. 

From Theorem \ref{thm:6.1} we know that
\begin{equation}\label{eq:6.4}
| \Irr_0 (\mh H)| = |\Irr (W(B_n))| = |\Irr_0 (\mh H')|
\end{equation}
(i) and (ii) provide a way to count the right hand side:
\begin{itemize}
\item every $\langle s_{e_n} \rangle$-orbit of length two in $\Irr_0 (\mh H (\C^n,W(D_n),k_1))$
contributes one,
\item every $s_{e_n}^*$-fixed element of $\Irr_0 (\mh H (\C^n,W(D_n),k_1))$ contributes two
\end{itemize}
The restriction of any element of $\Irr_0 (\mh H)$ to $\mh H (\C^n,W(D_n),k_1)$ has all irreducible
constituents in $\Irr_0 (\mh H (\C^n,W(D_n),k_1))$. By Frobenius reciprocity, this implies that it is
a constituent of $\pi_2$ for some $\pi \in \Irr_0 (\mh H (\C^n,W(D_n),k_1))$. 

When $s_{e_n}^* (\pi) \not\cong \pi$, we saw in (i') that $\{ \pi, s_{e_n}^* (\pi) \}$ contributes
just one representation to $\Irr_0 (\mh H)$. 
In case $s_{e_n}^* (\pi) \cong \pi$, $\pi_2$ can be reducible. It has length at most two, 
because that is its length as $\mh H (\C^n,W(D_n),k_1)$-module. When $\pi_2$ would contribute only one 
representation to $\Irr_0 (\mh H)$, the sum of the contributions from the cases (i') and (ii') would be 
strictly smaller than the corresponding sum of the contributions from (i) and (ii) to $\Irr (\mh H')$.
However, that would contradict \eqref{eq:6.4}. We conclude that (ii') holds.

A $\pi_2$ as in (ii') has length 2, and both its irreducible constituents become isomorphic to
$\pi$ upon restriction to $\mh H (\C^n,W(D_n),k_1)$. As 
\[
\Res_{W(B_n)}(\pi_2) = \ind_{W(D_n)}^{W(B_n)} \Res_{W(D_n)}(\pi),
\]
the restriction to $W(B_n)$ of the two constituents of $\pi_2$ must be 
$\Res_{W(D_n)} (\pi) \otimes \triv$ and $\Res_{W(D_n)}(\pi) \otimes \mr{sign}$. 
Here $\Res_{W(D_n)}(\pi)$ extends to a representation of $W(B_n)$, while triv and sign are 
representations of $W(B_n) / W(D_n)$. In combination with case (i') we see that
\[
\Res_{W(B_n)}(\Irr_0 (\mh H')) = \Res_{W(B_n)} (\Irr_0 (\mh H)) .
\]
Hence there is a unique bijection $\Irr_0 (\mh H) \to \Irr_0 (\mh H')$ that preserves $W$-types.
\end{proof}

Finally, we settle the remaining cases of Theorem \ref{thm:6.2}.

\begin{proof}
Lemma \ref{lem:6.6} establishes Theorem \ref{thm:6.2} for the algebras in \eqref{eq:6.3}. Moreover
\eqref{eq:6.12} gives us a function $f_{R,k}$ that defines a useful partial order on
$\Irr_0 (\mh H)$ and $\Irr (W)$. With Lemma \ref{lem:6.7} we transfer all that to $\mh H (\C^n,W(B_n),k)$ 
with $k_2 / k_1 \in (0,1/2)$. Then we have Theorem \ref{thm:6.2} whenever $R$ is irreducible and
spans $\mf a^*$. As noted in the proof of Lemma \ref{lem:6.3}, that implies Theorem \ref{thm:6.2}
for all $\mh H$ (still with the condition on the parameters for type $F_4$ components).  
We finish the proof by applying Lemmas \ref{lem:6.5} and \ref{lem:6.6} another time.
\end{proof}

Suppose that $\mh H = \mh H (G,L,\mc L,\mb r) / (\mb r - r)$ for a cuspidal local system $\mc L$
on a nilpotent orbit for $L$ (as in Paragraph \ref{par:homology}). In terms of Theorem \ref{thm:5.5},
$\zeta_{\mh H}$ from Theorem \ref{thm:6.2} is just the map
$\widetilde{M}_{y,0,r,\rho} \mapsto \widetilde{M}_{y,0,0,\rho}$.
Here $(y,\rho) \mapsto \widetilde{M}_{y,0,0,\rho}$ is the generalized Springer correspondence from
\cite{Lus-Int}, twisted by the sign character of $W$. So, for a graded Hecke algebra that can be
constructed with equivariant homology, Theorem \ref{thm:6.2} recovers a generalized Springer 
correspondence for $W$.

Let us be more flexible, and call any nice parametrization of $\Irr (W)$ a generalized Springer
correspondence. Then Theorem \ref{thm:6.2} qualifies as such, and we can regard
$\zeta_{\mh H} : \Irr_0 (\mh H) \to \Irr (W)$ as a "generalized Springer correspondence with
graded Hecke algebras". This point of view has been pursued in \cite{Slo1}, where 
$\Irr_0 (\mh H (\C^n,W(B_n),k))$ has been parametrized with combinatorial data that mimic
the above pairs $(y,\rho)$.

\subsection{A generalized Springer correspondence with affine Hecke algebras} \

With Corollary \ref{cor:3.27} we can translate Theorem \ref{thm:6.2} into a statement about
all tempered irreducible representations of $\mc H = \mc H (\mc R, \lambda, \lambda^*,\mb q)$,
in relation with tempered $\C [X \rtimes W]$-representations. We want to generalize that
to all irreducible $\mc H$-representations, at least when $q_s \geq 1$ for all $s \in S_\af$.

Our main tool will be the Langlands classification. We need a version for graded Hecke algebras
extended with automorphism groups $\Gamma$ as in Theorem \ref{thm:6.2}. It can be obtained by
combining Theorem \ref{thm:3.10} for $\mh H$ with Clifford theory. Given a Langlands datum
$(P,\tau,\lambda)$ for $\mh H$, let $\Gamma_{P,\tau,\lambda}$ be its stabilizer in $\Gamma$,
and recall the 2-cocycle $\natural_{P,\tau,\lambda}$ from \eqref{eq:6.13}. We define a Langlands
datum for $\mh H \rtimes \Gamma$ to be a quadruple $(P,\tau,\lambda,\rho)$, where 
\begin{itemize}
\item $(P,\tau,\lambda)$ is a Langlands datum for $\mh H$ (so $\tau \in \Irr (\mh H_P)$
is tempered and $\lambda \in \mf a^{P++} + i \mf a^P$),
\item $\rho \in \Irr (\C [\Gamma_{P,\tau,\lambda},\natural_{P,\tau,\lambda}])$.
\end{itemize}
To such a quadruple we associate the irreducible $\mh H^P$-representation $((\tau \otimes \lambda) 
\otimes \rho, V_\tau \otimes_\C V_\rho)$ and the $\mh H \rtimes \Gamma$-representation
\begin{equation}\label{eq:6.29}
\pi (P,\tau,\lambda,\rho) = \ind_{\mh H^P \rtimes \Gamma_{P,\tau,\lambda}}^{\mh H \rtimes \Gamma}
(\tau \otimes \lambda \otimes \rho ) .
\end{equation}
As a consequence of \cite[Corollary 2.2.5]{SolAHA} and Paragraph \ref{par:analogues}, we find
an extended Langlands classification:

\begin{cor}\label{cor:6.8}
Let $(P,\tau,\lambda,\rho)$ be a Langlands datum for $\mh H \rtimes \Gamma$.
\enuma{
\item The $\mh H \rtimes \Gamma$-representation $\pi (P,\tau,\lambda,\rho)$ has a unique
irreducible quotient, which we call $L(P,\tau,\lambda,\rho)$.
\item For every irreducible $\mh H \rtimes \Gamma$-representation $\pi$, there exists a 
Langlands datum $(P',\tau',\lambda',\rho')$, unique up to the canonical $\Gamma$-action,
such that $\pi \cong \pi (P',\tau',\lambda',\rho')$.
\item $L (P,\tau,\lambda,\rho)$ and $\pi (P,\tau,\lambda,\rho)$ are tempered if and onlu if
$P = \Delta$ and $\lambda \in i \mf a^\Delta$.
}
\end{cor}

In this context we call $\pi (P,\tau,\lambda,\rho)$ a standard $\mh H \rtimes \Gamma$-module.
When $\tau$ has real central character, it follows from Theorem \ref{thm:6.2} that 
the representation $\Res_{W_P}(\tau) \otimes \lambda$ of $\mh H (\mf t, W_P, 0) = 
\mc O (\mf t) \rtimes W_P$ has stabilizer $\Gamma_{P,\tau,\lambda}$ in $\{ \gamma \in \Gamma :
\gamma (P) = P \}$. The $\mh H^P$-intertwining operators $I_\gamma \; (\gamma \in 
\Gamma_{P,\tau,\lambda})$ from \eqref{eq:6.3} are also $\mc O (\mf t) \rtimes W_P$-intertwining
operators, so $(\Res_{W_P}(\tau) \otimes \lambda \otimes \rho, V_\tau \otimes_\C V_\rho)$
is a well-defined representation of 
\[
\mh H (\mf t,W_P,0) \rtimes \Gamma_{P,\tau,\lambda} = 
\mc O (\mf t) \rtimes W_P \Gamma_{P,\tau,\lambda}.
\]
Its parabolic induction is 
\begin{equation}\label{eq:6.28}
\pi (P, \Res_{W_P}(\tau), \lambda, \rho) = \ind_{\mc O (\mf t) \rtimes W_P \Gamma_{P,\tau,\lambda}}^{
\mc O (\mf t) \rtimes W \Gamma} (\Res_{W_P}(\tau) \otimes \lambda \otimes \rho) .
\end{equation}
This representation has central character $W \Gamma \lambda$ and the construction mimics that in 
\eqref{eq:6.29}, so 
\[
\pi (P, \Res_{W_P}(\tau), \lambda, \rho) = \pi (P, \tau, \lambda, \rho) \quad \text{as} \quad
\C [W \Gamma]\text{-representations}.
\] 
We note that in general $\pi (P, \Res_{W_P}(\tau), \lambda, \rho)$ differs from 
$\pi (P, \Res_{W_P \rtimes \Gamma_{P,\tau,\lambda}}(\tau \otimes \lambda \otimes \rho)$, because
the latter has central character $0 \in \mf t / W \Gamma$. 
Of course the $\mh H \rtimes \Gamma$-representation \eqref{eq:6.28} may have more than one 
irreducible quotient, because $\Res_{W_P}(\tau)$ usually is reducible. 
Recall from the proof of Lemma \ref{lem:6.6} that we can arrange that 
\[
\Gamma_{P,\tau,\lambda} \to \mr{Aut}_\C (V_\pi) : \gamma \mapsto I_\gamma
\]
is a group homomorphism. Then $\natural_{P,\tau,\lambda} = 1$ and $\rho$ becomes simply an 
irreducible representation of $\Gamma_{P,\tau,\lambda}$. This construction yields a canonical map
\begin{equation}\label{eq:6.14}
\begin{array}{cccc}
\Res_{\mc O (\mf t) \rtimes W \Gamma} : & \{ \text{standard } \mh H \rtimes \Gamma\text{-modules}\} &
\longrightarrow & \Mod_f (\mc O (\mf t) \rtimes W \Gamma ) \\
 & \pi (P,\tau,\lambda,\rho) & \mapsto & \pi (P, \Res_{W_P}(\tau), \lambda, \rho) .
\end{array}
\end{equation}
For $\lambda = 0$, this just the restriction map $\Res_{W \rtimes \Gamma}$, in combination with
\eqref{eq:6.1}. In terms of \eqref{eq:3.31} and \eqref{eq:3.19}, we can express \eqref{eq:6.14} as 
\[
\Res_{\mc O (\mf t) \rtimes W \Gamma} \big( \Hom_{\mf R_{P,\delta,\lambda}} \big( \rho,
\ind_{\mh H^P}^{\mh H \rtimes \Gamma} (\delta \otimes \lambda) \big) \big) =
\Hom_{\mf R_{P,\delta,\lambda}} \big( \rho, \ind_{\mc O (\mf t) \rtimes W_P}^{\mc O (\mf t) 
\rtimes W \Gamma} \big( \Res_{W_P}(\delta) \otimes \lambda \big) \big) .
\]
Here we used \cite[Theorem 9.2]{SolGHA} to extend the notion of R-groups to $\mh H \rtimes \Gamma$.

\begin{lem}\label{lem:6.9}
Let $k : R \to \R$ be a parameter function as in Theorem \ref{thm:6.2}, so that in particular Theorem
\ref{thm:6.2}.b provides a total order $>$ on $\Irr (W_P)$. 
Let $(P,\tau,\lambda,\rho)$ be a Langlands datum for $\mh H \rtimes \Gamma$.

All irreducible constituents of $\pi (P,\Res_{W_P}(\tau),\lambda,\rho)$ different from 
$\pi (P,\zeta_{\mh H_P}(\tau),\lambda,\rho)$ are of the form
$\pi (P,\tau'_W ,\lambda,\rho'_W )$, where $\tau'_W > \zeta_{\mh H_P}(\tau)$ and
$(P,\tau'_W ,\lambda,\rho'_W )$ is a Langlands datum for $\mh H (\mf t,W,0) \rtimes \Gamma =
\mc O (\mf t) \rtimes W \Gamma$.
\end{lem}
\begin{proof}
By Theorem \ref{thm:6.2} every irreducible constituent $\tau'_W$ of $\Res_{W_P}(\tau)$ 
different from $\zeta_{\mh H_P}(\tau)$ is strictly larger than $\zeta_{\mh H_P}(\tau)$. 
Although the $\mc O (\mf t) \rtimes W_P$-representation $\tau'_W \otimes \lambda$ is irreducible,
its stabilizer in 
\[
\Gamma_{P,\lambda} = \{ \gamma \in \Gamma : \gamma (P) = P, \gamma (\lambda) = \lambda \}
\]
need not be $\Gamma_{P,\tau,\lambda}$. To overcome that, we rather work with $\Gamma_{P,\lambda}$.
Put $\tau' = \ind_{\mh H_P \rtimes \Gamma_{P,\tau,\lambda}}^{\mh H_P \rtimes \Gamma_{P,\lambda}}
(\tau \otimes \rho)$, so that 
\[
\pi (P,\tau,\lambda,\rho) = \ind_{\mh H_P \rtimes \Gamma_{P,\lambda}}^{\mh H \rtimes \Gamma}
(\tau' \otimes \lambda) .
\]
Take $\rho'_W \in \Irr (\Gamma_{P,\tau'_W,\lambda})$ such that 
\begin{equation}\label{eq:6.17}
\ind_{\mc O (\mf t) \rtimes W_P 
\Gamma_{P,\tau'_W,\lambda}}^{\mc O (\mf t) \rtimes W_P \Gamma_{P,\lambda}}(\tau'_W \otimes \rho'_W)
\end{equation} 
is a subrepresentation of
\begin{equation}\label{eq:6.15}
\Res_{W_P \rtimes \Gamma_{P,\lambda}} (\tau') = \ind_{\mc O (\mf t) \times W_P \Gamma_{P,\tau,
\lambda}}^{\mc O (\mf t) \rtimes W_P \Gamma_{P,\lambda}} (\Res_{W_P}(\tau) \otimes \rho) .
\end{equation}
Then $\pi (P,\tau'_W,\lambda,\rho'_W)$ is a subrepresentation of $\pi (P,\Res_{W_P}(\tau),\lambda,
\rho)$, and by Corollary \ref{cor:6.8} it is irreducible. With this construction we can obtain any
subrepresentation of \eqref{eq:6.15} whose $W_P$-constituents are not $\Gamma_{P,\lambda}$-associate 
to $\zeta_{\mh H_P}(\tau)$.

Since $\zeta_{\mh H_P}(\tau)$ appears with multiplicity one in $\Res_{W_P}(\tau)$, 
$\zeta_{\mh H_P}(\tau) \otimes \lambda \otimes \rho$ exhausts all $\Gamma_{P,\tau,\lambda}$-associates
of $\zeta_{\mh H_P}(\tau)$ in $\Res_{W_P}(\tau) \otimes \lambda \otimes \rho$. Then 
\begin{equation}\label{eq:6.16}
\ind_{\mc O (\mf t) \times W_P \Gamma_{P,\tau,\lambda}}^{\mc O (\mf t) \rtimes W_P \Gamma_{P,\lambda}}
(\zeta_{\mh H_P}(\tau) \otimes \lambda \otimes \rho)
\end{equation} 
exhausts all $\Gamma_{P,\lambda}$-associates of $\zeta_{\mh H_P}(\tau)$ in
\begin{equation}\label{eq:6.18} 
\ind_{\mc O (\mf t) \times W_P \Gamma_{P,\tau,\lambda}}^{\mc O (\mf t) 
\rtimes W_P \Gamma_{P,\lambda}} (\Res_{W_P}(\tau) \otimes \lambda \otimes \rho). 
\end{equation}
Consequently \eqref{eq:6.16} and the modules \eqref{eq:6.17} exhaust the whole of \eqref{eq:6.18}. 
This remains the case after inducing everything to $\mc O (\mf t) \rtimes W \Gamma$.\\ 
Therefore
$\pi (P,\Res_{W_P}(\tau),\lambda,\rho)$ does not have any other subrepresentations besides
$\pi (P,\zeta_{\mh H_P}(\tau),\lambda,\rho)$ and the $\pi (P,\tau'_W,\lambda,\rho'_W)$.
\end{proof}

From \eqref{eq:6.14} we will deduce a map whose image consists of irreducible representations of
$\mc O (\mf t) \rtimes W \Gamma$. By Corollary \ref{cor:2.1}, those are the same as standard
$\mc O (\mf t) \rtimes W \Gamma$-modules. We call a central character for $\mh H \rtimes \Gamma$
real if it lies in $\mf a / W\Gamma$.

\begin{prop}\label{prop:6.10}
Let $\mh H \rtimes \Gamma$ be as in Theorem \ref{thm:6.2}. 
\enuma{
\item There exists a unique bijection $\zeta_{\mh H \rtimes \Gamma}$ between:
\begin{itemize}
\item the set of standard $\mh H \rtimes \Gamma$-modules with real central character,
\item the set of irreducible $\mc O (\mf t) \rtimes W \Gamma$-representations with real central character,
\end{itemize}
such that $\zeta_{\mh H \rtimes \Gamma}(\pi)$ is always a constituent of 
$\Res_{\mc O (\mf t) \rtimes W \Gamma}(\pi)$.

For suitable total orders on these two sets, the matrix of (the linear extension of) 
$\zeta_{\mh H \rtimes \Gamma}$ is the identity while the matrix of $\Res_{\mc O (\mf t) \rtimes W \Gamma}$
is upper triangular and unipotent. 
\item There exists a natural bijection $\zeta'_{\mh H \rtimes \Gamma}$ from the set of irreducible 
$\mh H \rtimes \Gamma$-representa\-tions with real central character to the analogous set for 
$\mc O (\mf t) \rtimes W\Gamma$, such that:
\begin{itemize}
\item for any standard module $\pi$, $\zeta_{\mh H \rtimes \Gamma}(\pi)$ equals 
$\zeta'_{\mh H \rtimes \Gamma}$ of the irreducible quotient of $\pi$,
\item for tempered representations it coincides with part (a).
\end{itemize}
}
\end{prop}
\begin{proof}
(a) By Lemma \ref{lem:6.9}, any candidate for such a map must preserve the $P$ and the $\lambda$ 
in the Langlands datum of $\pi$ (from Corollary \ref{cor:6.8}). By Lemma \ref{lem:3.3} the condition 
on the central character of $\pi$ means that $\lambda \in \mf a^{P++}$ and the ingredient $\tau$ 
lies in $\Irr_0 (\mh H_P)$. Hence we can specialize to a map
\begin{multline*}
\{ (\tau,\rho) : \tau \in \Irr_0 (\mh H_P), \rho \in \Irr (\Gamma_{P,\tau,\lambda}) \} \longrightarrow \\
\{ (\tau_W, \rho_W) : \tau_W \in \Irr (W_P), \rho_W \in \Irr (\Gamma_{P,\tau_W,\lambda}) \} . 
\end{multline*}
Here the left hand side parametrizes $\Irr_0 (\mh H_P \rtimes W_{P,\lambda})$ and the right hand side
parametrizes $\Irr (W_P \rtimes \Gamma_{P,\lambda})$. For fixed $(P,\lambda)$ there is a commutative
diagram 
\begin{equation}\label{eq:6.26}
\begin{array}{ccc}
\left\{ \begin{array}{c} \text{standard } \mh H \rtimes \Gamma\text{-modules}\\
\text{with real central character} \end{array} \right\} &
\xrightarrow{\Res_{\mc O (\mf t) \rtimes W \Gamma}} & \Mod_f (\mc O (\mf t) \rtimes W \Gamma) \\
\uparrow \ind_{\mh H^P \rtimes \Gamma_{P,\lambda}}^{\mh H \rtimes \Gamma} \otimes \lambda& & 
\uparrow \ind_{\mc O (\mf t) \rtimes W_P \Gamma_{P,\lambda}}^{\mc O (\mf t) 
\rtimes W \Gamma} \otimes \lambda\\
\Irr_0 (\mh H_P \rtimes \Gamma_{P,\lambda}) & \xrightarrow{\Res_{W_P \rtimes \Gamma_{P,\lambda}}} &
\Mod_f ( \mc O (\mf t_P) \rtimes W_P \Gamma_{P,\lambda})
\end{array}
\end{equation}
where the vertical arrows send $\tau \rtimes \rho$ to $\pi (P,\tau,\lambda,\rho)$ and
$\tau_W \rtimes \rho_W$ to $\pi (P,\tau_W,\lambda,\rho_W)$. By Theorem \ref{thm:6.2} for
$\mh H_P \rtimes \Gamma_{P,\lambda}$, the matrix of $\Res_{W_P \rtimes \Gamma_{P,\lambda}}$ (with
respect to suitable total orders) is upper triangular and unipotent. Hence there is a unique map 
$\zeta_{\mh H \rtimes \Gamma}$ that fulfills the requirements for fixed $(P,\lambda)$, namely
\begin{equation}\label{eq:6.22}
\ind_{\mh H^P \rtimes \Gamma_{P,\lambda}}^{\mh H \rtimes \Gamma} (\tau \rtimes \rho \otimes \lambda)
\; \mapsto \; \ind_{\mc O (\mf t) \rtimes W_P \Gamma_{P,\lambda}}^{\mc O (\mf t) \rtimes W \Gamma} 
(\zeta_{\mh H_P \rtimes \Gamma_{P,\lambda}}(\tau \rtimes \rho) \otimes \lambda ) .
\end{equation}
This translates to 
\begin{equation}\label{eq:6.19}
\zeta_{\mh H \rtimes \Gamma} \pi (P,\tau,\lambda,\rho) = \pi (P,\zeta_{\mh H_P}(\tau), \lambda,\rho).
\end{equation}
Theorem \ref{thm:6.2} and the commutative diagram \eqref{eq:6.26} entail the required properties of
the matrices of $\zeta_{\mh H \rtimes \Gamma}$ and $\Res_{\mc O (\mf t) \rtimes W \Gamma}$.\\
(b) By Corollary \ref{cor:6.8} taking the irreducible quotient of a standard module provides a
natural bijection
\begin{equation}\label{eq:6.21}
\{\text{standard } \mh H \rtimes \Gamma\text{-modules}\} \longrightarrow \Irr (\mh H \rtimes \Gamma).
\end{equation}
Define $\zeta'_{\mh H \rtimes \Gamma}$ as the composition of the inverse of \eqref{eq:6.21} with
$\zeta_{\mh H \rtimes \Gamma}$ from part (a). By Corollary \ref{cor:6.8} every tempered irreducible
module is also standard, so for tempered representations the properties of 
$\zeta_{\mh H \rtimes \Gamma}$ remain valid for $\zeta'_{\mh H \rtimes \Gamma}$. 
\end{proof}

With Corollary \ref{cor:3.27} we will transfer Proposition \ref{prop:6.10} to affine Hecke 
algebras with parameters in $\R_{\geq 1}$. We may and will take 
$\mc H = \mc H (\mc R,\lambda,\lambda^*,\mb q)$ as in Lemma \ref{lem:6.11}.
Let $u \in T_\un$. By \eqref{eq:3.24} the parameter function $k_u$ for 
$\mh H (\mf t, W(R_u),k_u) \rtimes \Gamma_u$ takes values in $\R_{\geq 0}$. Recall from 
Lemma \ref{lem:3.30} that Theorems \ref{thm:3.25} and \ref{thm:3.26} are compatible with
temperedness and parabolic induction. Let $\pi (P,\tau,\lambda,\rho)$ be a standard module for 
$\mh H (\mf t, W(R_u),k_u) \rtimes \Gamma_u$, with real central character. Via Theorems 
\ref{thm:3.25} and \ref{thm:3.26} it corresponds naturally to a standard $\mc H$-module 
$\pi (P',\tau',t)$ with $\tau' \in \Irr (\mc H_{P'})$ tempered and $|t| = \exp (\lambda)$. 
We define
\[
\Res_{\mc O (T) \rtimes W} : \{ \text{standard } \mc H\text{-modules} \} \longrightarrow
\Mod_f (\mc O (T) \rtimes W) 
\]
by commutativity of the following diagram (for central characters in $W u \exp (\mf a)$):
\[
\begin{array}{ccc}
\{ \text{standard } \mc H\text{-modules} \} & \xrightarrow{\Res_{\mc O (T) \rtimes W}} &
\Mod_f (\mc O (T) \rtimes W) \\
\uparrow \text{Corollary \ref{cor:3.27}} & & \uparrow  \text{Corollary \ref{cor:3.27}} \\
\! \left\{ \!\! \begin{array}{c} \text{standard } \mh H (\mf t, W(R_u),k_u) \rtimes \Gamma_u -\\
\text{modules with real central character} \end{array} \!\! \right\} & \!\! 
\xrightarrow{\Res_{\mc O (\mf t) \rtimes W(R_u) \Gamma_u}} \!\! & 
\Mod_f (\mc O (\mf t) \rtimes W(R_u) \Gamma_u )
\end{array}
\]
The map $\Res_{\mc O (T) \rtimes W}$ can be made more explicit with \eqref{eq:3.15}
and Corollary \ref{cor:3.23}. In those terms
\begin{equation}\label{eq:6.20}
\Res_{\mc O (T) \rtimes W} \big( \pi (P,\delta,t,\rho) \big) = 
\Hom_{\C [\mf R_{P,\delta,t},\natural_{P,\delta,t}]} \big( \rho , \pi (P,\Res_{W_P}(\delta),t) \big).
\end{equation}
For tempered standard $\mc H$-modules (i.e. with $t \in T^P_\un$), $\Res_{\mc O (T) \rtimes W}$ can
really be considered as a restriction, see \cite[\S 4.4]{SolAHA}. The above map is the natural
generalization to all standard $\mc H$-modules. However, it is not a restriction (along some 
injective algebra homomorphism) because it can happen that $\pi (P,\delta,t,\rho)$ is reducible but
its image \eqref{eq:6.20} is irreducible.

We do not know how to extend $\Res_{\mc O (T) \rtimes W}$ to arbitrary $\mc H$-representations.
The best we can do is to define the $W$-type of any finite dimensional 
$\mc H$-representation, in the following way:
\begin{itemize}
\item By decomposing it as in \eqref{eq:3.32}, we may assume that all its $\mc O (T)$-weights
lie in a single $W$-orbit, say in $W u \exp (\mf a)$ with $u \in T_\un$.
\item Apply Theorems \ref{thm:3.25} and \ref{thm:3.26} to produce a representation of\\
$\mh H (\mf t,W(R_u),k_u) \rtimes \Gamma_u$.
\item Restrict to $\C [W(R_u) \rtimes \Gamma_u]$ and then induce to $\C [W]$.
\end{itemize}
Of course this mimics the earlier $W$-type maps for graded Hecke algebras. In 
\cite[\S 4.1--4.2]{SolAHA} it was shown that the $W$-type of an $\mc H$-representation can also be
obtained via a continuous deformation of $\mb q$ to 1. 

We are ready transfer Proposition \ref{prop:6.10} to affine Hecke algebras:

\begin{thm}\label{thm:6.12}
Let $\mc H = \mc H (\mc R,\lambda,\lambda^*,\mb q)$ be an affine Hecke algebra with parameter functions
$\lambda,\lambda^* : R \to \R_{\geq 0}$. Suppose that the restrictions $\lambda_i = \lambda_i^*$ to any 
type $F_4$ component $R_i$ of $R$ satisfy: either $\lambda_i$ is geometric or $\lambda_i (\alpha) = 0$
for an $\alpha \in R_i$.
\enuma{
\item There exists a unique bijection
\[
\zeta_{\mc H} : \{ \text{standard } \mc H\text{-modules} \} \longrightarrow \Irr (\mc O (T) \rtimes W)
\]
such that $\zeta_{\mc H}(\pi)$ is always a constituent of $\Res_{\mc O (T) \rtimes W}(\pi)$.

There exists a total order on $\Irr (\mc O (T) \rtimes W)$ such that, if we transfer it via $\zeta_{\mc H}$,
the matrix of the $\Z$-linear map 
\[
\Res_{\mc O (T) \rtimes W} : \Z \, \{ \text{standard } \mc H\text{-modules} \} 
\longrightarrow \Z \, \Irr (\mc O (T) \rtimes W)
\]
becomes upper triangular and unipotent.
\item There exists a natural bijection
\[
\zeta'_{\mc H} : \Irr (\mc H) \longrightarrow \Irr ( X \rtimes W) 
\]
such that:
\begin{itemize}
\item for standard module $\pi$, $\zeta_{\mc H}(\pi)$ equals $\zeta'_{\mh H}$ of the irreducible
quotient of $\pi$,
\item for irreducible tempered $\mc H$-representations it coincides with part (a).
\end{itemize}
}
\end{thm}
The allowed parameter functions for a type $F_4$ component of $R$ are
\[
(\lambda (\alpha), \lambda (\beta)) \in \{ (k,k), (2k,k), (k,2k), (4k,k), (k,0), (0,k), (0,0) \},
\]
where $k \in \R_{>0}, \alpha \in F_4$ is a short root and $\beta \in F_4$ is a long root. Like we 
mentioned after Theorem \ref{thm:6.2}, we believe that Theorem \ref{thm:6.12} is valid for all
parameter functions $\lambda,\lambda^* : R \to \R_{\geq 0}$.

We point out that the uniqueness/naturality is essential in Theorem \ref{thm:6.12}. Without that
condition, it would be much easier to derive it from \cite[\S 2.3]{SolAHA}.
\begin{proof}
By Lemma \ref{lem:6.11} and the asserted naturality of $\zeta_{\mc H}$ and $\zeta'_{\mc H}$, 
we may assume that $\lambda (\alpha) \geq \lambda^* (\alpha)$ for all $\alpha \in R$. 
For any $u \in T_\un$, the parameter function $k_u$ from \eqref{eq:3.24} takes values in $\R_{\geq 0}$.

For any irreducible component $R_i$ of $R$ not of type $F_4$, we claim that $R_i$ does not possess
a root subsystem isomorphic to $F_4$. This can be seen with a case-by-case consideration of 
irreducible root systems. To get a root subsystem of type $F_4$,
$R_i$ needs to possess roots of different lengths, so it has type $B_n,C_n,F_4$ or $G_2$. The rank
of $G_2$ is too low, so $R \not\cong G_2$. In type $B_n$ (resp. $C_n$) the short (resp. long) roots
form a subsystem of type $(A_1)^n$, and that does not contain $D_4$. As both the long and the short
roots in $F_4$ form a root system of type $D_4$, we can conclude that $B_n \not\cong R_i \not\cong C_n$.

Now we apply Corollary \ref{cor:3.27} and reduce the theorem to the graded Hecke algebras
$\mh H (\mf t,W(R_u),k_u) \rtimes \Gamma_u$, for all $u \in T_\un$. By the above, the parameter 
function $k_u$ fulfills the requirements of Theorem \ref{thm:6.2}. Finally, we apply 
Proposition \ref{prop:6.10}. 
\end{proof}

We note that Theorem \ref{thm:6.12} provides a natural bijection between the irreducible or standard
modules of two affine Hecke algebras with the same root datum but different parameters $q_s \geq 1$.

When $\mc H$ arises from a cuspidal local system $\mc L$ on a nilpotent orbit for a Levi subgroup
$L$ of $G$, Theorem \ref{thm:6.12} is related to Theorem \ref{thm:5.6}. We claim that in this case
\begin{equation}\label{eq:6.25}
\begin{array}{rll}
\Res_{\mc O (T) \rtimes W}(\overline{E}_{s,u,\rho}) & = & \overline{E}_{s,u,\rho} ,\\
\zeta_{\mc H (G,L,\mc L)} (\overline{E}_{s,u,\rho}) \; = \;
\zeta'_{\mc H (G,L,\mc L)} (\overline{M}_{s,u,\rho}) & = & \overline{M}_{s,u,\rho} ,
\end{array}
\end{equation}
where the terms on the right are representations of $\mc O (T) \rtimes W$, the version of
$\mc H (G,L,\mc L)$ with $\mb q = 1$. The reasons for the compatibility between these two ways to
go from $\mc H$-modules to $\mc O (T) \rtimes W$-modules are: 
\begin{itemize}
\item the constructions that led to $\zeta_{\mc H}$ are analogous to those behind Theorem \ref{thm:5.6},
\item for graded Hecke algebras from cuspidal local systems we imposed such compatibility in the
proof of Lemma \ref{lem:6.3}. 
\end{itemize}
Theorem \ref{thm:5.6} with $\mb q = 1$ can be considered as a generalized Springer correspondence
for the (extended) affine Weyl group $X \rtimes \Gamma = W_\af \rtimes \Omega$, with geometric
data $(s,u,\rho)$. Consequently Theorem \ref{thm:6.12} can also be regarded as a generalized
Springer correspondence of sorts, where the geometric data have been replaced by standard or
irreducible modules of an affine Hecke algebra with (nearly arbitrary) parameters 
$q_s \in \R_{\geq 1}$.

\subsection{Consequences for type $B_n / C_n$ Hecke algebras} \
\label{par:BnCn}

We illustrate the power of Theorem \ref{thm:6.2} with two results that rely on techniques
developed in Paragraph \ref{par:Wtypes}. Recall $\tilde{\Xi}^+$ from \eqref{eq:3.32}.

\begin{lem}\label{lem:6.13}
Let $\mh H$ be as in Theorem \ref{thm:6.2} and let $\tilde \xi \in \tilde \Xi^+$. Then
the 2-cocycle $\natural_{\tilde \xi}$ of $\mf R_{\tilde \xi}$ is trivial and
$\End_{\mh H}( \pi (\tilde \xi)) \cong \C [\mf R_{\tilde \xi}]$.
\end{lem}
\begin{proof}
Write $\tilde \xi = (P,\delta,\lambda)$, so $\delta \in \Irr (\mh H^P), \lambda \in \mf t^P$
and $\mf R_{\tilde \xi} \subset \mr{Stab}_W (P)$. Then $\mf R_{\tilde \xi}$ also
stabilizes $(P,\delta,z \lambda)$ for any $z \in \C$. The operators $\pi (w,P,\delta,z\lambda)$
with $w \in \mf R_{\tilde \xi}$ and $z \in \C$ come from \eqref{eq:3.33}, they are rational in $z$ 
and by Theorem \ref{thm:3.18} they are regular for $z \in \R_{\geq 0}$. Thus the family of
projective $\mf R_{\tilde \xi}$-representations $\pi (P,\delta,z \lambda)$ depends continuously
on $z \in \R_{\geq 0}$. Since a finite group has only finitely many isomorphism classes of 
projective representations of a given dimension, it follows that our family is constant (up to
isomorphism). Hence it suffices to prove the lemma in the case $\lambda = 0$.

We apply the proof of Lemma \ref{lem:6.6} to the algebra $\mh H^P \rtimes \mr{Stab}_W (P)$.
From \eqref{eq:6.27} we get a group homomorphism 
\[
\mf R_{\tilde \xi} \to \mr{Aut}_\C (V_\delta) : w \mapsto I_w
\]
such that $\ind_{\mh H^P}^{\mh H} (I_w)$ is a scalar multiple of $\pi (w,P,\delta,0)$. This 
shows that $\natural_{\tilde \xi}$ is trivial. Now the second claim follows from \eqref{eq:4.4}.
\end{proof}

Recall the notion of genericity for parameter functions $k : B_n \to \R$ from \eqref{eq:4.2}.

\begin{thm}\label{thm:6.14}
Let $k$ be a generic parameter function for $\mh H (\C^n,W(B_n),k)$.
\enuma{
\item All R-groups $\mf R_{\xi}$ with $\xi \in \tilde \Xi^+$ are trivial.
\item The map \eqref{eq:3.19} provides a bijection $\tilde \Xi^+ \to \Irr (\mh H (\C^n,W(B_n),k))$.
}
\end{thm}
\begin{proof}
Note that every parabolic subalgebra of $\mh H (\C^n,W(B_n),k)$ is the tensor product of an algebra 
of the same kind (but of smaller rank) and some graded Hecke algebras of type $A$ (for which we
saw in Example \ref{ex:Rgroups} that all R-groups are trivial).\\
(a) For discrete series representations, this follows from \cite[Proposition 3.9 and subsequent 
remark]{OpSo1}. That also implies part (a) for the discrete series of parabolic subalgebras of 
$\mh H (\C^n,W(B_n),k)$. By \cite[Theorem 9.1]{SolGHA} 
\[
\End_{\mh H} (\pi (P,\delta,\lambda)) \subset \End_{\mh H} (\pi (P,\delta,0)) .
\]
Together with \eqref{eq:4.4} this implies that $\mf R_{(P,\delta,\lambda)}$ is trivial whenever
$\mf R_{(P,\delta,0)}$ is trivial. Therefore it suffices to consider the case $\lambda = 0$. 

For every fixed $P$, Theorem \ref{thm:6.2} and \cite[Proposition 3.9]{OpSo1} tell us that 
the irreducible discrete representations of $\mh H^P = \mh H (\C^n,W_P,k)$ naturally give rise
to a basis of the space of elliptic representations $W_P$. By the latter we mean the representation 
ring of $W_P$ modulo the span of the representations induced from proper parabolic subgroups.
The dimension of this space equals the number of conjugacy classes of $W_P$ that are elliptic,
that is, do not intersect any proper parabolic subgroup of $W_P$ \cite[\S 2]{Ree}. The union
over $P$ of these sets of elliptic conjugacy classes, altogether up to $W(B_n)$-equivalence,
is precisely the set of conjugacy classes of $W(B_n)$. That is also the dimension of the
representation ring of $W(B_n)$. It follows that the set 
\[
\{ \ind_{W_P}^W (\zeta_{\mh H^P}(\delta)) : (P,\delta,0) \in \tilde \Xi \} / W(B_n) 
\]
is a basis of the representation ring of $W(B_n)$. In particular it has $|\Irr (W(B_n))|$ elements. 
By Theorem \ref{thm:6.2}.c the bijection $\zeta_{\mh H^P}$ is equivariant for $\mr{Stab}_W (P) / W_P$. 
Hence the number of $W(B_n)$-orbits of elements $(P,\delta,0) \in \tilde \Xi$ is also $|\Irr (W (B_n))|$. 

Suppose that $\mf R_{(P,\delta,0)}$ is nontrivial for some $(P,\delta,0)$. By Lemma \ref{lem:6.13}
there is more than one $\rho \in \Irr (\C [\mf R_{(P,\delta,0)}, \natural_{(P,\delta,0)}])$, so the
forgetful map $\tilde{\Xi}^+_e \to \tilde{\Xi}^+$ is not injective on such elements. 
With \eqref{eq:3.19} we see that $\Irr_0 (\mh H)$ has more elements than $\Irr (W)$. 
That contradicts Theorem \ref{thm:6.2}.\\
(b) This follows directly from (a) and \eqref{eq:3.19}.
\end{proof}

We want to use Corollary \ref{cor:3.27} to obtain a version of Theorem \ref{thm:6.14} 
for affine Hecke algebras. Let $T$ be the ``diagonal" maximal torus of $SO_{2n+1}$ and consider 
the root datum 
\[
\mc R (B_n) := \mc R (SO_{2n+1},T)
\] 
Let $\mc H (\mc R (B_n),\lambda,\lambda^*,\mb q)$ be an affine Hecke algebra as in Lemma 
\ref{lem:6.11}. As in Example \ref{ex:4.1} there are three independent parameters
\[
q_1 = \mb{q}^{\lambda (e_i \pm e_j)}, q_2 = \mb{q}^{\lambda (e_i)}, 
q_0 = \mb{q}^{\lambda^* (e_i)} 
\]
while Lemma \ref{lem:6.11} ensures that $q_1 ,q_2 ,q_0 \in \R_{\geq 1}$ and $q_2 \geq q_0$.
For $u \in T_\un$, the parameter function $k_u$ from \eqref{eq:3.24} satisfies 
\begin{equation}\label{eq:6.31}
\begin{aligned}
& k_u (e_i \pm e_j) = \lambda (e_i \pm e_j) \log (\mb q) =: k_1 , \\
& k_u (e_i) = (\lambda (e_i) \pm \lambda^* (e_i)) \log (\mb q) /2 =: k_2^\pm .
\end{aligned}
\end{equation}
Motivated by \eqref{eq:4.2} and \eqref{eq:6.31}, we call $q$ and $(\lambda, \lambda^*)$ generic if
\[
q_1 \neq 1 \text{ and } 
q_2 q_0^{\pm 1} \neq q_1^j \text{ for any } j \in \Z \text{ with } |j| \leq 2 (n-1) .
\]

\begin{prop}\label{prop:6.15}
Assume that $(\lambda,\lambda^*)$ is generic and as in Lemma \ref{lem:6.11}.
Then all R-groups for $\mc H (\mc R (B_n),\lambda,\lambda^*,\mb q)$ are trivial
and Corollary \ref{cor:3.23}.d provides a bijection
\[
\Xi^+ \to \Irr \big( \mc H (\mc R (B_n),\lambda,\lambda^*,\mb q) \big) .
\]
\end{prop}
\begin{proof}
Consider any $u \in T_\un \cong (S^1)^n$. Replacing it by a $W(B_n)$-conjugate, we may assume that
it is of the form
\[
u = (1)^{m_+} (-1)^{m_-} (\lambda_1)^{m_1} \cdots (\lambda_d )^{m_d} ,
\]
where $\lambda_i \in S^1 \setminus \{1,-1\}$ and $\lambda_i \neq \lambda_j^{\pm 1}$ for
any $i,j$. The notation $(\lambda)^m$ means that $m$ consecutive coordinates of $u$ are equal
to $\lambda$. Then 
\[
R_u \cong B_{m_+} \times B_{m_-} \times A_{m_1 - 1} \times \cdots \times A_{m_d - 1} 
\]
and $W(B_n)_u = W(R_u), \Gamma_u = 1$. Hence the algebra $\mc H (\mc R_t, \lambda, \lambda^*,\mb q)$
appearing in Theorem \ref{thm:3.25} is isomorphic to
\[
\mc H (\mc R (B_{m_+}),\lambda,\lambda^*,\mb q) \otimes 
\mc H (\mc R (B_{m_-}),\lambda,\lambda^*,\mb q) \otimes
\mc H_{m_1} (q_1) \otimes \cdots \otimes \mc H_{m_d}(q_1) .
\]
Next we apply Theorem \ref{thm:3.26} and we end up with the graded Hecke algebra
\begin{multline}\label{eq:6.30}
\mh H (\C^{m_+},W(B_{m_+}),k_1,k_2^+) \otimes \mh H (\C^{m_-},W(B_{m_-}),k_1,k_2^-) \, \otimes \\
\mh H (\C^{m_1},S_{m_1},k_1) \otimes \cdots \otimes \mh H (\C^{m_d},S_{m_d},k_1).
\end{multline}
The genericity of $(\lambda,\lambda^*)$ implies that the parameters for every tensor factor
of \eqref{eq:6.30} are generic and Lemma \ref{lem:6.11} ensures that they are positive.
Corollary \ref{cor:3.27} says that the relevant part of the module category of 
$\mc H (\mc R (B_n),\lambda,\lambda^*,\mb q)$ is equivalent with the relevant part
of the module category of \eqref{eq:6.30}. We know from Theorem \ref{thm:6.14} and Example
\ref{ex:Rgroups} that all R-groups of \eqref{eq:6.30} are trivial. With a view on the construction 
of R-groups in Paragraph \ref{par:Rgroups}, we conclude that the R-groups
for $\mc H (\mc R (B_n),\lambda,\lambda^*,\mb q)$ are also trivial.

We conclude with an application of Corollary \ref{cor:3.23}.
\end{proof}

We point out that Theorem \ref{thm:6.14}.b and Proposition \ref{prop:6.15} give an effective 
classification of the irreducible representations of the involved Hecke algebras, because all 
discrete series representations of parabolic subalgebras have been classified entirely in terms
of residual points, see Theorem \ref{thm:4.4} and the subsequent discussion.

It is interesting to compare this with \cite[Theorem D]{Kat3}. Kato provides a geometric 
construction and classification of the irreducible representations of \\
$\mc H (\mc R (B_n),\lambda,\lambda^*,\mb q))$ for all generic parameters $(\lambda,\lambda^*)$. 
We stress that this is a much wider range of $q$-parameters 
than in Proposition \ref{prop:6.15}, most of them are complex. The indexing set for 
$\Irr (\mc H (\mc R (B_n),\lambda,\lambda^*,\mb q))$ in \cite{Kat3} is a generalization of 
Kazhdan--Lusztig parameters, only with an exotic version of the nilpotent cone. Unfortunately
the results of \cite{Kat3} do not extend to non-generic parameters.

\end{document}